\documentclass{amsart}

\usepackage[letterpaper,margin=1in]{geometry}

\usepackage{amsmath,amssymb,amsthm,amscd}
\usepackage{mathtools}
\usepackage{enumerate}
\usepackage{tikz-cd}
\usepackage{pifont}

\theoremstyle{plain}
\newtheorem{alphathm}{Theorem}

\newtheorem*{corollary*}{Corollary}
\newtheorem*{proposition*}{Proposition}

\theoremstyle{definition}
\newtheorem*{defn*}{Definition}

\theoremstyle{plain}
\newtheorem{thm}{Theorem}[section]
\newtheorem{prop}[thm]{Proposition}
\newtheorem{lem}[thm]{Lemma}
\newtheorem{cor}[thm]{Corollary}
\newtheorem{obs}[thm]{Observation}
\newtheorem*{conj*}{Conjecture}

\theoremstyle{definition}
\newtheorem{defn}[thm]{Definition}

\newcounter{pcounter}
\theoremstyle{remark}
\newtheorem{remark}[thm]{Remark}
\newtheorem{example}{Example}

\numberwithin{equation}{section}

\newcommand\eps{\varepsilon}

\newcommand\C{\mathbb{C}}
\newcommand\R{\mathbb{R}}
\newcommand\N{\mathbb{N}}
\newcommand\Z{\mathbb{Z}}
\newcommand\E{\mathbb{E}}
\newcommand\F{\mathbb{F}}

\newcommand\cA{\mathcal{A}}
\newcommand\cM{\mathcal{M}}
\newcommand\cN{\mathcal{N}}
\newcommand\cP{\mathcal{P}}
\newcommand\cQ{\mathcal{Q}}
\newcommand\cD{\mathcal{D}}
\newcommand\cU{\mathcal{U}}
\newcommand\cR{\mathcal{R}}

\newcommand{\actson}{\curvearrowright}

\DeclareMathOperator{\id}{id}

\DeclareMathOperator{\Tr}{Tr}

\DeclareMathOperator{\Lip}{Lip}
\DeclareMathOperator{\Span}{Span}

\DeclareMathOperator{\ev}{ev}
\DeclareMathOperator{\diag}{diag}
\DeclareMathOperator{\re}{Re}

\DeclarePairedDelimiter{\norm}{\lVert}{\rVert}
\DeclarePairedDelimiter{\ip}{\langle}{\rangle}
\DeclarePairedDelimiter{\floor}{\lfloor}{\rfloor}

\def\smallint{\begingroup\textstyle \int\endgroup}

\begin{document}
\title{A Random Matrix Approach to Absorption in Free Products}

\author{Ben Hayes}
\author{David Jekel}
\author{Brent Nelson}
\author{Thomas Sinclair}

\thanks{B. Hayes was partially supported by NSF grants DMS-1600802 and DMS-1827376. D.~Jekel was partially supported by NSF grant DMS-1762360. B.~Nelson was partially supported by NSF grant DMS-1856683. T.~Sinclair was partially supported by NSF grant DMS-1600857.}

\address[B.~Hayes]{Department of Mathematics,
University of Virginia, 141 Cabell Drive, Kerchof Hall, Charlottesville, VA, 22904}
\email{brh5c@virginia.edu}

\address[D.~Jekel]{Department of Mathematics, University of California, Los Angeles, Box 951555,
Los Angeles, CA 90095-1555}
\email{davidjekel@math.ucla.edu}

\address[B.~Nelson]{Department of Mathematics, Michigan State University, 619 Red Cedar Road, C212 Wells Hall, East Lansing, MI 48824-3402}
\email{brent@math.msu.edu}

\address[T.~Sinclair]{Department of Mathematics, Purdue University, 150 N. University St., West Lafayette, IN 47907-2067}
\email{tsincla@purdue.edu}
%\date{\today}

% This paper gives a free probabilistic perspective on  maximal amenability/amenable absorption results for free group factors, as well as other amalgamated free products. In particular, we give the first free entropy dimension proof of Popa's famous result on maximal amenability of the generator MASA in free group factors. We also partially recover the amenable absorption results and Gamma stability results for free products due to Houdayer. Moreover, we give a unified approach to these results using $1$-bounded entropy, by studying diffuse subalgebras that are maximal with respect to having $1$-bounded entropy zero, which we call \emph{Pinsker algebras}.  We give sufficient conditions for an subalgebra $\cP \leq \cM$ of $1$-bounded entropy zero to be Pinsker, which require random matrix models for $\cM$ that are asymptotically concentrated on microstate spaces ``relative" to this subalgebra, have exponential concentration of measure, and ``simulate" the conditional expectation onto this subalgebra.

\begin{abstract}
This paper gives a free entropy theoretic perspective on amenable absorption results for free products of tracial von Neumann algebras. In particular, we give the first free entropy proof of Popa's famous result that the generator MASA in a free group factor is maximal amenable, and we partially recover Houdayer's results on amenable absorption and Gamma stability.  Moreover, we give a unified approach to all these results using $1$-bounded entropy.  We show that if $\cM = \cP * \cQ$, then $\cP$ absorbs any subalgebra of $\cM$ that intersects it diffusely and that has $1$-bounded entropy zero (which includes amenable and property Gamma algebras as well as many others).  In fact, for a subalgebra $\cP \leq \cM$ to have this absorption property, it suffices for $\cM$ to admit random matrix models that have exponential concentration of measure and that ``simulate'' the conditional expectation onto $\cP$.
\end{abstract}

\maketitle

\section*{Introduction}

A foundational discovery of Popa \cite{Popa1983} showed that the generator MASA $\mathcal{A} = L(\Z)$ in a free group factor $\cM = L(\Z*\F_{d-1})$ is a maximal amenable subalgebra (or equivalently a maximal hyperfinite subalgebra by \cite{Connes}). This provided the first example of an abelian subalgebra of a $\mathrm{II}_1$ factor that is maximal amenable (it is even maximal property Gamma), and it answered in the negative the question, stated by Kadison at the 1967 Baton Rouge conference, of whether any self-adjoint operator in a II$_1$ factor is contained in a hyperfinite subfactor.
%This answered in the negative a conjecture of Kadison stated during the Baton Rouge conference in 1967.
The fundamental insight of Popa was a detailed analysis of the $\mathcal{A}$-central sequences in $\cM$ through his \emph{asymptotic orthogonality property}. The approach of Popa is very fruitful and many authors have built on his work to establish maximal amenability in several cases, see \cite{Ge96, Shen06, FangMaximalAmena, CFRW, Gao10, Brothier, Houdayerexotic, CyrilSAOP,Bleary}. Several other celebrated structural results for free group factors can be shown using Popa's breakthrough deformation/rigidity theory initiated in \cite{Po01a, PopaL2Betti, PopaStrongRigidity, PopaStrongRigidtyII}. Moreover, deformation/rigidity theory allows one to prove  structural results not just for free group factors, but also von Neumann algebras of more general groups.  We refer the reader to \cite{PopaICM,Va06a,Va10a,Io12b,Io17c} for a survey of these results, many of which include solutions to long-standing open problems. This theory has been particularly fruitful in proving absence of Cartan subalgebras, as well as uniqueness of Cartan subalgebras for crossed product algebras. See, e.g., \cite{OzPopaCartan, OzPopaII, ChifanPeterson, ChifanSinclair,  PopaVaesFree, PopaVaesHyp, CartanAFP}.

In this paper, we present an approach to structural results for free products through Voiculescu's free entropy dimension theory.
%which, roughly speaking, quantifies how many matrix approximations there are for the generators of a von Neumann algebra and relates this to various other von Neumann algebraic properties.
Free entropy dimension theory was initiated by Voiculescu in a series of papers \cite{VoiculescuFreeEntropy2, Voiculescu1996}, and gives powerful tools to prove various indecomposability and structural results on free group factors. In one celebrated achievement, Voiculescu gave the first proof of absence of Cartan subalgebras for free group factors using free entropy methods \cite[Theorem 5.2]{Voiculescu1996}. Shortly thereafter Ge \cite{GePrime} and Ge and Popa \cite{PopaGeThin} used this machinery to prove that free group factors are prime and \emph{thin} (cannot be decomposed as a tensor product of hyperfinite factors), respectively. See \cite{DykemaFreeEntropy, Jung2007,HoudShlStrongSolid} for other applications of free entropy dimension to the structure of free group factors.
%We remark here that Popa's deformation/rigidity theory  applies to a wider class of algebras than free products, and does not require that the algebras have the Connes approximate embedding property as one needs by the free entropy dimension approach.  On the other hand, if we do assume Connes embeddability, the tools of free entropy dimension are well-suited for studying free group factors and free products because these von Neumann algebras have natural random matrix models, thanks to Voiculescu's asymptotic freeness theorem \cite{VoicAsyFree, Voiculescu1998}.
While Popa's deformation/rigidity theory applies to a wider class of algebras than free products, and does not require that the algebras have the Connes approximate embedding property, the tools of free entropy dimension are well-suited for studying free group factors and free products of Connes embeddable algebras due to the presence of natural random matrix models that come from Voiculescu's asymptotic freeness theorem \cite{VoicAsyFree, Voiculescu1998}. Additionally, there are certain indecomposability results for free group factors that can be proved using free entropy techniques which are currently out of reach by other methods. (See, e.g., \cite[Theorem 1.3, Corollary 1.4, Corollary 1.7]{Hayes2018}, as well as \cite{DykemaFreeEntropy, PopaGeThin}.)

Given the success of free entropy theory and random matrix theory in proving other structural results for free products, it is reasonable to suppose that the maximal amenability of the generator MASA in a free group factor can be proved using random matrices and free entropy dimension. We provide such a proof using the notion of $1$-bounded entropy, which is implicit in \cite{Jung2007} and explicitly defined in \cite{Hayes2018}. The $1$-bounded entropy is defined similarly to free entropy dimension, but it is useful specifically for studying von Neumann algebras with free entropy dimension $1$; the $1$-bounded entropy is less than $+\infty$ if and only if the algebra is strongly $1$-bounded in the sense of Jung \cite{Jung2007}.  The $1$-bounded entropy has the advantage that it is known to be an invariant for tracial von Neumann algebras and can be computed on any set of generators. This invariant provides a unified  and efficient way to recast the proofs of Voiculescu and Ge on structural properties of free group factors, see Section \ref{subsec: 1 bounded entropy properties}.

%The $1$-bounded entropy also provides a conceptual picture for Jung's result that being strongly $1$-bounded is a von Neumann algebra invariant, since being strongly $1$-bounded is just equivalent to having $1$-bounded entropy less than $\infty$.

We now state the main result of the paper. As a convention, when the context is clear we will write ``$\cN\leq \cM$'' to indicate that $\cN$ is a von Neumann subalgebra of $\cM$.

\begin{alphathm} \label{thm:main2}
Let $(\cM_1,\tau_{\cM_1})$ and $(\cM_2,\tau_{\cM_2})$ be tracial $\mathrm{W}^*$-algebras such that every separable subalgebra is embeddable into $\mathcal{R}^\omega$.  Let $(\mathcal{D},\tau_\mathcal{D})$ be a common atomic subalgebra of $\cM_1$ and $\cM_2$ with $\tau_{\cD} = \tau_{\cM_1}|_{\cD} = \tau_{\cM_2}|_{\cD}$, and consider the amalgamated free product $\cM = \cM_1 *_\mathcal{D} \cM_2$. Then for any $\cN\leq \cM$ with $h(\cN: \cM) = 0$ and $\cN\cap \cM_1$ diffuse, one has $\cN \subseteq \cM_1$.
\end{alphathm}

When $\cM$ is Connes-embeddable, the condition $h(\cN:\cM)=0$ is satisfied whenever $\cN$ is diffuse and amenable, non-prime, has a Cartan subalgebra, or all nonzero direct summands have property Gamma (see \S \ref{subsec: 1 bounded entropy properties}). Thus, for instance, we obtain maximal amenability of the generator MASA in a free group factor since the above result implies that it is a maximal subalgebra with $1$-bounded entropy zero.  At the same time, our argument shows that the generator MASA is a maximal subalgebra with property Gamma.
% The maximal amenability of the generator MASA in a free group factor is obtained since the above result implies that it is a maximal subalgebra with $1$-bounded entropy zero, and amenable von Neumann algebras always have $1$-bounded entropy zero.  This proof will have other consequences besides maximal amenability, since there are several other classes of von Neumann algebras that have $1$-bounded entropy $\leq 0$ (zero if Connes-embeddable and $-\infty$ otherwise).  Thus, for instance, our proof implies that the generator MASA is a maximal subalgebra with property Gamma.
In this way, our approach gives a unified proof for amenability, property Gamma, having Cartan subalgebras, and so forth.  However, a significant limitation is that we can only handle free absorption results when each side is finite, and so for instance, our methods cannot reach type III algebras. Moreover, unlike the asymptotic orthogonality or deformation/rigidity approaches, we have to assume $\cM_{1},\cM_{2}$ both embed into an ultrapower of $\cR$, and our results do not yet have concrete applications beyond free products.

We remark that if $\mathcal{D}$ is diffuse, then the conclusions of Theorem \ref{thm:main2} can fail for rather trivial reasons (in contrast to the results of Brown--Dykema--Jung \cite{BDJ2008} on free entropy dimension for free products with amalgamation over any amenable $\mathcal{D}$).  For instance, suppose that $\mathcal{D}$ is diffuse abelian, let $\mathcal{K}_1$ and $\mathcal{K}_2$ be finite-dimensional algebras, and let $\cM_1 = \mathcal{D} \otimes \mathcal{K}_1$ and $\cM_2 = \mathcal{D} \otimes \mathcal{K}_2$.  Then, in the algebra $\cM$, the intersection $\cM_1 \cap \cM_2$ is diffuse and $h(\cM_2:\cM) = 0$, yet $\cM_2$ is not contained in $\cM_1$.

The conclusion of Theorem \ref{thm:main2} relates to previous results about amenable absorption. For example, in \cite{CyrilAOP} Houdayer defined a generalization of Popa's asymptotic orthogonal property (see \cite[Theorem 3.1]{CyrilAOP}) now called the \emph{strong asymptotic orthogonality property} from which one can deduce an amenable absorption property: if $\mathcal{Q}\leq \cM$ is amenable and $\mathcal{Q}\cap \cA$ is diffuse, then $\mathcal{Q} \subseteq \cA$. This strengthening of the asymptotic orthogonality property was then used in \cite{WenAOP, 2AuthorsOneCup, ParShiWen} to give other examples of situations where one has amenable absorption for maximal amenable subalgebras of free group factors. We remark in passing that it is a conjecture of Peterson and Thom (see the discussion following \cite[Proposition 7.7]{PetersonThom}) that every maximal amenable subalgebra of a free group factor has the amenable absorption property.
There is also another recent framework for analyzing maximal amenability in terms of singular states due to Boutonnet--Carderi \cite{BC2015}, which was modified by Ozawa to give a short proof of maximal amenability of the generator MASA in \cite{Ozawa2015}. This was carried out to great effect by Boutonnet--Houdayer in \cite{BH2018} to give complete results on amenable absorption in amalgamated free products.

Since amenable algebras always have $1$-bounded entropy zero, Theorem \ref{thm:main2} recovers these amenable absorption results for certain amalgamated free products. Moreover, since factors with property Gamma also have $1$-bounded entropy zero, we recover some of the Gamma stability results shown in \cite{CyrilAOP}, and more. For example: if $\cQ\leq \cM_{1}*\cM_{2}$ and $\cQ\cap \cM_{1}$ is diffuse, and if $\cQ$ has $1$-bounded entropy zero, then $\cQ\subseteq \cM_{1}$. So if $\cM_{1}$ is non-prime then $\cM_{1}$ absorbs any  $\cQ$ that is non-prime, has the property that every nonzero direct summand has property Gamma, or has a Cartan, provided that $\cQ$ intersects $\cM_{1}$ diffusely. In the spirit of Gamma stability, we can replace the assumption that $\cQ$ is non-prime, or that every nonzero direct summand has property Gamma, or that it has a Cartan, with the requirement that $\cQ'\cap \cM^{\omega}$ is diffuse and still conclude that $\cQ \subseteq \cM_{1}$ (provided $\cQ\cap \cM_{1}$ is diffuse).
Additionally, by \cite{Houdayerexotic,HoudShlStrongSolid,Hayes2018} one can find von Neumann algebras which are strongly solid, have the Haagerup property, and the complete metric approximation property and yet still have $1$-bounded entropy zero, and so similar remarks apply to these examples.

From Theorem \ref{thm:main2}, we deduce new indecomposability results for certain amalgamated free products. This relates to various weakenings of the normalizer, such as the quasi-normalizer defined in \cite{IzumiLongoPopa,PimnserPopa,PopaQR} (building off of ideas in \cite{PopaOrthoPairs}), or the wq-normalizer as defined in \cite{PopaStrongRigidity, PopaCohomologyOE, IPP, GalatanPopa}. We use $qN_{\cM}(\cN)$ for the quasi-normalizer of $\cN$ inside $\cM$. See also \cite[Definition 1.1]{Hayes2018} for the definition of the ``singular subspace" which is larger than the normalizer, quasi-normalizer, and the wq-normalizer (see \cite[Proposition 3.2]{Hayes2018}). The next corollary is an immediate consequence of Theorem \ref{thm:main2} and \cite[Theorem 3.8]{Hayes2018}.

\begin{corollary*}\label{cor:quasinormal absorption}
Let $(\cM_1,\tau_{\cM_1})$ and $(\cM_2,\tau_{\cM_2})$ be tracial $\mathrm{W}^*$-algebras such that every separable subalgebra is embeddable into $\mathcal{R}^\omega$.  Let $(\mathcal{D},\tau_\mathcal{D})$ be a common atomic subalgebra of $\cM_1$ and $\cM_2$ with $\tau_{\cM_1}|_{\cD} = \tau_{\cM_2}|_{\cD} = \tau_{\cD}$.  Set $\cM=\cM_{1}*_{\mathcal D}\cM_{2}$. If $\cN\leq \cM$ is diffuse and $h(\cN:\cM)\leq 0$, and if $W^{*}(qN_{\cM}(\cN))\cap \cM_{1}$ is diffuse, then $\cN\subseteq \cM_{1}$. The same results holds if $qN_{\cM}(\cN)$ is replaced with the wq-normalizer, or the singular subspace of $\cN\subseteq\cM$.
\end{corollary*}

In fact, by \cite[Theorem 3.8, Corollary 4.1]{Hayes2018}, we can start with a diffuse, amenable $\cN$ and iterate the process of taking $W^{*}(qN_{\cM}(\cN))$ (or even the singular subspace of $\cN$ inside $\cM$) across all ordinals and if for  some ordinal the algebra we obtain has diffuse intersection with $\cM_{1}$, then it follows that $\cN\leq \cM_{1}$.

The proof of Theorem \ref{thm:main2} proceeds (modulo a quick reduction to the separable case) by constructing random matrix models that have certain properties relative to the inclusion $\cM_1 \leq \cM_1 * \cM_2$.  In fact, all our conclusions hold for an inclusion $\cP \leq \cM$ if we only assume the existence of random matrix models satisfying these properties. We state this result as a theorem in its own right.  Here $\tau_n$ denotes the normalized trace $\frac1n \Tr$ on the matrix algebra $M_n(\C)$.  Also, $\norm{\cdot}_2$ denotes the $2$-norm $\norm{x}_2 = \tau(x^*x)^{1/2}$ associated to any tracial $\mathrm{W}^*$-algebra $(\cM,\tau)$, including $(M_n(\C),\tau_n)$.

\begin{alphathm} \label{thm:main}
Suppose that $\cP \subseteq \cM$ is an inclusion of tracial $\mathrm{W}^*$-algebras.  Let $\mathbf{x} = (x_i)_{i \in I}$ be a tuple of bounded self-adjoint generators for $\cM$.  Suppose $n(k) \to \infty$ and that $\mathbf{X}^{(k)} = (X_i^{(k)})_{i \in I}$ is a tuple of random $n(k) \times n(k)$ self-adjoint matrices satisfying the following properties:
\begin{enumerate}
	\item We have $\norm{X_i^{(k)}} \leq R_i$ for some $R_i$ independent of $k$. \label{I:operator norm bound thm main}
	\item For every non-commutative polynomial $p \in \C\ip{t_i: i \in I}$, we have $\tau_{n(k)}(p(\mathbf{X}^{(k)})) \to \tau(p(\mathbf{x}))$ in probability. \label{I:convergence in law thm main}
	\item The probability distributions of $(X_i^{(k)})_{i \in F}$ exhibit exponential concentration of measure at the scale $n(k)^2$ as $k \to \infty$ for each finite $F \subseteq I$.  (See \S \ref{subsec:concentration} for the precise definition.) \label{I:ECM thm main}
	\item For each non-commutative polynomial $p \in \C\ip{t_i: i \in I}$, we have
	\[
	\lim_{k \to \infty} \norm{ \E[p(\mathbf{X}^{(k)})] }_2 = \norm{E_{\cP}[p(\mathbf{x})]}_2.
	\]
	Here $E_{\cP}$ is the trace-preserving conditional expectation of $\cM$ onto $\cP$. \label{I:EAP thm main}
\end{enumerate}
Then for any $\mathrm{W}^*$-algebra $\cN \leq \cM$ with $h(\cN: \cM)=0$ and $\cN\cap \cP$ diffuse, one has $\cN \subseteq \cP$.
\end{alphathm}

The hypotheses (\ref{I:operator norm bound thm main}) and (\ref{I:convergence in law thm main}) are standard descriptions of how random matrix models simulate the given operators $\mathbf{x}$ in the large $k$ limit.  Exponential concentration (\ref{I:ECM thm main}) is also familiar in random matrix theory (at least in the case when $I$ is finite); see for instance \cite[\S 2.3 and \S 4.4]{AGZ2009}.  The hypothesis (\ref{I:EAP thm main}) is new to our work, and it says intuitively that the probabilistic expectation for the random matrix models simulates the (trace-preserving) conditional expectation of $\cM$ onto $\cP$ rather than the trace.

In the case of Theorem \ref{thm:main2} where $\cM = \cM_1 *_{\cD} \cM_2$ and $\cP = \cM_1$, such random matrix models will be constructed by adapting the methods of Brown--Dykema--Jung \cite{BDJ2008}, who gave a version of Voiculescu's asymptotic freeness results for amalgamation over a finite-dimensional algebra.

The application of Theorem \ref{thm:main} in the case where $h(\cP:\cM) = 0$ naturally motivates the following definition.

\begin{defn*}
Let $(\cM,\tau)$ be a tracial von Neumann algebra. We say that a diffuse $\mathcal{P}\leq \cM$ is a \emph{Pinsker algebra} if $h(\cP:\cM) = 0$, and given any $\mathcal{P}\leq \mathcal{Q}\leq \mathcal{M}$ with $h(\mathcal{Q}:\cM) = 0$ we have $\mathcal{Q}=\mathcal{P}$.
\end{defn*}

It follows from general properties of $1$-bounded entropy that if $\cQ\leq \cM$ is diffuse and $h(\cQ:\cM)\leq 0$, then there is a unique Pinsker algebra $\cP\leq \cM$ with $\cP\supseteq \cQ$. In particular, two Pinsker algebras which intersect diffusely are equal. We refer the reader to Section \ref{subsec: 1 bounded entropy properties} for a more detailed discussion. We also discuss in Section \ref{subsec: 1 bounded entropy properties} the motivation for the terminology ``Pinsker algebra," which comes from entropy in ergodic theory. Hypotheses (\ref{I:convergence in law thm main})-(\ref{I:EAP thm main}) of Theorem \ref{thm:main} have natural interpretations in ergodic theory. For example, exponential concentration in the ergodic theory context naturally leads to the ``almost blowing up property," a condition on a probability measure-preserving action which is equivalent to the action being Bernoulli. We refer the reader to \cite{MartonShieldsABUP} as well as \cite[Section III.4]{ShieldsBook} for more information.
% In fact, part of the proof of Theorem \ref{thm:main} is inspired by notes of L.~Bowen \cite{LewisECM} which give an approach in terms of exponential concentration of measure to prove that actions of sofic groups are Bernoulli. Austin also used such methods in \cite{AustinWeakPinsker} to solve the weak Pinsker conjecture for actions of $\Z$, which was a long-standing open problem.
We remark that \cite[Theorem 3.8]{Hayes2018} (which is the main result of that paper) can be rephrased as saying that if $\cP\leq \cM$ is Pinsker, then $\cP\leq \cM$ is coarse in the sense of Popa \cite{PopaWeakInter}. So Pinsker algebras connect to coarse embeddings, and thus also to the coarseness conjecture independently formulated by the first named author \cite[Conjecture 1.12]{Hayes2018} and Popa  \cite[Conjecture 5.2]{PopaWeakInter}.
As a corollary of Theorem \ref{thm:main}, we obtain examples of Pinsker algebras in amalgamated free products.

\begin{corollary*}
Let $\cP \leq \cM$ be as in Theorem \ref{thm:main} and suppose that $h(\cP: \cM) = 0$.  Then $\cP$ is a Pinsker algebra.  In particular, if $\cM = \cM_1 *_{\cD} \cM_2$ is as in Theorem \ref{thm:main2} and $h(\cM_1: \cM) = 0$, then $\cM_1$ is a Pinsker algebra in $\cM$.
\end{corollary*}

We remark that the proofs of Theorems \ref{thm:main2} and \ref{thm:main} are designed to work even with infinite generating sets (that is, the index set $I$ in Theorem \ref{thm:main} may be infinite).  Of course, the statement and proof would be marginally simpler to write down in the finitely generated case, and random matrix results often use finite generating sets.  However, some of the most natural applications of $1$-bounded entropy and Pinsker algebras deal with cases where we do not know a priori that the $\mathrm{W}^*$-algebras are finitely generated (and it is unknown whether every tracial $\mathrm{W}^*$-algebra with separable predual is finitely generated).  That is why we have taken care to set up our machinery to handle infinite generating sets.  We also show that the existence of random matrix models satisfying the hypotheses of Theorem \ref{thm:main} is independent of the choice of generators for $\cM$; see Proposition \ref{prop:independentofgenerators} and the discussion preceding it.

The rest of the paper is organized as follows.  In Section \ref{sec:background}, we review background, definitions, and notation.  In particular, we summarize the properties of $1$-bounded entropy that are needed for the applications of our main theorems discussed in the introduction.  In Section \ref{sec:functionalcalculus}, we present a multivariable functional calculus, based on the work of the second author in \cite{Jekel2018, Jekel2019}, that will be used as a technical tool to transform between microstate spaces for different choices of generators.  In Section \ref{sec:Theorem1}, we prove Theorem \ref{thm:main}.  In Section \ref{sec:examples}, we apply Theorem \ref{thm:main} to the generator MASA in free group factors and we describe the relationship with previous work on random matrix models with convex potentials.  In Section \ref{sec:Theorem2}, we prove Theorem \ref{thm:main2} which handles the case of free products with amalgamation over atomic subalgebras.  In Section \ref{sec:remarks}, we further discuss the random matrix models used in Theorem \ref{thm:main}, relating them with other ideas from von Neumann algebras and ergodic theory.

\subsection*{Acknowledgements.} The initial stages of this work were carried out at the Hausdorff Research Institute for Mathematics during the 2016 trimester program ``Von Neumann Algebras." The first, third, and last named authors thank the Hausdorff institute for their hospitality. The first named author thanks Tim Austin and Lewis Bowen for discussions related to exponential concentration of measure in the sofic entropy context which were inspirational for this work. In particular, he thanks Bowen for sharing notes on exponential concentration of measure in the sofic entropy context which led to the initial stages of this paper.  All the authors thank the Banff International Research Station for their hospitality during the workshop ``Classification Problems in von Neumann Algebras'' in Fall 2019. The authors thank Ionut Chifan for useful comments and for suggesting the corollary following Theorem \ref{thm:main2}.

\section{Background} \label{sec:background}

\subsection{Algebras, Traces, and Laws}

We begin by recalling some basic definitions, concepts, and notations from operator algebras.

A \emph{tracial $\mathrm{W}^*$-algebra} is pair $(\cM, \tau)$, where $\cM$ is a $\mathrm{W}^*$-algebra and $\tau$ is a faithful, normal, tracial state.  This is equivalent (in the separable case) to a finite von Neumann algebra with a designated choice of faithful normal trace.  We denote by $\cM_{sa}$ the set of self-adjoint elements of $\cM$.  We denote $\norm{x}_2 = \tau(x^*x)^{1/2}$.

We will view $\cM$ as an algebra of ``bounded random variables'' and $\tau$ as the ``expectation.''  Given $(\cM, \tau)$ and a $\mathrm{W}^*$-subalgebra $\cN$, there is a unique trace-preserving conditional expectation $E_{\cN}: \cM \to \cN$.  In this paper, the default meaning of ``conditional expectation $\cM \to \cN$'' will be this trace-preserving conditional expectation.

Given a (possibly infinite) index set $I$, we want to discuss the non-commutative law of a tuple $(x_i)_{i \in I} \in \cM_{sa}^I$, as well as a topology on the space of non-commutative laws.  These laws are defined in terms of the trace of non-commutative polynomials in the variables $(x_i)_{i \in I}$.

Let $\C\ip{t_i:i\in I}$ be the algebra of non-commutative complex polynomials in $(t_i)_{i \in I}$ (i.e. the free $\C$-algebra on the set $I$). We give $\C\ip{t_i:i\in I}$ the unique $*$-algebra structure which makes the $t_i$ self-adjoint.  If $\cM$ is a $\mathrm{W}^*$-algebra and $\mathbf{x} = (x_i)_{i \in I}\in \cM_{sa}^I$% is a family of self-adjoint elements
, then there is a unique $*$-homomorphism $\ev_\mathbf{x}: \C\ip{t_i: i \in I} \to \cM$ such that $\ev_{\mathbf{x}}(t_i) = x_i$.  For a non-commutative polynomial $p \in \C\ip{t_i: i \in I}$, we define $p(\mathbf{x}) = p((x_i)_{i \in I})$ to be $\ev_\mathbf{x}(p)$.

A tracial \emph{non-commutative law} of a self-adjoint $I$-tuple is a linear functional $\lambda: \C\ip{t_i: i \in I} \to \C$ that is
\begin{enumerate}
	\item unital, that is, $\lambda(1) = 1$;
	\item positive, that is, $\lambda(p^*p) \geq 0$;
	\item tracial, that is, $\lambda(pq) = \lambda(qp)$;
	\item exponentially bounded, that is, for some $(R_i)_{i \in I} \in (0,+\infty)^I$, we have
	\[
	|\lambda(t_{i(1)} \dots t_{i(\ell)})| \leq R_{i(1)} \dots R_{i(\ell)}
	\]
	for all $\ell$ and all $i(1)$, \dots, $i(\ell) \in I$.
\end{enumerate}
Given $\mathbf{R} = (R_i)_{i \in I} \in (0,+\infty)^I$, we define $\Sigma_{\mathbf{R}} = \Sigma_{(R_i)_{i \in I}}$ to be the set of non-commutative laws satisfying (4) for our given choice of $(R_i)_{i \in I}$.  We equip $\Sigma_{\mathbf{R}}$ with the topology of pointwise convergence on $\C\ip{t_i: i \in I}$.
In this topology, $\Sigma_{\mathbf{R}}$ is a compact Hausdorff space, and if $I$ is countable, then $\Sigma_{\mathbf{R}}$ is metrizable.

Given a tracial $\mathrm{W}^*$-algebra $(\cM, \tau)$, a tuple $\mathbf{x} = (x_i)_{i \in I} \in \cM_{sa}^I$ and $\mathbf{R} = (R_i)_{i \in I} \in (0,+\infty)^I$ satisfying $\norm{x_i} \leq R_i$, we define the \emph{non-commutative law of $\mathbf{x}$} as the map
\[
\lambda_{\mathbf{x}}: \C\ip{t_i: i \in I} \to \C: \quad p \mapsto \tau(p(x)).
\]
It is straightforward to verify that $\lambda_{\mathbf{x}}$ is in $\Sigma_{\mathbf{R}}$.  Conversely, given any $\lambda \in \Sigma_{\mathbf{R}}$, there exists some $(\cM, \tau)$ and $\mathbf{x} \in \cM_{sa}^I$ such that $\lambda_{\mathbf{x}} = \lambda$ and $\norm{x_i} \leq R_i$ for all $i \in I$.  This follows from a version of the GNS construction; the proof is the same as in \cite[Proposition 5.2.14(d)]{AGZ2009}.
We also remark that $(\cM, \tau)$ could be $M_n(\C)$ with the normalized trace $\tau_n = (1/n) \Tr$.  Thus, if $\mathbf{x} \in M_n(\C)_{sa}^I$, then $\lambda_{\mathbf{x}}$ is a well-defined non-commutative law.

At several points in the paper we will use the following folklore result, whose proof we leave as a exercise to the reader. See, e.g., \cite[Lemma 2.9]{JungTubularity} or \cite[Lemma 1.10 and Proposition 1.7]{Atkinson2019}.

\begin{lem} \label{lem:hyper folklore}
Let $(\cM,\tau)$ be a hyperfinite tracial von Neumann algebra with separable predual, and let $I$ be a countable index set. Let $\mathbf{x} \in \cM_{sa}^{I}$. Suppose that $(n(k))_{k}$ is a sequence of positive integers with $n(k) \to \infty$. If $\mathbf{A}^{(k)}, \mathbf{B}^{(k)}\in M_{n(k)}(\C)_{sa}^{I}$ satisfy
    $\sup_{k}\|A^{(k)}_{i}\|,\sup_{k}\|B^{(k)}_{i}    \|<\infty$  for all $i\in I$,
and $\lambda_{\mathbf{A}^{(k)}}\to\lambda_{\mathbf{x}}$,  $\lambda_{\mathbf{B}^{(k)}}\to \lambda_{\mathbf{x}}$, then there is a sequence $U^{(k)}\in M_{n(k)}(\C)$ with
\[\|U^{(k)}A^{(k)}_{i}(U^{(k)})^{*}-B^{(k)}_{i}\|_{2}\to_{k\to\infty} 0,%\mbox{ for all $i\in I$.}
\]
for all $i\in I$.
\end{lem}

\subsection{Microstate Spaces and $1$-Bounded Entropy}\label{subsec: 1 bounded entropy properties}

Microstates free entropy as well as $1$-bounded  entropy are defined by looking at the exponential growth rate of volumes/covering numbers of so-called matricial microstate spaces, that is, spaces of matrices that have approximately the same non-commutative law as a given tuple $\mathbf{y}$.

Let $\mathbf{R} \in (0,+\infty)^J$, let $\mathbf{y} \in \cM_{sa}^J$ be a self-adjoint tuple with $\norm{y_j} \leq R_j$, and let $\mathcal{U} \subseteq \Sigma_{\mathbf{R}}$ be a neighborhood of $\lambda_{\mathbf{y}}$.  Then we define the microstate space
\[
\Gamma_{\mathbf{R},n}(\mathbf{y}; \mathcal{U}) := \{\mathbf{B} \in M_n(\C)_{sa}^J: \norm{B_j} \leq R_j \text{ for all } j \text{ and } \lambda_{\mathbf{B}} \in \mathcal{U}\}.
\]

These microstate spaces are invariant under conjugating a tuple $\mathbf{B}$ by a fixed $n \times n$ unitary matrix.  In order to remove the inherent ambiguity of unitary equivalence, we will often fix a self-adjoint $z \in \cM$ and fix a sequence $(C^{(n)})_{n \in \N}$ such that $C^{(n)} \in M_n(\C)$, $\norm{C^{(n)}}$ is bounded, and $\lambda_{C^{(n)}} \to \lambda_z$. For shorthand, we will often denote this as $C^{(n)}\rightsquigarrow z$.  Thus, we define \emph{microstate spaces for $\mathbf{y}$ relative to $C^{(n)} \leadsto z$} as follows: For each $\mathbf{R} \in (0,+\infty)^{J \sqcup \{0\}}$ with $\norm{y_j} < R_j$ and $\norm{z} < R_0$, for each sequence of microstates $C^{(n)}$ for $z$ as above with $\norm{C^{(n)}} < R_0$, and for each neighborhood $\mathcal{U}$ of $\lambda_{(\mathbf{y}, z)} \in \Sigma_{\mathbf{R}}$, define
\begin{align*}
\Gamma_{\mathbf{R},n}&(\mathbf{y} | C^{(n)} \rightsquigarrow z; \mathcal{U}):= \{\mathbf{B} \in M_n(\C)_{sa}^J: \norm{B_j} \leq R_j, \lambda_{(\mathbf{B},C^{(n)})} \in \mathcal{U}\}.
\end{align*}

Finally, we define microstate spaces for $\mathbf{y}$ relative to $C^{(n)} \leadsto z$ which take into account the presence of another tuple $\mathbf{x} \in (\cM)_{sa}^I$ from the ambient algebra $\cM$.  Specifically, the \emph{microstate spaces for $\mathbf{y}$ in the presence of $\mathbf{x}$ relative to $C^{(n)} \leadsto z$} are defined as follows:  For each $\mathbf{R} \in (0,+\infty)^{I \sqcup J \sqcup \{0\}}$ with $\norm{x_i} < R_i$ and $\norm{y_j} < R_j$ and $\norm{z} < R_0$, for each sequence of microstates $C^{(n)}$ for $z$ as above with $\norm{C^{(n)}} < R_0$, and for each neighborhood $\mathcal{U}$ of $\lambda_{(\mathbf{x}, \mathbf{y}, z)} \in \Sigma_{\mathbf{R}}$, define
\begin{align*}
\Gamma_{\mathbf{R},n}&(\mathbf{y}: \mathbf{x} | C^{(n)} \rightsquigarrow z; \mathcal{U})\\
    &:= \{\mathbf{B} \in M_n(\C)_{sa}^J: \exists \mathbf{A} \in M_n(\C)_{sa}^I \text{ such that } \norm{A_i} \leq R_i, \norm{B_j} \leq R_j, \lambda_{(\mathbf{A},\mathbf{B},C^{(n)})} \in \mathcal{U}\}
\end{align*}
In other words, it is the projection onto the $J$-coordinates of a microstate space of $(\mathbf{x}, \mathbf{y})$ relative to $C^{(n)} \leadsto z$, or a space of microstates for $\mathbf{y}$ for which there exist compatible microstates for $\mathbf{x}$ (again relative to $C^{(n)} \leadsto z$).

In the case of finitely many variables, the $1$-bounded entropy is defined using $\eps$-covering numbers for the microstate spaces.  For infinitely many variables, one can proceed by looking at the finite-dimensional marginals.  Alternatively, we can use coverings by cylinder sets that are defined by looking at balls in finitely many coordinates, as we will do here.  These cylinder sets have a natural interpretation in the framework of uniform spaces in topology, because they form a fundamental system of neighborhoods for the uniform structure on $M_n(\C)_{sa}^I$ given as the product of the uniform structures on $M_n(\C)_{sa}$ obtained from the $2$-norm.  However, to minimize the technical background, we will state all our definitions directly for this particular case rather than importing the entire formalism of uniform structures.

Given $\mathbf{A} \in M_n(\C)_{sa}^I$, we define for each finite $F \subseteq I$ and $\eps > 0$,
\[
N_{F,\eps}(\mathbf{A}) = \{\mathbf{B} \in M_n(\C)_{sa}^I: \norm{A_i - B_i}_2 < \eps \text{ for all } i \in F\},
\]
and we refer to this set as the \emph{$(F,\eps)$-neighborhood centered at $\mathbf{A}$}.  Similarly, we define the \emph{$(F,\eps)$-neighborhood} of a set $\Omega \subseteq M_n(\C)_{sa}^I$
\[
N_{F,\eps}(\Omega) = \bigcup_{\mathbf{A} \in \Omega} N_{F,\eps}(\mathbf{A}),
\]
that is, the set of $\mathbf{B}$ such that there exists $\mathbf{A}\in \Omega$ with $\norm{A_i - B_i}_2 < \eps$ for all $i \in F$.

For a set $\Omega \subseteq M_n(\C)_{sa}^I$, we define the covering number
\[
K_{F,\eps}(\Omega) = \inf \{ |\Omega_0|: \Omega_0 \subseteq \Omega \text{ and } \Omega \subseteq N_{F,\epsilon}(\Omega_0)\}.
\]
In other words, this is the minimum number of points whose $(F,\eps)$-neighborhoods cover $\Omega$.  (This will be finite if $\Omega$ is contained in a product of $\norm{\cdot}_2$-balls $\{B_i(x_i,R_i)\}_{i \in I}$.)

\begin{defn}
Let $(\cM,\tau)$ be a diffuse tracial $\mathrm{W}^*$-algebra and $\cN \leq \cM$.  Fix some $z \in \cN_{sa}$ with diffuse spectrum, some $R_0 > \norm{z}$, and a sequence $C^{(n)} \in M_n(\C)$ such that $\norm{C^{(n)}} \leq R_0$ and ${C^{(n)}} \rightsquigarrow z$.
Fix a tuple $\mathbf{y} \in \cN_{sa}^J$ of generators for $\cN$ and a tuple $\mathbf{x} \in \cM_{sa}^I$ of generators for $\cM$.  Fix $\mathbf{R} \in (0,+\infty)^{I \sqcup J}$ and $R_i > \norm{x_i}$ for $i \in I$ and $R_j > \norm{y_j}$ for $j \in J$.  Then we define
\[
h(\cN: \cM) := \sup_{F,\eps} \inf_{\mathcal{U}} \limsup_{n \to \infty} \frac{1}{n^2} \log K_{F,\eps} \bigl(\Gamma_{\mathbf{R},n}(\mathbf{y}: \mathbf{x} | C^{(n)} \rightsquigarrow z; \mathcal{U}) \bigr),
\]
where the infimum is taken over all neighborhoods $\mathcal{U}$ of $\lambda_{(\mathbf{x},\mathbf{y},z)}$, and the supremum is taken over all $F \subseteq I$ finite and all $\eps > 0$.
\end{defn}

It is implicit from \cite[Theorem 3.2]{Jung2007} and explicitly shown in \cite[Theorem A.9]{Hayes2018} that $h(\cN: \cM)$ is well-defined in the sense that it is independent of the choice of generators $\mathbf{x}$ and $\mathbf{y}$, the choice of $z \in \cN_{sa}$, and the choice of microstates $C^{(n)}$.  The fact that it is independent of the choice of $\mathbf{R}$ follows from completely standard techniques (compare \cite[proof of Proposition 2.4]{VoiculescuFreeEntropy2} and \cite{BB2003}). We call $h(\cN:\cM)$ the \emph{$1$-bounded entropy of $\cN$ in the presence of $\cM$}. We write $h(\mathcal{M}):=h(\mathcal{M}:\mathcal{M})$ call this the \emph{$1$-bounded entropy of $\cM$.}

We list here some important properties of $1$-bounded entropy, along with pointers in the literature to where they are proved. All von Neumann algebras listed in the properties below are assumed to be tracial and embeddings are assumed to be trace-preserving.

\begin{list}{{\bf P\arabic{pcounter}:} ~ }{\usecounter{pcounter}}

\item $h(\mathcal{N}:\mathcal{M})\geq 0$ if $\mathcal{N}\leq \mathcal{M}$ and every von Neumann subalgebra of $\mathcal{M}$ with separable predual embeds into an ultrapower of $\mathcal{R}$, and $h(\cN:\cM)=-\infty$ if there exists a von Neumann subalgebra of $\cM$ with separable predual which does not embed into an ultrapower of $\mathcal{R}$. (Exercise from the definitions.)

\item $h(\cN_{1}:\cM_{1})\leq h(\cN_{2}:\cM_{2})$ if $\cN_{1}\leq \cN_{2}\leq \cM_{2}\leq \cM_{1}$, if $\cN_{1}$ is diffuse. \label{I:monotonicity of 1 bounded entropy} (Exercise from the definitions.)

\item $h(\cN:\cM)=0$ if $\cN\leq \cM$ and $\cN$ is diffuse and hyperfinite. \label{I:hyperfinite has 1-bdd ent zero}
(Exercise from Lemma \ref{lem:hyper folklore}, using the ``orbital" definition of $1$-bounded entropy given in \cite[Corollary A.6]{Hayes2018}.)

\item For $\cM$ diffuse, $h(\cM)<\infty$ if and only if $\cM$ is strongly $1$-bounded in the sense of Jung. (See \cite[Proposition A.16]{Hayes2018}.) \label{I:preserves SB}

\item $h(\cM)=\infty$ if $\cM = \mathrm{W}^*(x_{1},\cdots,x_{n})$ where $x_{j}\in \cM_{sa}$ for all $1\leq j\leq n$ and $\delta_{0}(x_{1},\cdots,x_{n})>1$. For example, this applies if $\cM=L(\F_{n})$, for $n>1$. (This follows from Property \ref{I:preserves SB} and \cite[Corollary 3.5]{Jung2007}).)

\item $h(\cN_{1}\vee \cN_{2}:\cM)\leq h(\cN_{1}:\cM)+h(\cN_{2}:\cM)$ if $\cN_{1},\cN_{2}\leq \cM$ and $\cN_{1}\cap \cN_{2}$ is diffuse. (See \cite[Lemma A.12]{Hayes2018} \label{I:subadditivity of 1 bdd ent}.)

\item Suppose that $(\cN_{\alpha})_{\alpha}$ is an increasing chain of diffuse von Neumann subalgebras of a von Neumann algebra $\cM$. Then
\[h\left(\bigvee_{\alpha}\cN_{\alpha}:M\right)=\sup_{\alpha}h(\cN_{\alpha}:M).\]
\label{I:increasing limits of 1bdd ent first variable} (See \cite[Lemma A.10]{Hayes2018}.)

\item $h(\cN:\cM)=h(\cN:\cM^{\omega})$ if $\cN\leq \cM$ is diffuse, and $\omega$ is a free ultrafilter on an infinite set. (See \cite[Proposition 4.5]{Hayes2018}.) \label{I:omegafying in the second variable}

\item $h(\mathrm{W}^*(N_{\cM}(\cN)):\cM)=h(\cN:\cM)$ if $\cN\leq\cM$ is diffuse.  Here $N_{\cM}(\cN)=\{u\in \mathcal{U}(\cM):u\cN u^{*}=\cN\}$. (This is a special case of \cite[Theorem 3.8]{Hayes2018}.) \label{I:passing to normalizers preserves entropy}

% \end{enumerate}
\end{list}

By \cite[Theorem 3.8 and Proposition 3.2]{Hayes2018}, we can replace the normalizer in Property \ref{I:passing to normalizers preserves entropy} with various other weakenings of the normalizer. For example, this works for the  quasi-normalizer, the wq-normalizer, or even the singular subspace. We refer the reader to \cite{Hayes2018} (in particular Theorem 3.8 and Proposition 3.2 of that paper) for a more detailed discussion.

% By \cite[Theorem 3.8 and Proposition 3.2]{Hayes2018}, we can replace the normalizer in Property \ref{I:passing to normalizers preserves entropy} with various other weakenings of the normalizer. For example, this works for the  quasi-normalizer as defined in \cite{IzumiLongoPopa, PimnserPopa, PopaQR} (ideas behind the quasi-normalizer trace back to \cite{PopaOrthoPairs}), wq-normalizer as defined in \cite{PopaStrongRigidity, PopaCohomologyOE, IPP, GalatanPopa}. In fact there is a single subspace of $L^{2}(\cM)$ called the ``singular subspace" so that the conclusion of \ref{I:passing to normalizers preserves entropy} holds with the normalizer replaced by the singular subspace, and so that the singular subspace contains all previously defined ``weakenings" of the normalizer. We refer the reader to \cite{Hayes2018} (in particular Theorem 3.8 and Proposition 3.2 of that paper) for a more detailed discussion.

We briefly give some examples of tracial von Neumann algebras with nonpositive $1$-bounded entropy (in particular $1$-bounded entropy zero if they satisfy the Connes embedding conjecture). The main starting point is Property \ref{I:hyperfinite has 1-bdd ent zero}.

\begin{example}
If $\cM$ has a Cartan subalgebra, then $h(\cM)\leq 0$. This follows from Properties \ref{I:hyperfinite has 1-bdd ent zero} and \ref{I:passing to normalizers preserves entropy}. If we analyze the proofs of these properties,
%\ref{I:hyperfinite has 1-bdd ent zero},\ref{I:passing to normalizers preserves entropy},
then the proof of the fact that $\cM$ having a Cartan implies $h(\cM)\leq 0$ is not substantially different than the proof of absence of Cartan for $L(\F_n)$, $n>1$ in \cite{Voiculescu1996}.
\end{example}

\begin{example}
If $\cM$ is non-prime, then $h(\cM)\leq 0$. To see this, suppose that $\cM=\cM_{1}\overline{\otimes}\cM_{2}$, where $\cM_{1},\cM_{2}$ are diffuse. Fix $\cA_{j}\leq \cM_{j}$ diffuse and abelian. Then $h(\cM_{1}\overline{\otimes}\cA_{2}:\cM),h(\cA_{1}\overline{\otimes}\cM_{2}:\cM)\leq 0$ by Properties \ref{I:passing to normalizers preserves entropy}, \ref{I:hyperfinite has 1-bdd ent zero}, and \ref{I:monotonicity of 1 bounded entropy}. Thus by Property \ref{I:subadditivity of 1 bdd ent},
\[h(\cM)=h(\cM:\cM)\leq h(\cM_{1}\overline{\otimes}\cA_{2}:\cM)+h(\cA_{1}\overline{\otimes}\cM_{2}:\cM)\leq 0.\]
This is similar to Ge's proof of primeness of free group factors in \cite{GePrime}.
\end{example}

\begin{example}
If $\cN\leq \cM$ and $\cN$ has diffuse center, then $h(\cN:\cM)\leq 0$. Indeed, by Properties \ref{I:passing to normalizers preserves entropy} and \ref{I:hyperfinite has 1-bdd ent zero},
\[
h(\cN:\cM) \leq h(W^{*}(N_{\cM}(Z(\cN))):\cM) \leq h(Z(\cN):\cM)\leq 0.
\]
\end{example}

\begin{example}
% If every non-zero direct summand of $\cM$ has property Gamma, then $h(\cM)\leq 0$.
Suppose there is a tracial von Neumann algebra $(\cM_{0},\tau_{0})$ which is either zero or has diffuse center, a countable (potentially empty) set $I$, and $\textrm{II}_{1}$-factors $(\cM_{i})_{i\in I}$ with property Gamma so that
\[\cM=\cM_{0}\oplus \bigoplus_{i\in I}\cM_{i}.\]
Then $h(\cM)\leq 0$.  To see this, note that  by \cite[Proposition 1.10]{DixmierGamma} property Gamma for every non-zero direct summand implies
that we can find a sufficiently large infinite set $J$, and a free ultrafilter $\omega$ on $J$, so that there is a diffuse, abelian $\cA\leq \cM'\cap \cM^{\omega}$. We then have by Properties \ref{I:omegafying in the second variable}, \ref{I:monotonicity of 1 bounded entropy}, \ref{I:passing to normalizers preserves entropy}, and \ref{I:hyperfinite has 1-bdd ent zero} that
\[
h(\cM) = h(\cM:\cM) = h(\cM:\cM^{\omega}) \leq h(\mathrm{W}^*(N_{\cM^{\omega}}(\cA)):\cM^{\omega}) = h(\cA:\cM^{\omega})\leq h(\cA:\cA) = h(\cA) = 0.
\]
This is similar to Voiculescu's proof that free group factors do not have Property Gamma in \cite{Voiculescu1996}. Note that the same proof shows that if $\cQ \leq \cM$ is diffuse and $\cQ'\cap \cM^{\omega}$ is diffuse, then $h(\cQ:\cM)\leq 0$.
\end{example}

More generally, any von Neumann algebra with Property (C') in the sense of Galatan-Popa \cite[Definition 3.6]{GalatanPopa}  has $1$-bounded entropy zero (see \cite[Corollary 4.8]{Hayes2018}). There are even examples of von Neumann algebras which have the Haagerup property, are strongly solid, and have the complete metric approximation property which have nonpositive $1$-bounded entropy (see \cite{Houdayerexotic, HoudShlStrongSolid}, as well as \cite[Corollary 4.13]{Hayes2018}).
We remark that it is not known if $h(\cM) < \infty$ implies $h(\cM)\leq 0$.  Thus there is \emph{a priori} a gap between  strongly $1$-bounded algebras and algebras with nonpositive $1$-bounded entropy.
% One important aspect of $1$-bounded entropy is that it singles out a collection of subalgebras of a given tracial von Neumann algebra which automatically, as a consequence of the general properties of $1$-bounded entropy, have strong absoprtion/singularity properties.
Another important consequence of the general properties of $1$-bounded entropy is the existence of Pinsker algebras.
Indeed, by Properties \ref{I:subadditivity of 1 bdd ent} and \ref{I:increasing limits of 1bdd ent first variable} if $\cN\leq \cM$ is diffuse and has $h(\cN:\cM)=0$, then there is a \emph{unique} Pinsker algebra $\mathcal{P} \leq \cM$ with $\cN \subseteq \mathcal{P}$. In particular, if $\mathcal{P}_{1}$, $\mathcal{P}_{2}$ are two Pinsker algebras and $\mathcal{P}_{1}\cap \mathcal{P}_{2}$ is diffuse, then $\mathcal{P}_1 = \mathcal{P}_{2}$.

The motivation and terminology for Pinsker algebras comes from ergodic theory. In the context of dynamical entropy of probability measure-preserving actions of groups it is well known that there is a maximal factor action which has nonpositive entropy, called the \emph{Pinsker factor}, and that the original action is a complete positive entropy extension of the Pinsker factor action. For the case of $\Z$, this was first established by Pinsker in \cite{PinskerPinskerFactor}. For sofic entropy this is a folklore result, and for Rokhlin entropy this is due to Seward \cite{SewardPinskerProd}.  We should note here that the Pinsker algebra as we have defined it is really more analogous to the \emph{outer Pinsker factor} that appears in the study of sofic entropy (implicit in \cite{KerrPartition} and explicitly defined in \cite{HayesProceedings}) and Rokhlin entropy (see \cite{SewardPinskerProd}).
% Nevertheless, we have elected to call this the \emph{Pinsker algebra} instead of the \emph{outer Pinsker algebra},
% since in ergodic theory
It appears that outer Pinsker factors have better permanence properties than Pinsker factors (see e.g., \cite{HayesProceedings, HayesPinskerProd, SewardPinskerProd, SewardKoopman}), and we expect the same to hold in the $1$-bounded entropy setting.
For example, a probability measure-preserving action always has Lebesgue absolutely continuous spectrum over its outer Pinsker factor (see \cite{HayesProceedings, SewardKoopman}). In the context of $1$-bounded entropy, the analogous result is if $\mathcal{P}\leq \cM$ is Pinsker, then $\mathcal{P}\leq \cM$ is \emph{coarse} in the sense of Popa \cite{PopaWeakInter}. This is merely a rephrasing of \cite[Theorem 3.8]{Hayes2018}.

\subsection{Concentration of Measure} \label{subsec:concentration}

As in Theorem \ref{thm:main}, we will be consider a random matrix tuple $\mathbf{X}^{(k)} \in M_{n(k)}(\C)_{sa}^I$ for a possibly infinite index set $I$.  The distribution of such a random variable is a Borel probability measure $\mu^{(k)}$ on the product space $M_{n(k)}(\C)_{sa}^I$, endowed with the product topology coming from the usual topology on $M_{n(k)}(\C)_{sa}$.

\begin{defn}
Given a probability measure $\mu$ on $M_n(\C)_{sa}^I$, a finite $F \subseteq I$, and $\eps > 0$, we define the \emph{concentration function} of $\mu$ by
\[
\alpha_\mu(F,\eps) = \sup \{ \mu(N_{F,\eps}(\Omega)^c): \Omega \subseteq M_n(\C)_{sa}^I \text{ Borel, } \mu(\Omega) \geq 1/2\},
\]
where $N_{F,\eps}$ is an $(F,\epsilon)$-neighborhood.
\end{defn}

We have slightly modified the usual definition of the concentration function for a probability measure on a metric space.  The concentration $\alpha_\mu(F,\eps)$ used here is simply the metric concentration function for the marginal on the coordinates indexed by $F$ (where the metric is given by the maximum of the $2$-norms of the coordinates).  Alternatively, $\alpha_\mu(F,\eps)$ can be viewed as the concentration function of a uniform structure (rather than a metric) on $M_n(\C)_{sa}^I$, namely the uniform structure given by the neighborhoods $N_{F,\eps}(\mathbf{A})$, which is the product of the uniform structures on $M_n(\C)_{sa}$ given by the $\norm{\cdot}_2$ metric.

\begin{defn} \label{defn:concentration}
Given a sequence $n(k) \to \infty$ and a sequence of probability measures $\mu^{(k)}$ on $M_{n(k)}(\C)_{sa}^I$, we say that $\mu^{(k)}$ \emph{has exponential concentration} if for every finite $F \subseteq I$ and every $\eps > 0$,
\[
\limsup_{k \to \infty} \frac{1}{n(k)^2} \log \alpha_{\mu^{(k)}}(F,\eps) < 0.
\]
\end{defn}

Intuitively, exponential concentration of measure says that if a sequence of Borel sets $\Omega_k \subseteq M_{n(k)}(\C)_{sa}^I$ have measure at least $1/2$, then their $(F,\eps)$-neighborhoods will include everything except for a set of exponentially small measure as $k \to \infty$.  As we will explain later, this concentration phenomenon occurs in many natural examples from random matrix theory.  Another consequence of concentration is that if $\Omega_k$ is not exponentially small, then the complement of the $(F,2\eps)$ neighborhood of $\Omega_k$ must be exponentially small, which follows from the next lemma.

\begin{lem} \label{lem:concentrationdichotomy}
Let $\mu$ be a probability measure on $M_n(\C)_{sa}^I$.  Then for every Borel set $\Omega$, every finite $F \subseteq I$, and every $\eps > 0$, we have
\[
\mu(\Omega) > \alpha_\mu(F,\eps) \implies \mu(N_{F,2\eps}(\Omega)) \geq 1 - \alpha_\mu(F,\eps).
\]
\end{lem}

\begin{proof}
Suppose $\mu(\Omega) > \alpha_\mu(F,\eps)$.  Let $\Upsilon = N_{F,\eps}(\Omega)^c$.  Note that $\Omega \subseteq N_{F,\eps}(\Upsilon)^c$.  Thus, $\mu(N_{F,\eps}(\Upsilon)^c) > \alpha_\mu(F,\eps)$, which implies that $\mu(\Upsilon) < 1/2$, by applying the definition of $\alpha_\mu(F,\eps)$ to the set $\Omega$ in contrapositive.  Thus, $\mu(N_{F,\eps}(\Omega)) = \mu(\Upsilon^c) \geq 1/2$.  So by the definition of the concentration function, we have $\mu(N_{F,2\eps}(\Omega)^c) \leq \alpha_\mu(F,\eps)$.  Hence, $\mu(N_{F,2\eps}(\Omega)) \geq 1 - \alpha_\mu(F,\eps)$.
\end{proof}

\section{An $L^2$-continuous Functional Calculus} \label{sec:functionalcalculus}

Following the ideas of the second author from \cite{Jekel2018,Jekel2019,JekelThesis}, we define a notion of ``functions of several non-commuting real variables'' that is more flexible than the notions of non-commutative polynomials.  The space $\mathcal{F}_{\mathbf{R},\infty}$ which we define has the following useful properties:
\begin{itemize}
	\item Given $(\cM, \tau)$ and $\mathbf{x} \in \cM_{sa}^I$ with $\norm{x_i} \leq R_i$, every element of $\mathrm{W}^*(\mathbf{x})$ can be realized as $f(\mathbf{x})$ for some $f \in \mathcal{F}_{\mathbf{R},\infty}$ (see Proposition \ref{prop:realizationofoperators}).
	\item Given a tuple $\mathbf{f} \in \mathcal{F}_{\mathbf{R},\infty}$, the law of $\mathbf{f}(\mathbf{x})$ depends continuously on the law of $\mathbf{x}$ (see Proposition \ref{prop:pushforwardcontinuity}), which will mean that $\mathbf{f}$ will map microstate spaces for $\mathbf{x}$ into microstate spaces of $\mathbf{f}(\mathbf{x})$ (see Corollary \ref{cor:microstatemapping}).
	\item Given $f \in \mathcal{F}_{\mathbf{R},\infty}$, the evaluation $\mathbf{x} \mapsto f(\mathbf{x})$ is $L^2$-uniformly continuous in a certain sense (see Proposition \ref{prop:L2uniformcontinuity}), which will be helpful later for ``pushing forward'' concentration of measure (see Corollary \ref{cor:concentrationpushforward}).
\end{itemize}

Here we approach the definitions from a slightly different point of view than \cite{Jekel2019} and \cite{JekelThesis}.  Moreover, we generalize to the setting of infinite tuples without assuming embeddability into $\mathcal{R}^\omega$.  However, we make the simplifying restriction of only handling functions that are defined in a product of operator norm balls rather than on arbitrary self-adjoint tuples.

As explained in \cite[{\S 3.1}]{Jekel2019} and \cite[\S 13.1 - 13.2]{JekelThesis}, the space $\mathcal{F}_{\mathbf{R},\infty}$ is closely related to the trace polynomials used in previous work such as \cite{Rains1997,Cebron2013, DHK2013, Kemp2016, Kemp2017, DGS2016}.  Moreover, from the model-theoretic viewpoint of \cite{FHS2013,FHS2014a,FHS2014b}, the functions in $\mathcal{F}_{\mathbf{R},\infty}$ are certain quantifier-free definable functions in the language of tracial von Neumann algebras; see \cite[Remark 3.3]{Jekel2019} and \cite[\S 13.7]{JekelThesis}.  A vector-bundle viewpoint on this space is explained in \cite[\S 13.6]{JekelThesis}.  We also explain in Remark \ref{rem:tracialcompletion} how $\mathcal{F}_{\mathbf{R},\infty}$ can be viewed as the tracial completion of a certain $\mathrm{C}^*$-algebra with respect to a family of traces, similar to ideas in \cite{Ozawa2013} and \cite{BBSTWW2019}.

\subsection{Construction of the Space}

\begin{defn}
Fix an index set $I$ and $\mathbf{R} \in (0,+\infty)^I$.  Consider the space
\[
\mathcal{A}_{\mathbf{R}} = C(\Sigma_{\mathbf{R}}) \otimes \C\ip{t_i: i \in I}.
\]
Given $(\cM,\tau)$ and $\mathbf{x} \in \cM_{sa}^I$ with $\norm{x_i} \leq R_i$, we define the evaluation map
    \begin{align*}
        \ev_{\mathbf{x}}: \mathcal{A}_{\mathbf{R}} &\to \cM\\
        \phi \otimes p &\mapsto \phi(\lambda_{\mathbf{x}}) p(\mathbf{x}).
    \end{align*}
Then we define a semi-norm on $C(\Sigma_{\mathbf{R}}) \otimes \C\ip{t_i: i \in I}$ by
\[
\norm{f}_{\mathbf{R},2} = \sup_{(\cM,\tau), \mathbf{x}} \norm{\ev_{\mathbf{x}}(f)}_{L^2(\cM)},
\]
where the supremum is over all tracial $\mathrm{W}^*$-algebras $(\cM,\tau)$ and all $\mathbf{x}\in \cM_{sa}^I$ with $\norm{x_i} \leq R_i$. Denote by $\mathcal{F}_{\mathbf{R},2}$ the completion of $\mathcal{A}_{\mathbf{R}} / \{f \in \mathcal{A}_{\mathbf{R}}: \norm{f}_{\mathbf{R},2} = 0\}$.
\end{defn}

It is immediate that for every $(\cM,\tau)$, for every self-adjoint tuple $\mathbf{x} \in \cM_{sa}^I$ with $\norm{x_i} \leq R_i$, the evaluation map $\ev_{\mathbf{x}}: \mathcal{A}_{\mathbf{R}} \to \cM$ passes to a well-defined map $\mathcal{F}_{\mathbf{R},2} \to L^2(\cM)$, which we continue to denote by $\ev_{\mathbf{x}}$, and we will also write $f(\mathbf{x}) = \ev_{\mathbf{x}}(f)$.  Moreover, it is clear that $f(\mathbf{x}) := \ev_{\mathbf{x}}(f)$ always lies in $L^2(\mathrm{W}^*(\mathbf{x}))$ because this holds when $f$ is a simple tensor.

It will be convenient often to restrict our attention to elements of $\mathcal{F}_{\mathbf{R},2}$ that are bounded in operator norm, and we will show that these in fact form a $\mathrm{C}^*$-algebra.

\begin{defn}
For $f \in \mathcal{F}_{\mathbf{R},2}$, let us define
\[
\norm{f}_{\mathbf{R},\infty} = \sup_{(\cM,\tau), \mathbf{x}} \norm{\ev_{\mathbf{x}}(f)},
\]
and then set
\[
\mathcal{F}_{\mathbf{R},\infty} = \{f \in \mathcal{F}_{\mathbf{R},2}: \norm{f}_{\mathbf{R},\infty} < +\infty\}.
\]
\end{defn}

\begin{lem}
$\mathcal{F}_{\mathbf{R},\infty}$ is a $\mathrm{C}^*$-algebra with respect the norm $\norm{\cdot}_{\mathbf{R},\infty}$ and the multiplication and $*$-operation arising from the natural ones on simple tensors.
\end{lem}

\begin{proof}
We can define a multiplication map $\mathcal{A}_{\mathbf{R}} \times \mathcal{A}_{\mathbf{R}} \to \mathcal{A}_{\mathbf{R}}$ by
\[
(\phi \otimes p, \psi \otimes q) \mapsto \phi \psi \otimes pq,
\]
and a $*$-operation $(\phi \otimes p)^* = \overline{\phi} \otimes p^*$.  Then the evaluation maps $\ev_{\mathbf{x}}$ are $*$-homomorphisms.

Note that for $f \in \mathcal{A}_{\mathbf{R}}$, we have $\norm{f}_{\mathbf{R},\infty} < +\infty$.  Moreover, $f, g \in \mathcal{A}_{\mathbf{R}}$,
\[
\norm{fg}_{\mathbf{R},2} \leq \norm{f}_{\mathbf{R},\infty} \norm{g}_{\mathbf{R},2}
\]
because $\norm{f(\mathbf{x}) g(\mathbf{x})}_2 \leq \norm{f(\mathbf{x})}_\infty \norm{g(\mathbf{x})}_2$ for every $\mathbf{x}$ coming from a tracial von Neumann algebra.  Therefore, the multiplication passes to a well-defined map
\[
\mathcal{A}_{\mathbf{R}} \times \mathcal{F}_{\mathbf{R},2} \to \mathcal{F}_{\mathbf{R},2}.
\]
On the other hand, given $f \in \mathcal{A}_{\mathbf{R}}$ and $g \in \mathcal{F}_{\mathbf{R},\infty}$, we can check that
\[
\norm{fg}_{\mathbf{R},2} \leq \norm{f}_{\mathbf{R},2} \norm{g}_{\mathbf{R},\infty},
\]
so that multiplication is well-defined $\mathcal{F}_{\mathbf{R},2} \times \mathcal{F}_{\mathbf{R},\infty} \to \mathcal{F}_{\mathbf{R},2}$.  Then by checking that $\norm{fg}_{\mathbf{R},\infty} \leq \norm{f}_{\mathbf{R},\infty} \norm{g}_{\mathbf{R},\infty}$, we see that $\mathcal{F}_{\mathbf{R},\infty}$ has a well-defined multiplication operation.

This multiplication operation is characterized by the fact that for self-adjoint tuples $\mathbf{x}$ from $(\cM,\tau)$ with $\norm{x_i} \leq R_i$, we have $(fg)(\mathbf{x}) = f(\mathbf{x}) g(\mathbf{x})$, since $fg$ is uniquely determined by its evaluation on self-adjoint tuples.  This easily implies associativity of multiplication, compatibility with the $*$-operation, and the $\mathrm{C}^*$-identity for the norm.  Completeness of $\mathcal{F}_{\mathbf{R},\infty}$ is also a standard exercise (every Cauchy sequence in $\norm{\cdot}_{\mathbf{R},\infty}$ would also be Cauchy in $\norm{\cdot}_{\mathbf{R},2}$, and so forth).
\end{proof}

\begin{prop} \label{prop:realizationofoperators}
Given $(\cM, \tau)$ and $\mathbf{x} \in \cM_{sa}^I$ with $\norm{x_i} \leq R_i$, the evaluation map $\ev_{\mathbf{x}}:\mathcal{F}_{\mathbf{R},2}
\to L^2(\mathrm{W}^*(\mathbf{x}))$ is surjective, and it restricts to a surjective $*$-homomorphism $\mathcal{F}_{\mathbf{R},\infty} \to \mathrm{W}^*(\mathbf{x})$.
\end{prop}

\begin{proof}
Let $\mathbf{x}$ be a self-adjoint tuple from $(\cM, \tau)$.  Let $z \in L^2(\mathrm{W}^*(\mathbf{x}))$.  Then there is a sequence of non-commutative polynomials $\{p_k\}_{k \in \N}$ such that $p_k(\mathbf{x}) \to z$ in $L^2(\mathrm{W}^*(\mathbf{x}))$.  By passing to a subsequence, we can assume without loss of generality that $\norm{p_{k+1}(\mathbf{x}) - p_k(\mathbf{x})}_2 < 1/2^k$ for $k \geq 1$.  Now define $\mathcal{U}_k \subseteq \Sigma_{\mathbf{R}}$ by
\[
\mathcal{U}_k = \{\lambda: \lambda((p_{k+1} - p_k)^2) < 1/2^k\}.
\]
Then $\mathcal{U}_k$ is an open subset of $\Sigma_{\mathbf{R}}$ containing $\lambda_{\mathbf{x}}$.  By Urysohn's lemma, there exists $\phi_k \in C(\Sigma_{\mathbf{R}})$ such that
\[
0 \leq \phi_{k} \leq 1,  \qquad \phi_{k}(\lambda_{\mathbf{x}}) = 1, \qquad \phi_{k}|_{\mathcal{U}_k^c} = 0.
\]
This implies that
\[
\norm{\phi_k \otimes (p_{k+1} - p_k)}_{\mathbf{R},2} \leq \frac{1}{2^k}.
\]
Therefore,
\[
1 \otimes p_1 + \sum_{k=1}^\infty \phi_k \otimes (p_{k+1} - p_k)
\]
converges absolutely in $\mathcal{F}_{\mathbf{R},2}$ to some function $f$.  By construction,
\[
f(\mathbf{x}) = p_1(\mathbf{x}) + \sum_{k=1}^\infty (p_{k+1}(\mathbf{x}) - p_k(\mathbf{x})) = z.
\]

Now we turn to the case of $\ev_{\mathbf{x}}: \mathcal{F}_{\mathbf{R},\infty} \to \mathrm{W}^*(\mathbf{x})$, which we already showed is a $*$-homomorphism.  Fix $z \in \mathrm{W}^*(\mathbf{x})$.  We can assume without loss of generality that $z$ is self-adjoint.  Choose a sequence of non-commutative polynomials $\{p_k\}$ and continuous functions $\phi_k$ as above.  Assume without loss of generality that $p_k = p_k^*$ and $\phi_k$ is real.  Then let
\[
f_n = 1 \otimes p_1 + \sum_{k=1}^n \phi_k \otimes (p_{k+1} - p_k),
\]
which is a self-adjoint element of $\mathcal{F}_{\mathbf{R},\infty}$.

Choose $h \in C_0(\R)$ satisfying $h(t) = t$ for $|t| \leq \norm{z}$ and $\norm{h}_{C_0(\R)} \leq \norm{z}$.  Since $\mathcal{F}_{\mathbf{R},\infty}$ is a $\mathrm{C}^*$-algebra, $h(f_n)$ is well defined and $(h(f_n))(\mathbf{x}) = h(f_n(\mathbf{x}))$.  We claim that $h(f_n)$ converges in $\norm{\cdot}_{\mathbf{R},2}$ to some $g \in \mathcal{F}_{\mathbf{R},\infty}$ and that $g(\mathbf{x}) = z$.

We will use the observation that every $h \in C_0(\R)$ is uniformly continuous with respect to $\norm{\cdot}_2$ in the following sense:  For every $\eps > 0$ there exists $\delta > 0$ such that if $(\cM, \tau)$ is a tracial von Neumann algebra and $a, b$ are self-adjoint operators in $\cM$, then $\norm{a - b}_2 < \delta$ implies $\norm{h(a) - h(b)}_2 < \eps$.  Clearly, this holds if $h(t)$ is the resolvent $(t + i)^{-1}$, or if $h$ is a $*$-polynomial in $(t + i)^{-1}$.  But $*$-polynomials in $(t + i)^{-1}$ are dense in $C_0(\R)$ by the Stone-Weierstrass theorem, so the claim holds for all $h \in C_0(\R)$ and hence for our particularly chosen $h$.

To show that $\{h(f_n)\}$ is Cauchy in $\norm{\cdot}_{\mathbf{R},2}$, fix $\eps > 0$.  By the $L^2$-uniform continuity of $h$, we may choose $\delta > 0$ such that $\norm{a - b}_2 < \delta$ implies $\norm{h(a) - h(b)}_2 < \eps$ where $a, b$ as above.  Since $\{f_n\}$ is Cauchy in $\norm{\cdot}_{\mathbf{R},2}$, we have $\norm{f_n - f_m}_{\mathbf{R},2} < \delta$.  So for every self-adjoint tuple $\mathbf{y} \in (\cN,\tau)_{sa}^I$ with $\norm{y_i} \leq R_i$, we have $\norm{f_n(\mathbf{y}) - f_m(\mathbf{y})}_2 < \delta$, so $\norm{h(f_n(\mathbf{y})) - h(f_m(\mathbf{y}))}_2 < \eps$, thus making $\norm{h(f_n) - h(f_m)}_{\mathbf{R},2} \leq \eps$.

So $\{h(f_n)\}$ converges in $\norm{\cdot}_{\mathbf{R},2}$ to some $g$.  Since $\norm{h(f_n)}_{\mathbf{R},\infty} \leq \norm{z}$ for all $n$, we have $\norm{g}_{\mathbf{R},\infty} \leq \norm{z}$.  Finally, since $f_n(\mathbf{x}) \to z$ in $\norm{\cdot}_2$, the $L^2$-uniform continuity of $h$ implies that $h(f_n(\mathbf{x})) \to h(z) = z$ in $\norm{\cdot}_2$, and therefore $g(\mathbf{x}) = z$ as desired.
\end{proof}

\begin{remark} \label{rem:tracialcompletion}
From the $\mathrm{C}^*$-algebraic viewpoint, the space $\mathcal{F}_{\mathbf{R},\infty}$ can be described as the separation-completion of a certain $\mathrm{C}^*$-algebra with respect to a family of traces.  Ozawa \cite[p.\ 351-352]{Ozawa2013} defined the completion of a $\mathrm{C}^*$-algebra with respect to the uniform $2$-norm over \emph{all} traces, but the definition still makes sense if we consider a \emph{subset} $\mathcal{S}$ of the trace space.  Specifically, let $\mathcal{C}$ be a $\mathrm{C}^*$-algebra and $\mathcal{S}$ a nonempty set of traces on $\mathcal{C}$. Then we define
\[
\norm{c}_{\mathcal{S},2} = \sup \{\tau(c^*c)^{1/2}: \tau \in \mathcal{S}\}.
\]
Then $\overline{\mathcal{C}}^{\mathcal{S}}$ is defined to be the set of sequences $(c_n)_{n \in \N}$ from $\mathcal{C}$ which are bounded in operator norm and Cauchy in $\norm{\cdot}_{\mathcal{S},2}$, modulo those sequences which go to zero in $\norm{\cdot}_{\mathcal{S},2}$.  This is a $\mathrm{C}^*$-algebra and there is a canonical map $\mathcal{C} \to \overline{\mathcal{C}}^{S}$.  This map could fail to be injective if the representations of $\mathcal{C}$ associated to traces in $\mathcal{S}$ are not sufficient to recover the operator norm on $\mathcal{C}$.  Thus, $\overline{\mathcal{C}}^{\mathcal{S}}$ is in general a separation-completion rather than a completion.  The idea of completing a $\mathrm{C}^*$-algebra with respect to a family of traces is related to current progress on the classification of $\mathrm{C}^*$-algebras and their $*$-homomorphisms; see for instance \cite{BBSTWW2019}.

We can describe $\mathcal{F}_{\mathbf{R},\infty}$ as a tracial separation-completion as follows.  Let $\mathcal{B}_{\mathbf{R}}$ be the universal $\mathrm{C}^*$-algebra generated by self-adjoint operators $(t_i)_{i \in I}$ with $\norm{t_i} \leq R_i$.  It is well-known that $\Sigma_{\mathbf{R}}$ is isomorphic to the space of (normalized) traces on $\mathcal{B}_{\mathbf{R}}$.  If $\lambda \in \Sigma_{\mathbf{R}}$, then $\tau_\lambda = \delta_\lambda \otimes \lambda$ is a trace on the $\mathrm{C}^*$-tensor product $C(\Sigma_{\mathbf{R}}) \otimes \mathcal{B}_{\mathbf{R}}$ (there is a unique $\mathrm{C}^*$-tensor product since $C(\Sigma_{\mathbf{R}})$ is commutative, whence nuclear).  If $\mathbf{x}$ is an $I$-tuple in $(\cM,\tau)$ with the law $\lambda$ and if $\phi \in C(\Sigma_{\mathbf{R}})$ and $p \in \C\ip{t_i: i \in I}$, then
\[
\tau_\lambda(\phi \otimes p) = \phi(\lambda) \tau(p(\mathbf{x})) = \tau[\ev_{\mathbf{x}}(\phi \otimes p)].
\]
Of course, this identity extends to the algebraic tensor product $C(\Sigma_{\mathbf{R}}) \otimes \C\ip{t_i: i \in I}$.  Hence, for $f$ in the algebraic tensor product,
\[
\norm{f}_{\mathbf{R},2} = \sup \{\tau_\lambda(f^*f)^{1/2}: \lambda \in \Sigma_{\mathrm{R}}\},
\]
which is the uniform $2$-norm associated to the family of traces $\{\tau_\lambda: \lambda \in \Sigma_{\mathbf{R}}\}$ on $C(\Sigma_{\mathbf{R}}) \otimes \mathcal{B}_{\mathbf{R}}$.  One can check that the $*$-homomorphism $C(\Sigma_{\mathbf{R}}) \otimes \C\ip{t_i: i \in I} \to \mathcal{F}_{\mathbf{R},\infty}$ extends to a $*$-homomorphism $\rho: C(\Sigma_{\mathbf{R}}) \otimes \mathcal{B}_{\mathbf{R}} \to \mathcal{F}_{\mathbf{R},\infty}$.  We claim that $\mathcal{F}_{\mathbf{R},\infty}$ is isomorphic to the separation-completion of $C(\Sigma_{\mathbf{R}}) \otimes \mathcal{B}_{\mathbf{R}}$ with respect to the family of traces $\{\tau_\lambda: \lambda \in \Sigma_{\mathbf{R}} \}$, such that $\rho$ corresponds to the canonical map from this $\mathrm{C}^*$-algebra into its separation-completion. The main thing to check is that every element in $\mathcal{F}_{\mathbf{R},\infty}$ can be approximated in $\norm{\cdot}_{\mathbf{R},2}$ by a sequence of elements in the image of $\rho$ that are bounded in $\norm{\cdot}_{\mathbf{R},\infty}$.  It is clear from the definition of $\mathcal{F}_{\mathbf{R},\infty}$ that there is some sequence of self-adjoints in the image of $\rho$ that approximates a given self-adjoint element of $\mathcal{F}_{\mathbf{R},\infty}$ in $\norm{\cdot}_{\mathbf{R},2}$.  To arrange boundedness of the sequence in $\norm{\cdot}_{\mathbf{R},\infty}$, we simply apply a cut-off function $h \in C_0(\R)$ as in the proof of Proposition \ref{prop:realizationofoperators} or of the Kaplansky density theorem.
\end{remark}

\subsection{Push-Forwards of Non-commutative Laws}

Now we turn our attention to the way that tuples from $\mathcal{F}_{\mathbf{R},\infty}$ push forward non-commutative laws.

\begin{defn}
Let $I$ and $I'$ be index sets.  Let $\mathbf{R} \in (0,+\infty)^I$ and $\mathbf{R}' \in (0,+\infty)^{I'}$.  We define
\[
\mathcal{F}_{\mathbf{R},\mathbf{R}'} = \{ \mathbf{f} = (f_i)_{i \in I'} \in (\mathcal{F}_{\mathbf{R},\infty})_{sa}^{I'}: \norm{f_i}_{\mathbf{R},\infty} \leq R_i' \text{ for all } i\in I'\}.
\]
\end{defn}

\begin{prop} \label{prop:pushforwardcontinuity}
Let $\mathbf{R} \in (0,+\infty)^I$ and $\mathbf{R}' \in (0,+\infty)^{I'}$.  Let $\mathbf{f} = (f_i)_{i \in I'} \in \mathcal{F}_{\mathbf{R},\mathbf{R}'}$.
\begin{enumerate}
	\item Given $(\cM,\tau)$ and $\mathbf{x} \in \cM_{sa}^I$ with $\norm{x_i} \leq R_i$, we set $\mathbf{f}(\mathbf{x}) = (f_i(\mathbf{x}))_{i \in I'}$.  Then $\lambda_{\mathbf{f}(\mathbf{x})}$ is uniquely determined by $\lambda_{\mathbf{x}}$.
	\item Let $\mathbf{f}_*$ be the ``push-forward'' mapping $\Sigma_{\mathbf{R}} \to \Sigma_{\mathbf{R}'}$ defined by $\mathbf{f}_* \lambda_{\mathbf{x}} = \lambda_{\mathbf{f}(\mathbf{x})}$ for all such tuples $\mathbf{x}$.  Then $\mathbf{f}_*$ is continuous.
\end{enumerate}
\end{prop}

\begin{proof}
It suffices to show that for every non-commutative polynomial $p \in \C\ip{t_i: i \in I'}$, the quantity $\tau(p(\mathbf{f}(\mathbf{x})))$ is uniquely determined by $\lambda_{\mathbf{x}}$, and that it depends continuously on $\lambda_{\mathbf{x}}$.  Now since $\mathcal{F}_{\mathbf{R},\infty}$ is a $\mathrm{C}^*$-algebra, $f := p(\mathbf{f})$ is an element of $\mathcal{F}_{\mathbf{R},\infty} \subseteq \mathcal{F}_{\mathbf{R},2}$.  Thus, it suffices to show that for $f \in \mathcal{F}_{\mathbf{R},2}$, the quantity $\tau(f(\mathbf{x}))$ is uniquely determined by $\lambda_{\mathbf{x}}$ and depends continuously on it.  Now $|\tau(f(\mathbf{x}))| \leq \norm{f(\mathbf{x})}_2 \leq \norm{f}_{\mathbf{R},2}$, so we can reduce to the case where $f$ comes from a dense subset, say $\mathcal{A}_{\mathbf{R}}$ (or rather its image under the quotient map).  Then by linearity, we reduce to the case where $f$ is given by a simple tensor $\phi \otimes p$.  But in this case,
\[
\tau(f(\mathbf{x})) = \phi(\lambda_{\mathbf{x}}) \tau(p(\mathbf{x})) = \phi(\lambda_{\mathbf{x}}) \lambda_{\mathbf{x}}(p),
\]
which only depends on $\lambda_{\mathbf{x}}$, and which is given by the continuous function $\lambda \mapsto \phi(\lambda) \lambda(p)$ on $\Sigma_{\mathbf{R}}$.
\end{proof}

The elements of $\mathcal{F}_{\mathbf{R},\mathbf{R}'}$ can be used to map between microstate spaces in the following way.  Let $\mathbf{f} \in \mathcal{F}_{\mathbf{R},\mathbf{R}'}$.  If $\mathbf{A} \in M_n(\C)_{sa}^I$ satisfies $\norm{A_i} \leq R_i$, then $\mathbf{f}(\mathbf{A})$ is a well-defined element of $M_n(\C)_{sa}^{I'}$ (since $M_n(\C)$ is a tracial $\mathrm{W}^*$-algebra).

\begin{cor} \label{cor:microstatemapping}
Let $\mathbf{R} \in (0,+\infty)^I$, $\mathbf{R}' \in (0,+\infty)^{I'}$, and $\mathbf{f} \in \mathcal{F}_{\mathbf{R},\mathbf{R}'}$.  Let $\mathbf{x}$ be a self-adjoint tuple from $(\cM, \tau)$ with $\norm{x_i} \leq R_i$.
\begin{enumerate}
	\item For every neighborhood $\mathcal{V}$ of $\lambda_{(\mathbf{x},\mathbf{f}(\mathbf{x}))}$ in $\Sigma_{(\mathbf{R},\mathbf{R}')}$, there exists a neighborhood $\mathcal{U}$ of $\lambda_{\mathbf{x}}$ such that
	\[
	\mathbf{f} \left( \Gamma_{\mathbf{R},n}(\mathbf{x}; \mathcal{U}) \right) \subseteq \Gamma_{(\mathbf{R},\mathbf{R}'),n}(\mathbf{f}(\mathbf{x}): \mathbf{x}; \mathcal{V}) \text{ for all } n.
	\]
	\item Similarly, fix $R_0$, a $z \in \cM_{sa}$ with diffuse spectrum and microstates $(C^{(n)})_{n\in \mathbb{N}}$ for $z$.  Then for every neighborhood $\mathcal{V}$ of $\lambda_{(\mathbf{x},\mathbf{f}(\mathbf{x}),z)}$, there exists a neighborhood $\mathcal{U}$ of $\lambda_{(\mathbf{x},z)}$ such that
	\[
	\mathbf{f} \left( \Gamma_{\mathbf{R},n}(\mathbf{x} | C^{(n)} \rightsquigarrow z; \mathcal{U}) \right) \subseteq \Gamma_{(\mathbf{R},\mathbf{R}'),n}(\mathbf{f}(\mathbf{x}): \mathbf{x} | C^{(n)} \rightsquigarrow z; \mathcal{V}) \text{ for all } n.
	\]
\end{enumerate}
\end{cor}

\begin{proof}
(1) It follows from the previous result that the non-commutative law of $(\mathbf{a}, \mathbf{f}(\mathbf{a}))$ depends continuously on the non-commutative law of $\mathbf{a}$.  This means that $\mathcal{U} = (\id, \mathbf{f})^{-1}(\mathcal{V})$ is open in $\Sigma_{\mathbf{R}}$.

(2) The argument is similar, using the fact that the law of $(\mathbf{a}, \mathbf{f}(\mathbf{a}), \mathbf{c})$ depends continuously on the law of $(\mathbf{a}, \mathbf{c})$.
\end{proof}

\subsection{$L^2$-uniform Continuity}

The elements of $\mathcal{F}_{\mathbf{R},2}$ are all ``$L^2$-uniformly continuous functions'' in the following sense.

\begin{prop} \label{prop:L2uniformcontinuity}
Let $I$ be an index set and $\mathbf{R} \in (0,+\infty)^I$, and let $f \in \mathcal{F}_{\mathbf{R},2}$.  Then for every $\eps > 0$ there exists a finite $F \subseteq I$ and a $\delta > 0$ such that for every $(\cM,\tau)$ and $\mathbf{x}$, $\mathbf{y} \in \cM_{sa}^I$ with $\norm{x_i}, \norm{y_i} \leq R_i$, if $\norm{x_i - y_i}_2 < \delta$ for all $i \in F$, then $\norm{f(\mathbf{x}) - f(\mathbf{y})}_2 < \epsilon$.
\end{prop}

\begin{proof}
First, suppose that $f$ has the form $1 \otimes p$ where $p$ is a non-commutative polynomial (or in other words, $f(\mathbf{x}) = p(\mathbf{x})$).  Then it is a straightforward exercise to check this uniform continuity property (for instance, by handling each monomial explicitly).

Second, consider the case where $f$ has the form $\phi \otimes 1$ for some $\phi \in C(\Sigma_{\mathbf{R}})$.  Let $A$ be the set of all $\phi \in C(\Sigma_{\mathbf{R}})$ such that $\phi \otimes 1$ satisfies the above uniform continuity property.  One checks easily that $A$ is a $*$-subalgebra of $C(\Sigma_{\mathbf{R}})$.  Moreover, since $\norm{\phi \otimes 1}_{\mathbf{R},2}$ is simply $\norm{\phi}_{C(\Sigma_{\mathbf{R}})}$, we see that $A$ is closed in $C(\Sigma_{\mathbf{R}})$.

In light the first case, $A$ contains every function of the form $\phi(\lambda_{\mathbf{x}}) = \tau(p(\mathbf{x})))$ for $p \in \C\ip{t_i: i \in I}$.  This means that $A$ separates points in $\Sigma_{\mathbf{R}}$, by definition of $\Sigma_{\mathbf{R}}$.  Therefore, by the Stone-Weierstrass Theorem (since $A$ contains $1$), we have $A = C(\Sigma_{\mathbf{R}})$.  Thus, we have the desired continuity property for $\phi \otimes 1$.

Finally, we check that every function of the form $\phi \otimes p$ in $\mathcal{F}_{\mathbf{R},2}$ satisfies the desired continuity property by using Cases 1 and 2 and the inequality
\[
\norm{\phi(\lambda_{\mathbf{x}}) p(\mathbf{x}) - \phi(\lambda_{\mathbf{y}}) p(\mathbf{y})}_2 \leq \norm{\phi}_{C(\Sigma_{\mathbf{R}})} \norm{p(\mathbf{x}) - p(\mathbf{y})}_2 + |\phi(\lambda_{\mathbf{x}}) - \phi(\lambda_{\mathbf{y}})| \norm{1 \otimes p}_{\mathbf{R},2}.
\]
Linear combinations of these functions are dense by definition of $\mathcal{F}_{\mathbf{R},2}$, and hence the proof is complete.
\end{proof}

The $L^2$-uniform continuity has several consequences that we will use in our handling of microstate spaces.  The first is well known to experts.  We will deduce it from the previous lemma, although it is also easy to prove directly.

\begin{lem}
Let $I$ be an index set and $\mathbf{R} \in (0,+\infty)^I$.  Let $\mathcal{U} \subseteq \Sigma_{\mathbf{R}}$ be open and let $\mu \in \mathcal{U}$.  Then there exists an open $\mathcal{V} \ni \mu$, a finite $F \subseteq I$, and an $\delta > 0$ such that for any tracial von Neumann algebra $(\cM, \tau)$ and $\mathbf{x}, \mathbf{y} \in \cM_{sa}^I$ with $\norm{x_i}, \norm{y_i} \leq R_i$, if $\lambda_{\mathbf{x}} \in \mathcal{V}$ and $\norm{x_i - y_i}_2 < \delta$, then $\lambda_{\mathbf{y}^{(k)}} \in \mathcal{V}$.
\end{lem}

\begin{proof}
By Urysohn's lemma, there exists $\phi \in C(\Sigma_{\mathbf{R}})$ such that $\phi(\mu) = 1$ and $\phi$ is supported in $\mathcal{U}$.  Let $\mathcal{V} = \{\lambda: \phi(\lambda) > 1/2\}$.  By the previous lemma, there exist $F \subseteq I$ finite and $\delta > 0$ such that for any $(\cM, \tau)$ and $\mathbf{x}, \mathbf{y} \in \cM_{sa}^I$ with $\norm{x_i}, \norm{y_i} \leq R_i$, if $\norm{x_i - y_i}_2 < \delta$ for $i \in I$ implies that $|\phi(\lambda_{\mathbf{x}}) - \phi(\lambda_{\mathbf{y}})| < 1/2$.  This choice of $\mathcal{V}$, $F$, and $\delta$ works.
\end{proof}

The analogous statement below for sequences of laws follows immediately.

\begin{cor} \label{cor:perturbedconvergence}
Given $I$ and $\mathbf{R} \in (0,+\infty)^I$.  Let $(\cM^{(k)}, \tau^{(k)})$ for $k \in \N$ be a sequence of tracial von Neumann algebras, and $\mathbf{x}^{(k)}, \mathbf{y}^{(k)} \in (\cM^{(k)})_{sa}^I$ with $\norm{x_i^{(k)}}, \norm{y_i^{(k)}} \leq R_i$.  If $\lambda_{\mathbf{x}^{(k)}} \to \lambda$ in $\Sigma_{\mathbf{R}}$ and if $\norm{x_i^{(k)} - y_i^{(k)}}_2 \to 0$ as $k \to \infty$ for each $i \in I$, then $\lambda_{\mathbf{y}} \to \lambda$.
\end{cor}

Another consequence of the $L^2$-uniform continuity is that exponential concentration of measure is preserved when we push forward  a sequence of probability measures on matrix tuples through a function $\mathbf{f} \in \mathcal{F}_{\mathbf{R},\mathbf{R}'}$.  The push-forward here is technically different from that of Proposition \ref{prop:pushforwardcontinuity}.  If $\mu^{(k)}$ is a probability measure on $M_{n(k)}(\C)_{sa}^I$ supported on matrix tuples $\mathbf{A}$ with $\norm{A_i} \leq R_i$, then $\mathbf{f}_* \mu^{(k)}$ is defined by using the evaluation of $\mathbf{f}$ on tuples $\mathbf{A}$ of $n(k) \times n(k)$ self-adjoint matrices satisfying the given operator norm bounds.  The following result on preservation of exponential concentration will of course be used later for the measures coming from the random matrix models in Theorem \ref{thm:main}.

\begin{cor} \label{cor:concentrationpushforward}
Let $\mathbf{R} \in (0,+\infty)^I$ and $\mathbf{R}' \in (0,+\infty)^{I'}$.  Let $\mu^{(k)}$ be a probability measure on $M_{n(k)}(\C)_{sa}^I$ supported on $\{\mathbf{A}: \norm{A_i} \leq R_i\}$, and let $\mathbf{f} \in \mathcal{F}_{\mathbf{R},\mathbf{R}'}$.  If $(\mu^{(k)})_{k \in \N}$ has exponential concentration, then so does $(\mathbf{f}_* \mu^{(k)})_{k\in \N}$.
\end{cor}

\begin{proof}
Fix a finite $F' \subseteq I'$ and $\eps > 0$.  Applying the previous proposition to each $f_i$ for $i \in F'$, we see that there exists $F \subseteq I$ and $\delta > 0$ such that for tuples of matrices $\mathbf{A}$ and $\mathbf{B}$ with $\norm{A_i} \leq R_i$ and $\norm{B_i} \leq R_i$, we have
\[
\norm{A_i - B_i}_2 < \delta \text{ for } i \in F \implies \norm{f_i(\mathbf{A}) - f_i(\mathbf{A})}_2 < \eps \text{ for } i \in F'.
\]
Therefore, if $\Omega \subseteq M_{n(k)}(\C)_{sa}^{I'}$ with $(\mathbf{f}_* \mu^{(k)})(\Omega) = \mu^{(k)}(\mathbf{f}^{-1}(\Omega)) \geq 1/2$, then $N_{F,\delta}(\mathbf{f}^{-1}(\Omega)) \subseteq \mathbf{f}^{-1}(N_{F',\eps}(\Omega))$, and consequently
\[
\mu^{(k)}(N_{F,\delta}(\mathbf{f}^{-1}(\Omega))^c) \geq (\mathbf{f}_* \mu^{(k)})( N_{F',\eps}(\Omega)^c).
\]
Since this holds for all Borel $\Omega$, we get
\[
\alpha_{\mathbf{f}_* \mu^{(k)}}(F',\eps) \leq \alpha_{\mu^{(k)}}(F,\delta),
\]
and hence
\[
\limsup_{k \to \infty} \frac{1}{n(k)^2} \log \alpha_{\mathbf{f}_* \mu^{(k)}}(F',\eps) \leq \limsup_{k \to \infty} \frac{1}{n(k)^2} \log \alpha_{\mu^{(k)}}(F,\delta) < 0. \qedhere
\]
\end{proof}

\begin{remark}
The space $\mathcal{F}_{\mathbf{R},\mathbf{R}'}$ can be used to slightly simplify the proof that $1$-bounded entropy is a $\mathrm{W}^*$-algebra invariant \cite[Theorem A.9]{Hayes2018}.  By a variant of Corollary \ref{cor:microstatemapping}, we can arrange a function $\mathbf{f}$ that maps one microstate space to the other, and the $L^2$-uniform continuity of $\mathbf{f}$ allows us to push forward a $\delta$-dense subset of the first microstate space to an $\eps$-dense subset of the second one.
\end{remark}

\section{Proof of Theorem \ref{thm:main}} \label{sec:Theorem1}

\subsection{Setup} \label{subsec:thm1setup}

We remind the reader of the setup of Theorem~\ref{thm:main}:  Let $(\cM,\tau)$ be a tracial $\mathrm{W}^*$-algebra and $\mathbf{x} \in \cM_{sa}^I$ be a set of self-adjoint generators indexed by $I$, and let $\mathbf{R} \in (0,+\infty)^I$ with $\norm{x_i} < R_i$.  Suppose that $\cP$ is a $\mathrm{W}^*$-subalgebra of $\cM$.  We assume $n(k) \to \infty$ and that $\mathbf{X}^{(k)}$ is an $I$-tuple of random $n(k) \times n(k)$ matrices, satisfying (1) - (4) below.  For the reader's convenience, we restate these conditions also in terms of the probability distribution $\mu^{(k)}$ of $\mathbf{X}^{(k)}$ (which is a probability measure on $M_{n(k)}(\C)_{sa}^I$).
\begin{enumerate}
	\item $\norm{X_i^{(k)}}_\infty \leq R_i$, or equivalently $\mu^{(k)}$ is supported on $\{\mathbf{A} \in M_{n(k)}(\C)_{sa}^I: \norm{A_i} \leq R_i\}$.
	\item $\tau_{n(k)}(p(\mathbf{X}^{(k)})) \to \tau(p(\mathbf{x}))$ in probability for every non-commutative polynomial $p$.  Equivalently, $\mu^{(k)}$ is \emph{asymptotically supported on the microstate spaces for $\mathbf{x}$}, meaning that
	\[
	\mu^{(k)}(\Gamma_{\mathbf{R},n(k)}(\mathbf{x}; \mathcal{U})) \to 1
	\]
	for every neighborhood $\mathcal{U}$ of $\lambda_{\mathbf{x}}$ in $\Sigma_{\mathbf{R}}$.
	\item The measures $\mu^{(k)}$ exhibit exponential concentration in the sense of Definition \ref{defn:concentration}.
	\item For each non-commutative polynomial $p \in \C\ip{t_i: i \in I}$, we have
	\[
	\lim_{k \to \infty} \norm*{ \E[p(\mathbf{X}^{(k)})] }_2 = \lim_{k \to \infty} \norm*{ \int p(\mathbf{A}) \,d\mu^{(k)}(\mathbf{A}) }_2 = \norm{E_{\cP}[p(\mathbf{x})]}_2.
	\]
	We refer to this condition as the \emph{external averaging property}.
\end{enumerate}
The conclusion to the theorem is that if $\cN \leq \cM$ such that $\cN \cap \cP$ is diffuse and $h(\cN: \cM) = 0$, then $\cN \subseteq \cP$.

\subsection{The External Averaging Property}

As the first ingredient in the proof, we provide several equivalent interpretations of the external averaging property.

\begin{lem} \label{lem:EAPequivalences}
Let $\cP \leq \cM = \mathrm{W}^*(\mathbf{x})$, let $\mathbf{R} \in (0,+\infty)^I$ with $\norm{x_i} \leq R_i$, and let $\mu^{(k)}$ be a sequence of probability measures on $M_{n(k)}(\C)_{sa}^I$ satisfying conditions (1) and (2) of \S \ref{subsec:thm1setup}.  Then the following are equivalent:
\begin{enumerate}
	\item For every non-commutative polynomial $p \in \C\ip{t_i: i \in I}$,
	\[
	\lim_{k \to \infty} \norm*{\int p\,d\mu^{(k)}}_2 = \norm{E_{\cP}(p(\mathbf{x}))}_2.
	\]
	\item For every pair of non-commutative polynomials $p,q\in \C\ip{t_i: i \in I}$,
	\[
	\lim_{k \to \infty} \int \norm*{p(\mathbf{A}) - \int q\,d\mu^{(k)}}_2^2\,d\mu^{(k)}(\mathbf{A}) = \norm{p(\mathbf{x}) - E_{\cP}(q(\mathbf{x}))}_2^2.
	\]
	\item For every $f \in \mathcal{F}_{\mathbf{R},\infty}$,
	\[
	\lim_{k \to \infty} \norm*{\int f \,d\mu^{(k)}}_2 = \norm{E_{\cP}(f(\mathbf{x}))}_2.
	\]
	\item For every pair $f, g \in \mathcal{F}_{\mathbf{R},\infty}$,
	\[
	\lim_{k \to \infty} \int \norm*{f(\mathbf{A}) - \int g\,d\mu^{(k)}}_2^2\,d\mu^{(k)}(\mathbf{A}) = \norm{f(\mathbf{x}) - E_{\cP}(g(\mathbf{x}))}_2^2.
	\]
\end{enumerate}
\end{lem}

\begin{proof}
The overall structure of the proof will be (2) $\iff$ (1) $\iff$ (3) $\iff$ (4).  The implication (3) $\implies$ (1) is immediate because every polynomial is also in $\mathcal{F}_{\mathbf{R},\infty}$.  The implication (2) $\implies$ (1) follows by substituting $q = 0$ and similarly (4) $\implies$ (3) follows by substituting $g = 0$.

To show (1) $\implies$ (3), first suppose that $f$ is given by a sum of simple tensors, that is, $f(\mathbf{y}) = \sum_{j=1}^N \phi_j(\lambda_{\mathbf{y}}) q_j(\mathbf{y})$ for some non-commutative polynomials $q_j$ and $\phi_j \in C(\Sigma_{\mathbf{R}})$.  Define $\tilde{f}(\mathbf{y}) = \sum_{j=1}^N \phi_j(\lambda_{\mathbf{x}}) q_j(\mathbf{y})$, which is a non-commutative polynomial since $\phi_j$ has been replaced by the constant $\phi_j(\lambda_{\mathbf{x}})$. Since $\lambda_{\mathbf{X}^{(k)}} \to \lambda_{\mathbf{x}}$ in probability, we have
\[
\norm*{f(\mathbf{X}^{(k)}) - \tilde{f}(\mathbf{X}^{(k)})}_2 \to 0 \text{ in probability.}
\]
But because $\norm{f - \tilde{f}}_2$ is uniformly bounded, it also converges to zero in expectation, and hence
\[
\lim_{k \to \infty} \norm{\E(f(\mathbf{X}^{(k)}))}_2 = \lim_{k \to \infty} \norm{\E(\tilde{f}(\mathbf{X}^{(k)}))}_2 = \norm{E_{\cP}(\tilde{f}(\mathbf{x}))}_2 = \norm{E_{\cP}(f(\mathbf{x}))}_2,
\]
since $f(\mathbf{x}) = \tilde{f}(\mathbf{x})$.  Thus, (3) holds when $f$ is a finite linear combination of simple tensor.  But every $f \in \mathcal{F}_{\mathbf{R},2}$ (hence every $f \in \mathcal{F}_{\mathbf{R},\infty}$) can be approximated in $\norm{\cdot}_{\mathbf{R},2}$ by linear combinations of simple tensors, and thus by a straightforward approximation argument (3) extends to this case as well.

Next, let us show (3) $\implies$ (4).  Using the polarization identity for inner products, (3) implies that for all $f$, $g \in \mathcal{F}_{\mathcal{R},\infty}$,
\[
\lim_{k \to \infty} \ip*{ \int f\,d\mu^{(k)}, \int g\,d\mu^{(k)} }_2 = \ip*{E_{\cP}[f(\mathbf{x})], E_{\cP}[g(\mathbf{x})]}_2.
\]
Then note that
\[
\int \norm*{f(\mathbf{A}) - \int g \,d\mu^{(k)} }_2^2\,d\mu^{(k)}(\mathbf{A}) = \int \norm{f}_2^2\,d\mu^{(k)} - 2 \re \ip*{ \int f\,d\mu^{(k)}, \int g\,d\mu^{(k)} }_2 + \norm*{ \int g\,\mu^{(k)} }_2^2.
\]
It suffices to show that $\lim_{k \to \infty} \int \norm{f}_2^2\,d\mu^{(k)} = \norm{f(\mathbf{x})}_2^2$, because that would imply
\begin{align*}
\lim_{k \to \infty} \int \norm*{f(\mathbf{A}) - \int g \,d\mu^{(k)} }_2^2\,d\mu^{(k)}(\mathbf{A})
&= \norm*{f(\mathbf{x})}_2^2 - 2 \re \ip*{E_{\cP}[f(\mathbf{x})], E_{\cP}[g(\mathbf{x})]}_2 + \norm{E_{\cP}[g(\mathbf{x})]}_2^2 \\
&= \norm*{f(\mathbf{x}) - E_{\cP}[g(\mathbf{x})}_2^2.
\end{align*}
To prove that $\lim_{k \to \infty} \int \norm{f}_2^2\,d\mu^{(k)} = \norm{f(\mathbf{x})}_2^2$, note that $\norm{f(\mathbf{x})}_2$ is uniformly bounded when $\norm{x_i} \leq R_i$, so that the integrals and conditional expectations above are well-defined.  For self-adjoint tuples $\mathbf{y}$ bounded by $\mathbf{R}$, the expression $\norm{f(\mathbf{y})}_2^2 = \tau(f(\mathbf{y})^2)$ depends continuously on the law of $f(\mathbf{y})$ and hence depends continuously on the law of $\mathbf{y}$ by Proposition \ref{prop:pushforwardcontinuity}.  We assumed that the measures $\mu^{(k)}$ are asymptotically supported on microstates of $\mathbf{x}$.  In other words, for the associated random matrix tuple $\mathbf{X}^{(k)}$, the non-commutative law $\lambda_{\mathbf{X}^{(k)}}$ converges in probability to $\lambda_{\mathbf{x}}$.  Hence, $\norm{f(\mathbf{X}^{(k)})}_2^2 \to \norm{f(\mathbf{x})}_2^2$ in probability.  But it is also uniformly bounded, so convergence in probability implies convergence in expectation.  Thus, $\lim_{k \to \infty} \int \norm{f}_2^2\,d\mu^{(k)} = \norm{f(\mathbf{x})}_2^2$ holds as desired, concluding the proof of (3) $\implies$ (4).

The argument for (1) $\implies$ (2) is the same except that $f$ and $g$ are replaced by non-commutative polynomials $p$ and $q$.  All the facts that we used for $f, g \in \mathcal{F}_{\mathbf{R},\infty}$ hold in particular for non-commutative polynomials.
\end{proof}

In particular, the lemma shows that if $\mu^{(k)}$ has the external averaging property, then for $f \in \mathcal{F}_{\mathbf{R},\infty}$
\[
\lim_{k \to \infty} \int \norm*{ f(\mathbf{A}) - \int f\,d\mu^{(k)} }_2^2\,d\mu^{(k)}(\mathbf{A}) = \norm{f(\mathbf{x}) - E_{\cP}[f(\mathbf{x})]}_2^2,
\]
which follows by taking $g = f$ in (4).  This has following interpretation in terms of Hilbert space geometry.  Let us denote by $L^2(\mu^{(k)},M_{n(k)}(\C))$ the $L^2$-space of functions $(M_{n(k)}(\C)_{sa}^n, \mu^{(k)}) \to (M_{n(k)}(\C), \norm{\cdot}_2)$ (random matrices defined by the measure $\mu^{(k)}$).  The expectation (or integral with respect to $\mu^{(k)}$) is the orthogonal projection onto the subspace of deterministic matrices, and of course the $\mathrm{W}^*$-algebraic conditional expectation $E_{\cP}$ is the orthogonal projection on $L^2(\cP)$.  Thus, the external averaging property says that for $f \in \mathcal{F}_{\mathbf{R},\infty}$, the distance of $f(\mathbf{X}^{(k)})$ from deterministic matrices in $L^2(\mu^{(k)},M_{n(k)}(\C))$ converges to the distance from $f(\mathbf{x})$ to $L^2(\cP)$ in $L^2(\cM)$.  Hence, a function $f(\mathbf{x})$ will be in $\cP$ if and only if $f(\mathbf{X}^{(k)})$ is well-approximated by deterministic matrices in $L^2(\mu^{(k)},M_{n(k)}(\C))$, which we state precisely in the next proposition.

\begin{prop} \label{prop:approximatelyconstant}
Let $\cP \leq \cM = \mathrm{W}^*(X)$, and let $\mu^{(k)}$ be a sequence of measures satisfying (1), (2), and (4) from \S \ref{subsec:thm1setup}.  Let $z \in \mathrm{W}^*(\mathbf{x})_{sa}$ and recall $z = f(\mathbf{x})$ for some $f \in (\mathcal{F}_{\mathbf{R},\infty})_{sa}$ by Proposition \ref{prop:realizationofoperators}.  Then we have $z \in \cP$ if and only if there exists a sequence $(C^{(k)})_{k\in\mathbb{N}}$ of deterministic matrices such that $\norm{f - C^{(k)}}_{L^2(\mu^{(k)},M_{n(k)}(\C))} \to 0$.  Moreover, in this case, we can choose $(C^{(k)})_{k\in\mathbb{N}}$ to be bounded in operator norm.
\end{prop}

\begin{proof}
Suppose that $z = f(\mathbf{x}) \in \cP$.  Then we have
\[
\lim_{k \to \infty} \int \norm*{f(\mathbf{A}) - \int f\,d\mu^{(k)}}_2^2\,d\mu^{(k)}(\mathbf{A}) = \norm{f(\mathbf{x}) - E_{\cP}(f(\mathbf{x}))}_2^2 = 0.
\]
Therefore, we can set $C^{(k)} = \int f(\mathbf{A})\,d\mu^{(k)}(\mathbf{A})$, which is bounded in operator norm because $f$ is bounded in operator norm.

Conversely, suppose that such a sequence of matrices $(C^{(k)})_{k\in\mathbb{N}}$ exists.  Since the expectation of a random matrix is the orthogonal projection onto the subspace of deterministic matrices, we have
\[
\int \norm*{f(\mathbf{A}) - \int f\,d\mu^{(k)}}_2^2\,d\mu^{(k)}(\mathbf{A}) \leq \int \norm*{f(\mathbf{A}) - C^{(k)}}_2^2\,d\mu^{(k)}(\mathbf{A}) \to 0.
\]
Therefore,
\[
\norm{f(\mathbf{x}) - E_{\cP}(f(\mathbf{x}))}_2^2 = \lim_{k \to \infty} \int \norm*{f(\mathbf{A}) - \int f\,d\mu^{(k)}}_2^2\,d\mu^{(k)}(\mathbf{A}) = 0. \qedhere
\]
\end{proof}

\subsection{Microstate Collapse}

To motivate our next main ingredient (Proposition \ref{prop:microstatecollapse}), let us sketch our strategy for proving Theorem \ref{thm:main}.  Consider $\cP \leq \cM = \mathrm{W}^*(\mathbf{x})$ as in the theorem.  We need to show that if $\cN \cap \cP$ is diffuse and $h(\cN:\cM) = 0$, then $\cN \subseteq \cP$.  If $\mathbf{y}$ is a tuple of self-adjoint generators for $\cN$ and $\mathbf{y} = \mathbf{f}(\mathbf{x})$ for some $\mathbf{f} \in \mathcal{F}_{\mathbf{R},\mathbf{R}'}$, we want to show that each $y_i$ must be in $\cP$.  By Proposition \ref{prop:approximatelyconstant}, it suffices to show that $f_i(\mathbf{X}^{(k)})$ is well-approximated in $L^2(\mu^{(k)},M_{n(k)}(\C))$ by a sequence deterministic matrices $C_i^{(k)}$.  In other words, we want to show that most of the mass of $\mathbf{f}_* \mu^{(k)}$ is localized to small neighborhoods $N_{F,\eps}(\mathbf{C}^{(k)})$.

Since $\mathbf{f}_* \mu^{(k)}$ is a natural measure that is asymptotically supported on the microstate space of $\mathbf{y}$ in the presence of $\mathbf{x}$ (by Corollary \ref{cor:microstatemapping}, see proof of Proposition \ref{prop:microstatecollapse} below), what we want to prove is intuitively that ``most'' of the microstates for $\mathbf{y}$ in the presence of $\mathbf{x}$ (or at least those induced from microstates of $\mathbf{x}$) are close together.  Thus, the second main ingredient in the theorem is to show that exponential concentration of measure for $\mu^{(k)}$ together with $h(\cN:\cM) = 0$ causes such a ``collapse'' of the microstate space, which is the conclusion of the following proposition.  We state the proposition is a slightly more general setting, since we believe it has independent interest.

\begin{prop} \label{prop:microstatecollapse}
Let $(\cM,\tau)$ be generated by the self-adjoint tuple $\mathbf{x} \in \cM_{sa}^I$.  Let $\mathbf{y} \in \cM_{sa}^{I'}$, where $I'$ is an arbitrary index set.  Let $z \in \cM$ be a self-adjoint operator with diffuse spectrum.  Let $\cN = \mathrm{W}^*(\mathbf{y},z)$ and suppose that $h(\cN:\cM) = 0$.

Let $\mathbf{R} \in (0,+\infty)^I$, $\mathbf{R}' \in (0,+\infty)^{I'}$, and $R_0 \in (0,+\infty)$ be such that $\norm{x_i} < R_i$ for $i \in I$, $\norm{y_i} < R_i'$ for $i \in I'$, and $\norm{z} < R_0$.  Let $\mathbf{y} = \mathbf{f}(\mathbf{x})$ for some $\mathbf{f} \in \mathcal{F}_{\mathbf{R}, \mathbf{R}'}$.

Let $(C^{(k)})_{k\in\mathbb{N}}$ be a sequence of microstates for $z$ with $\norm{C^{(k)}} < R_0$.  Let $n(k) \to \infty$, and let $\mu^{(k)}$ be a probability measure on $M_{n(k)}(\C)^I$ supported on $\{\mathbf{A}: \norm{A_i} \leq R_i\}$.  Suppose that $\mu^{(k)}$ is asymptotically supported on the microstates of $\mathbf{x}$ relative to $C^{(k)} \rightsquigarrow z$; that is,
\[
\mu^{(k)}( \Gamma_{(\mathbf{R},R_0),n(k)}(\mathbf{x} | C^{(k)} \rightsquigarrow z; \mathcal{U}) ) \to 1
\]
for every neighborhood $\mathcal{U}$ of the non-commutative law of $\mathbf{x}$.  Furthermore, suppose that $\mu^{(k)}$ has exponential concentration.

Then there exist tuples $\mathbf{B}^{(k)} \in M_{n(k)}(\C)_{sa}^{I'}$ with $\norm{B_i^{(k)}} \leq R_i'$ such that for every finite $F \subseteq I'$ and every $\eps > 0$, we have
\[
\lim_{k \to \infty} (\mathbf{f}_* \mu^{(k)})(N_{F,\eps}(\mathbf{B}^{(k)})) = 1.
\]
\end{prop}

\begin{proof}
First, observe that by Corollary \ref{cor:microstatemapping}, $\mathbf{f}_* \mu^{(k)}$ is asymptotically supported on the microstate spaces for $\mathbf{y}$ in the presence of $\mathbf{x}$ relative to $C^{(k)} \rightsquigarrow z$, that is for every neighborhood $\mathcal{U}$ of the law of $(\mathbf{x}, \mathbf{y}, z)$, we have
\[
\lim_{k \to \infty} (\mathbf{f}_* \mu^{(k)})(\Gamma_{\mathbf{R},k}(\mathbf{y}:\mathbf{x}| C^{(k)} \rightsquigarrow z; \mathcal{U})) = 1.
\]
Second, observe that $\mathbf{f}_* \mu^{(k)}$ has exponential concentration by Corollary \ref{cor:concentrationpushforward}.

Now fix $\eps > 0$ and a finite index set $F \subseteq I'$.  Let
\[
\eta := -\limsup_{k \to \infty} \frac{1}{n(k)^2} \log \alpha_{\mathbf{f}_* \mu^{(k)}}(F,\eps/3),
\]
which is strictly positive because $\mathbf{f}_* \mu^{(k)}$ has exponential concentration.  Because $h(\cN:\cM) = 0$, there exists a neighborhood $\mathcal{U}$ of the law of $(\mathbf{x},\mathbf{y},z)$ such that
\[
\limsup_{k \to \infty} \frac{1}{n(k)^2} \log K_{F,\eps}(\Gamma_{(\mathbf{R},\mathbf{R}',R_0),n(k)}(\mathbf{y}: \mathbf{x} | C^{(k)} \rightsquigarrow z; \mathcal{U})) < \frac{\eta}{4}.
\]
Thus, for sufficiently large $k$, we have
\[
K_{F,\eps/3}(\Gamma_{(\mathbf{R},\mathbf{R}',R_0),n(k)}(\mathbf{y}: \mathbf{x} | C^{(k)} \rightsquigarrow z; \mathcal{U})) <  e^{n(k)^2 \eta/4}
\]
and at the same time, since $\mathbf{f}_* \mu^{(k)}$ is asymptotically supported on these microstate spaces, we have
\[
(\mathbf{f}_* \mu^{(k)})(\Gamma_{(\mathbf{R},\mathbf{R}',R_0),n(k)}(\mathbf{y}: \mathbf{x} | C^{(k)} \rightsquigarrow z; \mathcal{U})) \geq 1/2.
\]
Thus, the microstate space can be covered by the $(F,\eps/3)$-neighborhoods of $e^{n(k)^2 \eta / 4}$ many points from the microstate space, while the measure of the whole is at least $1/2$.  So by the pigeonhole principle, there is some $\mathbf{B}^{(k,F,\eps)} \in M_{n(k)}(\C)_{sa}^{I'}$ in the microstate space such that
\[
(\mathbf{f}_* \mu^{(k)})(N_{F,\eps/3}(\mathbf{B}^{(k,F,\eps)})) \geq \frac{1}{2} e^{-n(k)^2 \eta / 4}.
\]
For sufficiently large $k$,
\[
(\mathbf{f}_* \mu^{(k)})(N_{F,\eps/3}(\mathbf{B}^{(k,F,\eps)})) > e^{-n(k)^2 \eta /2} > \alpha_{\mathbf{f}_* \mu^{(k)}}(F,\eps/3),
\]
where the last inequality follows from our choice of $\eta$.  Therefore, by Lemma \ref{lem:concentrationdichotomy}, we have for sufficiently large $k$ that
\begin{align*}
(\mathbf{f}_* \mu^{(k)})(N_{F,\eps}(\mathbf{B}^{(k,F,\eps)})) &= (\mathbf{f}_* \mu^{(k)})(N_{F,2\eps/3}(N_{F,\eps/3}(\mathbf{B}^{(k,F,\eps)}))) \\
&\geq 1 - \alpha_{\mathbf{f}_* \mu^{(k)}}(F,\eps/3) \\
&\geq 1 - e^{-n(k)^2 \eta / 2}.
\end{align*}

To complete the proof, we must arrange that the same $\mathbf{B}^{(k)}$ works for all $(F,\eps)$.  Let $\mathbf{X}^{(k)}$ be a random matrix tuple given by the probability distribution $\mu^{(k)}$ and let $B_i^{(k)} = \E[f_i(\mathbf{X}^{(k)})]$.  Then for every $(F,\eps)$, the probability that $\norm{f_i(\mathbf{X}^{(k)}) - B_i^{(k,F,\eps)}}_2 \leq \eps$ for $i \in F$ tends to $1$.  Since this random variable is also uniformly bounded, we have
\[
\limsup_{k \to \infty} \norm{\E[f_i(\mathbf{X}^{(k)})] - B_i^{(k,F,\eps)}}_2 \leq \eps.
\]
Thus, since $B_i^{(k)} = \E[f_i(\mathbf{X}^{(k)})]$, the probability that $\norm{f_i(\mathbf{X}^{(k)}) - B_i^{(k)}}_2 \leq 2 \eps$ for $i \in F$ tends to $1$.  Since $(F,\eps)$ was arbitrary, we are done.
\end{proof}

Now we finish the proof of our first main theorem.

\begin{proof}[Proof of Theorem \ref{thm:main}]
Assume the setup of \S \ref{subsec:thm1setup}.  In particular, suppose that $\cP \leq \cM = \mathrm{W}^*(\mathbf{x})$, and let $\mu^{(k)}$ be a sequence of random matrix measures satisfying (1) - (4).  Let $\cN$ be a subalgebra of $\cM$ with $\cN \cap \cP$ diffuse and $h(\cP:\cM) = 0$.  Let $\mathbf{y} \in \cN_{sa}^{I'}$ be a set of generators for $\cN$ and let $\mathbf{R}' \in (0,+\infty)^{I'}$ satisfy $\norm{y_i} < R_i'$ for $i \in I'$.

In order to evaluate $1$-bounded entropy and apply Proposition~\ref{prop:microstatecollapse}, we fix $z \in \cN \cap \cP$ self-adjoint with diffuse spectrum, and we obtain a sequence of microstates for $z$ as follows.  By Proposition~\ref{prop:realizationofoperators}, we can write $z = g(\mathbf{x})$ for some $g \in (\mathcal{F}_{\mathbf{R},\infty})_{sa}$.  Let $R_0 > \norm{g}_{\mathbf{R},\infty}$.  By Proposition~\ref{prop:approximatelyconstant}, since $z \in \cP$, there is a sequence $(C^{(k)})_{k \in\mathbb{N}}$ of matrices such that $\norm{g(\mathbf{X}^{(k)}) - C^{(k)}}_{L^2(\mu^{(k)},M_k(\C))} \to 0$, where $\mathbf{X}^{(k)}$ is again a random matrix tuple given by $\mu^{(k)}$.

In order to apply Proposition \ref{prop:microstatecollapse}, we want to check that $\mu^{(k)}$ is asymptotically supported on the microstate spaces for $\mathbf{x}$ relative to $C^{(k)} \rightsquigarrow z$.  This is equivalent to saying that the non-commutative law of $(\mathbf{X}^{(k)},C^{(k)})$ converges to $\lambda_{(\mathbf{x},z)}$ in probability.  But we know by Proposition \ref{prop:pushforwardcontinuity} that the non-commutative law of $(\mathbf{X}^{(k)}, g(\mathbf{X}^{(k)}))$ converges to $\lambda_{(\mathbf{x},g(\mathbf{x}))} = \lambda_{(\mathbf{x},z)}$ in probability.  Also, $\norm{g(\mathbf{X}^{(k)}) - C^{(k)}}_2 \to 0$ in probability.  Hence, by Corollary \ref{cor:perturbedconvergence}, we have $\lambda_{(\mathbf{X}^{(k)},C^{(k)})} \to \lambda_{(\mathbf{x},z)}$ in probability.

Therefore, $\mu^{(k)}$ satisfies the conditions needed for Proposition~\ref{prop:microstatecollapse}.  Choose some tuple $\mathbf{y} \in \cN_{sa}^{I'}$ such that $(\mathbf{y},z)$ generates $\cN$.  Let $\mathbf{R}'$ be a tuple of bounds for the operator norms of $y_i$, and let $\mathbf{f} \in \mathcal{F}_{\mathbf{R},\mathbf{R}'}$ such that $\mathbf{y} = \mathbf{f}(\mathbf{x})$.  Applying the proposition to this $\mathbf{y}$, there exists a sequence $\mathbf{B}^{(k)}$ such that for every finite $F \subseteq I'$ and $\eps > 0$,
\[
\lim_{k \to \infty} (\mathbf{f}_* \mu^{(k)})(N_{F,\eps}(\mathbf{B}^{(k)})) = 1.
\]
In particular, $\norm{f_i(\mathbf{X}^{(k)}) - B_i^{(k)}}_2 \to 0$ for every $i \in I'$.  So $\mathbf{y}_i = f_i(\mathbf{x})$ is in $\cP$ by Proposition \ref{prop:approximatelyconstant}.  But $(\mathbf{y},z)$ generates $\cN$, hence $\cN \subseteq \cP$.
\end{proof}

\section{Examples from Free Probability and Random Matrix Theory} \label{sec:examples}

\subsection{The Generator MASA in $L(\F_d)$ is a Pinsker Algebra}

In this section, we will deduce from Theorem \ref{thm:main} the following proposition.

\begin{prop} \label{prop:generatorMASA}
Let us express the von Neumann algebra $L(\F_d)$ associated to the free group as the tracial free product $L(\Z) * L(\F_{d-1})$.  Then $L(\Z)$ is a Pinsker algebra in $L(\F_d)$, that is, a maximal subalgebra with $1$-bounded entropy zero in the presence of $L(\F_d)$.
\end{prop}

In particular, this implies that $L(\Z)$ is a maximal amenable subalgebra and a maximal subalgebra with property Gamma.  Indeed, if $\cN \geq L(\Z)$ is amenable or Gamma, then $h(\cN: \cN) = 0$, and hence $h(\cN: L(\F_d)) \leq h(\cN: \cN) = 0$.  So if $L(\Z)$ is a Pinsker algebra, then $\cN \subseteq L(\Z)$.

Now $L(\Z)$ is diffuse abelian and has $h(L(\Z): L(\F_d)) = 0$.  So for $L(\Z)$ to be a Pinsker algebra, it would be sufficient to show that if $\cN \cap L(\Z)$ is diffuse and $h(\cN: L(\F_d)) = 0$, then $\cN \subseteq L(\Z)$, which is precisely the conclusion of Theorem \ref{thm:main}.  Actually, it is no more difficult to establish the conclusion of Theorem \ref{thm:main} for the free product of two arbitrary finitely generated Connes-embeddable von Neumann algebras $\cM_1$ and $\cM_2$ rather than only $L(\Z)$ and $L(\F_{d-1})$.

\begin{prop} \label{prop:freeproduct}
Let $\cM_1$ and $\cM_2$ be Connes-embeddable tracial $\mathrm{W}^*$-algebras, and suppose that $\cM_1$ is generated by $\mathbf{x} = (x_1,\dots,x_m)$ and $\cM_2$ is generated by $\mathbf{y} = (y_1,\dots,y_n)$.  Let $(\cM,\tau)$ be the (tracial) free product of $\cM_1$ and $\cM_2$.  Then there exist random matrix models for $\mathbf{x}$ and $\mathbf{y}$ satisfying the hypothesis of Theorem~\ref{thm:main} with $\cP = \cM_1$.  Hence, by that theorem, if $\cN\leq \cM$, and if $\cN \cap \cP$ is diffuse with $h(\cN: \cM) = 0$, then $\cN \subseteq \cP$.
\end{prop}

This proposition is actually a special case of Theorem~\ref{thm:main2}, which we prove in the next section.  However, the proof of Proposition~\ref{prop:freeproduct} requires less technical preparation and already covers the case of the generator MASA in a free group factor.  Therefore, we will prove the Proposition here directly, both as an application of Theorem~\ref{thm:main} and as motivation for the proof of Theorem~\ref{thm:main2}.

\begin{proof}[Proof of Proposition~\ref{prop:freeproduct}]
Because $\cM_1$ and $\cM_2$ are Connes-embeddable, there exist (deterministic) tuples $\mathbf{X}^{(k)} = (X_1^{(k)}, \dots, X_m^{(k)})$ and $\mathbf{Y}^{(k)} = (Y_1^{(k)}, \dots, Y_n^{(k)})$ of $k \times k$ self-adjoint matrices satisfying:
\begin{itemize}
    \item $\norm{X_i^{(k)}} \leq \norm{x_i}$ and $\norm{Y_i^{(k)}} \leq \norm{y_i}$.
    \item $\lambda_{\mathbf{X}^{(k)}} \to \lambda_{\mathbf{x}}$ and $\lambda_{\mathbf{Y}^{(k)}} \to \lambda_{\mathbf{y}}$.
\end{itemize}
Let $U^{(k)}$ be a $k \times k$ Haar random unitary matrix, and consider the random matrix tuple $(\mathbf{X}^{(k)}, U^{(k)} \mathbf{Y}^{(k)} (U^{(k)})^*)$.  We claim that these random matrix models satisfy the hypothesis of Theorem \ref{thm:main} with respect to the generating set $(\mathbf{x},\mathbf{y})$ and the subalgebra $\cM_1 \leq \cM$.

(1) The random matrices are bounded in operator norm by construction.

(2) We must show that $\lambda_{(\mathbf{X}^{(k)}, U^{(k)} \mathbf{Y}^{(k)} (U^{(k)})^*)} \to \lambda_{(\mathbf{x},\mathbf{y})}$ in probability.  Using Voiculescu's asymptotic freeness theorem, specifically \cite[Corollary 2.13]{Voiculescu1998}, if $W^{(k)}$ is another independent Haar unitary, then the non-commutative law of $(W^{(k)} \mathbf{X}^{(k)} (W^{(k)})^*, U^{(k)} \mathbf{Y}^{(k)} (U^{(k)})^*)$ converges in probability to that of $(\mathbf{x},\mathbf{y})$.  But of course, the probability distribution of $\lambda_{(\mathbf{X}^{(k)}, U^{(k)} \mathbf{Y}^{(k)} (U^{(k)})^*)}$ is the same as that of $\lambda_{(W^{(k)} \mathbf{X}^{(k)} (W^{(k)})^*, U^{(k)} \mathbf{Y}^{(k)} (U^{(k)})^*)}$ by unitary invariance of non-commutative law and the fact that $(W^{(k)})^* U^{(k)}$ is also a Haar unitary.

(3) It is well known in random matrix theory that the Haar measure on the $n \times n$ unitary group satisfies exponential concentration of measure as $n \to +\infty$.  In particular, Meckes and Meckes gave explicit constants for the log-Sobolev inequality on the unitary group in \cite[Theorem 15]{Meckes2013}, which implies exponential concentration of measure for $U^{(k)}$ with rate $k^2$ with respect to the metric given by $\norm{\cdot}_2$.  But $(\mathbf{X}^{(k)}, U^{(k)} \mathbf{Y}^{(k)} (U^{(k)})^*)$ is a Lipschitz function of $U^{(k)}$ and hence also has exponential concentration of measure in the sense of Definition \ref{defn:concentration}.  (For details, refer to \S \ref{subsec:freeproductconcentration} below.)

(4) It remains to show the external averaging property. That is, for every non-commutative polynomial $p$ in $m + n$ variables, we have
\[
\lim_{k \to \infty} \norm{\E[p(\mathbf{X}^{(k)}, U^{(k)} \mathbf{Y}^{(k)} (U^{(k)})^*)]}_2 = \norm{E_{\cM_1}[p(\mathbf{x},\mathbf{y})]}_2,
\]
where $\E$ denotes the classical expectation and $E_{\cM_1}$ denotes the $\mathrm{W}^*$-algebraic conditional expectation.

Let $V^{(k)}$ be another Haar unitary independent from $U^{(k)}$.  Then we can rewrite
\begin{align*}
\norm{ \E[p(\mathbf{X}^{(k)}, U^{(k)} \mathbf{Y}^{(k)} (U^{(k)})^*)]}_2^2 &= \tau_{n(k)}( \E[p(\mathbf{X}^{(k)}, U^{(k)} \mathbf{Y}^{(k)} (U^{(k)})^*)]^* \E[p(\mathbf{X}^{(k)}, V^{(k)} \mathbf{Y}^{(k)} (V^{(k)})^*)]) \\
&= \E[ \tau_{n(k)}( p(\mathbf{X}^{(k)}, U^{(k)} \mathbf{Y}^{(k)} (U^{(k)})^*)^* p(\mathbf{X}^{(k)}, V^{(k)} \mathbf{Y}^{(k)} (V^{(k)})^*) ) ].
\end{align*}

Similar to the proof of (2), it follows from \cite[Corollary 2.13]{Voiculescu1998} that the non-commutative law of $(\mathbf{X}^{(k)}, U^{(k)} \mathbf{Y}^{(k)} (U^{(k)})^*, V^{(k)} \mathbf{Y}^{(k)} (V^{(k)})^*)$ converges in probability to the non-commutative law of $(\mathbf{x}, \mathbf{y}, \tilde{\mathbf{y}})$, where $\tilde{\mathbf{y}}$ is a copy of $\mathbf{y}$ freely independent from $\mathbf{x}$ and $\mathbf{y}$.  In other words, $(\mathbf{x},\mathbf{y},\tilde{\mathbf{y}})$ are natural generators for $\cM_1 * \cM_2 * \tilde{\cM}_2$, where $\tilde{\cM}_2$ is a copy of $\cM_2$.

Therefore,
\[
\lim_{k \to \infty} \norm{ \E[p(\mathbf{X}^{(k)}, U^{(k)} \mathbf{Y}^{(k)} (U^{(k)})^*)]}_2^2 = \tau\left( p(\mathbf{x}, \mathbf{y}) ^{*}p(\mathbf{x}, \tilde{\mathbf{y}})\right).
\]
Now note that in the algebra $\cM_1 * \cM_2 * \tilde{\cM}_2$, the two subalgebras $\cM_1 * \cM_2$ and $\cM_1 * \tilde{\cM}_2$ are freely independent with amalgamation over $\cM_1$ (see e.g. \cite[Proposition 4.1]{Hou07}). Therefore, we have
\[
E_{\cM_1}[p(\mathbf{x}, \mathbf{y})^* p(\mathbf{x}, \tilde{\mathbf{y}})] = E_{\cM_1}[p(\mathbf{x},\mathbf{y})]^* E_{\cM_1}[p(\mathbf{x}, \tilde{\mathbf{y}})]
\]
Given that $\cM_1 * \cM_2 \cong \cM_1 * \tilde{\cM}_2$, we have $E_{\cM_1}[p(\mathbf{x},\mathbf{y})] = E_{\cM_1}[p(\mathbf{x},\tilde{\mathbf{y}}]$.  Thus,
\begin{align*}
\lim_{k \to \infty} \norm{ \E[p(\mathbf{X}^{(k)}, U^{(k)} \mathbf{Y}^{(k)} (U^{(k)})^*)]}_2^2
&= \tau\left(E_{\cM_1}[p(\mathbf{x}, \mathbf{y})^* p(\mathbf{x}, \tilde{\mathbf{y}})]\right) \\
&= \tau\left( E_{\cM_1}[p(\mathbf{x},\mathbf{y}]^* E_{\cM_1}[p(\mathbf{x},\mathbf{y})]\right) \\
&= \norm{E_{\cM_1}[p(\mathbf{x},\mathbf{y})]}_2^2,
\end{align*}
which completes the proof of the external averaging property.
\end{proof}

\subsection{Random Matrix Models with Convex Interaction}

A standard random matrix model for the free group factor $L(\F_d)$ is the Gaussian unitary ensemble.  One can consider a tuple of $k \times k$ self-adjoint random matrices $\mathbf{S}^{(k)} = (S_1^{(k)}, \dots, S_d^{(k)})$ with the probability density
\[
\frac{1}{\zeta^{(k)}} e^{-k^2 \sum_{j=1}^d \tau_k(A_j^2) / 2}\,d\mathbf{A},
\]
on $M_k(\C)_{sa}^d$ where $d\mathbf{A}$ is Lebesgue measure and $\zeta^{(k)}$ is a normalizing constant.  Then $\lambda_{\mathbf{S}^{(k)}}$ converges in probability to the non-commutative law of $\mathbf{s} = (s_1,\dots,s_d)$ where the $s_j$'s are freely independent and each have a semicircular spectral distribution $(1/2\pi) \sqrt{4 - t^2} \mathbf{1}_{[-2,2]}(t)\,dt$; see \cite[Theorem 2.2]{Voiculescu1998}.

A natural generalization is a random matrix tuple $\mathbf{X}^{(k)}$ given by density
\[
\frac{1}{\zeta^{(k)}} e^{-k^2 V^{(k)}(\mathbf{A})}\,d\mathbf{A},
\]
where $V^{(k)}(\mathbf{A}) = \tau_k(p(\mathbf{A}))$ for some non-commutative polynomial $p$, such that $e^{-k^2 V^{(k)}}$ is integrable.  These models are much better understood when $V^{(k)}$ is close to the quadratic case, or is at least convex; see \cite{GMS2006,GS2009,GS2014}.  In particular, sufficient assumptions on $V^{(k)}$ will guarantee that $\lambda_{\mathbf{X}^{(k)}}$ converges in probability to $\lambda_{\mathbf{x}}$ for some non-commutative tuple $\mathbf{x}$, and that $\mathrm{W}^*(\mathbf{x}) \cong L(\F_d)$.

Our present goal is to explain how the conditional densities of these random matrix models naturally give rise to random matrix models satisfying the hypotheses of Theorem \ref{thm:main} for the subalgebra generated by a subset of our original generators (with $n(k) = k$).  Changing notation slightly, we consider a function $V^{(k)}: M_k(\C)_{sa}^{m+n} \to \R$ denoted $V^{(k)}(\mathbf{A},\mathbf{B})$ for $\mathbf{A} \in M_k(\C)_{sa}^m$ and $\mathbf{B} \in M_k(\C)_{sa}^n$.  Let $(\mathbf{X}^{(k)}, \mathbf{Y}^{(k)})$ be the corresponding random matrices, and suppose that $(\mathbf{x},\mathbf{y})$ is a tuple of non-commutative random variables describing the large $k$ limit.

The idea is to consider the probability distribution of $(\mathbf{X}^{(k)}, \mathbf{Y}^{(k)})$ conditioned on $\mathbf{X}^{(k)}$ being equal to some $\mathbf{A}^{(k)}$.  In other words, fix a particular sequence $\mathbf{A}^{(k)}$ in $M_k(\C)_{sa}^m$ with $\lambda_{\mathbf{A}^{(k)}} \to \lambda_{\mathbf{x}}$ and then choose $\mathbf{B}^{(k)}$ randomly according to the conditional distribution
\begin{equation} \label{eq:conditionaldistribution}
\frac{1}{\zeta^{(k)}(\mathbf{A}^{(k)})} e^{-k^2 V^{(k)}(\mathbf{A}^{(k)},\mathbf{B})}\,d\mathbf{B}.
\end{equation}
In order to obtain matrix models that are bounded in operator norm, we will replace $B_i^{(k)}$ with $\psi(B_i^{(k)})$ where $\psi: \R \to \R$ is a smooth bounded function satisfying $\psi(t) = t$ for $|t| \leq \norm{y_i}$.  This should not affect the convergence in law because $\psi(\mathbf{y}) = \mathbf{y}$.

Since we will invoke the results of \cite{Jekel2019}, we require the same technical setup.  Rather than assuming $V^{(k)}(\mathbf{A},\mathbf{B}) = \tau_k(p(\mathbf{A},\mathbf{B}))$, we assume that $V^{(k)}: M_k(\C)_{sa}^{m+n} \to \R$ is differentiable and the gradient $DV^{(k)}$ is \emph{asymptotically approximable by trace polynomials}.  In the notation of this paper, the asymptotic approximation condition means that for every $\mathbf{R} \in (0,+\infty)^{m+n}$, there exists $\mathbf{f} \in (\mathcal{F}_{\mathbf{R},2})_{sa}^{m+n}$ such that
\begin{equation} \label{eq:asymptoticapproximation}
\lim_{k \to \infty} \sup_{\substack{(\mathbf{A},\mathbf{B}) \in M_k(\C)_{sa}^{m+n} \\ \norm{(\mathbf{A},\mathbf{B})_i} \leq R_i}} \norm{DV^{(k)}(\mathbf{A},\mathbf{B}) - \mathbf{f}(\mathbf{A},\mathbf{B})}_2 = 0.
\end{equation}
We further assume that $V^{(k)}$ is invariant under unitary conjugation and that for some $0 < c < C$, the Hessian $HV^{(k)}$ satisfies $c \leq HV^{(k)} \leq C$ (see \cite[\S 1.3 and 2.1]{Jekel2019} for precise details).  Under these conditions, \cite[Theorem 4.1]{Jekel2018} shows that the random matrix models converge in non-commutative moments in probability to some $\mathbf{x}$.

\begin{remark}
Although the assumption $HV^{(k)} \leq C$ rules out the case where $V^{(k)}(\mathbf{A},\mathbf{B}) = \tau_k(p(\mathbf{A},\mathbf{B}))$ and $p$  is a non-commutative polynomial of degree greater than $2$, the condition \eqref{eq:asymptoticapproximation} gives us the flexibility to modify our $V^{(k)}$ outside an operator-norm ball, so that the random matrix models can still achieve the same limiting distribution as if $V^{(k)}(\mathbf{A},\mathbf{B}) = \tau_k(p(\mathbf{A},\mathbf{B}))$ when $p$ is non-commutative polynomial that is a small perturbation of a quadratic; see \cite[\S 8.3]{Jekel2018} or \cite[\S 18]{JekelThesis} for details.
\end{remark}

Now the following result from \cite{Jekel2019} describes the large $k$ behavior of the classical conditional expectation of a certain functions $f(\mathbf{X}^{(k)}, \mathbf{Y}^{(k)})$ given $\mathbf{X}^{(k)}$, and shows that it approximates the non-commutative conditional expectation $E_{\mathrm{W}^*(\mathbf{x})}[f(\mathbf{x},\mathbf{y})]$.  (Beware that $X$ and $Y$ are switched in that paper.)  These results are also contained in \S 15 of the thesis \cite{JekelThesis}.

\begin{thm}[{\cite[Theorem 5.9]{Jekel2019}}] \label{thm:conditionalexpectation}
Let $V^{(k)}: M_k(\C)_{sa}^{m+n} \to \R$ satisfy the hypotheses described above, and let $(\mathbf{X}^{(k)}, \mathbf{Y}^{(k)})$ be the random matrix models.  Let $f^{(k)}: M_k(\C)_{sa}^{m+n} \to M_k(\C)$ and suppose that $f^{(k)}$ is uniformly Lipschitz in $\norm{\cdot}_2$ and that $f^{(k)}$ is asymptotically approximable by trace polynomials.  Let $g^{(k)}(\mathbf{A})$ be the function defined by
\[
g^{(k)}(\mathbf{X}^{(k)}) = \E[f^{(k)}(\mathbf{X}^{(k)},\mathbf{Y}^{(k)}) | \mathbf{X}^{(k)}].
\]
Then $g^{(k)}$ is also asymptotically approximable by trace polynomials.  In fact, if $\mathbf{R}$ is a tuple of bounds for the operator norms of $\mathbf{x}$, then there exists $g \in \mathcal{F}_{\mathbf{R},2}$ satisfying
\[
\lim_{k \to \infty} \sup_{\substack{(\mathbf{A},\mathbf{B}) \in M_k(\C)_{sa}^{m+n} \\ \norm{(\mathbf{A},\mathbf{B})_i} \leq R_i}} \norm{g^{(k)}(\mathbf{A},\mathbf{B}) - g(\mathbf{A},\mathbf{B})}_2 = 0
\]
and $g(\mathbf{x}) = E_{\mathrm{W}^*(\mathbf{x})}[f(\mathbf{x},\mathbf{y})]$.
\end{thm}

This is the key to establishing the external averaging property and thus getting random matrix models satisfying the assumptions of Theorem \ref{thm:main}.

\begin{prop} \label{prop:convexpotentials}
Let $V^{(k)}: M_k(\C)_{sa}^d \to \R$ satisfy the assumptions of \cite{Jekel2019} (explained above).  Let $(\mathbf{x},\mathbf{y})$ be the $(m+n)$-tuple non-commutative random variables describing the large $k$ limit of the corresponding random matrix tuples $(\mathbf{X}^{(k)}, \mathbf{Y}^{(k)})$, and let $\mathbf{R}$ and $\mathbf{S}$ satisfy $\norm{x_i} < R_i$ and $\norm{y_i} < S_i$.

Let $\mathbf{A}^{(k)}$ be a deterministic tuple with $\norm{A_i^{(k)}} \leq R_i$ and $\lambda_{\mathbf{A}^{(k)}} \to \lambda_{\mathbf{x}}$.  Let $\mathbf{B}^{(k)}$ be a random matrix tuple chosen according to the conditional distribution \eqref{eq:conditionaldistribution} of $\mathbf{Y}^{(k)}$ given $\mathbf{X}^{(k)} = \mathbf{A}^{(k)}$.  For $i = 1,\ldots, n$, pick $\psi_i \in C_c^\infty(\R;\R)$ with $|\psi_i| \leq S_i$ and $\psi_i(t) = t$ for $|t| \leq \norm{y_i}$, and denote $\psi(\mathbf{B}^{(k)}) = (\psi_1(B_1^{(k)}),\dots,\psi_n(B_n^{(k)}))$.

Then the matrix models $(\mathbf{A}^{(k)},\psi(\mathbf{B}^{(k)}))$ satisfy the hypotheses of Theorem \ref{thm:main} with respect to $\mathrm{W}^*(\mathbf{x}) \subseteq \mathrm{W}^*(\mathbf{x},\mathbf{y})$, the generating set $(\mathbf{x},\mathbf{y})$, and the operator norm bounds $(\mathbf{R},\mathbf{S})$.
\end{prop}

\begin{proof}
Condition (1) of the theorem holds because we chose $A_i^{(k)}$ to be bounded in operator norm by $R_i$, and the function $\psi_i$ to be bounded by $S_i$.

To check the exponential concentration hypothesis (3), recall that $HV^{(k)} \geq c$.  It follows by restriction that also $H[V^{(k)}(\mathbf{A}^{(k)}, \cdot)] \geq c$.  Then using the standard machinery of the Bakry-Emery condition, the log-Sobolev inequality, and Herbst's argument, we obtain exponential concentration for the conditional distribution \eqref{eq:conditionaldistribution}.  For details, see \cite[\S 2.3 and \S 4.4]{AGZ2009}.

Next, we check (2) that $\lambda_{(\mathbf{A}^{(k)}, \psi(\mathbf{B}^{(k)}))}$ converges in probability to $\lambda_{\mathbf{x},\mathbf{y}}$.  We want to show that for every non-commutative polynomial $p$ in $m+n$ variables, we have $\tau_k[p(\mathbf{A}^{(k)},\psi(\mathbf{B}^{(k)}))] \to \tau[p(\mathbf{x},\mathbf{y})]$ in probability.  But due to the concentration of measure, it suffices to show that $\E \tau_k[p(\mathbf{A}^{(k)},\psi(\mathbf{B}^{(k)}))] \to \tau[p(\mathbf{x},\mathbf{y})]$.

Choose a smooth cut-off function $\phi_i$ with $\phi_i(t) = t$ for $|t| \leq R_i$ and let $\phi(\mathbf{A}) = (\phi_1(A_1),\dots,\phi_m(A_m))$ (analogous to the choice of $\psi$).  Define $f^{(k)}(\mathbf{A},\mathbf{B}) := \tau_k(p(\phi(\mathbf{A}), \psi(\mathbf{B})))$.  It follows from the discussion in \cite[\S 8.3]{Jekel2018} and \cite[Lemma 3.14]{Jekel2019} that $\phi$ and $\psi$ are $\norm{\cdot}_2$-Lipschitz and asymptotically approximable by trace polynomials, and hence the same is true for $f^{(k)}$ (since the images of $\phi$ and $\psi$ are contained in an operator norm ball, and $p$ is $\norm{\cdot}_2$-Lipschitz on the operator norm ball).  Thus, we can apply Theorem \ref{thm:conditionalexpectation} to $f^{(k)}$.  If $g^{(k)}$ and $g$ are as in the theorem, then
\[
g^{(k)}(\mathbf{A}^{(k)}) = \E[f^{(k)}(\mathbf{A}^{(k)},\psi(\mathbf{B}^{(k)}))].
\]
(using the fact that $\phi(\mathbf{A}^{(k)}) = \mathbf{A}^{(k)}$) and
\[
g(\mathbf{x}) = E_{\mathrm{W}^*(\mathbf{x})}[p(\mathbf{x},\mathbf{y})]
\]
(using the fact that $\phi(\mathbf{x}) = \mathbf{x}$ and $\psi(\mathbf{y}) = \mathbf{y}$).  We also know from the theorem that $\norm{g^{(k)}(\mathbf{A}^{(k)}) - g(\mathbf{A}^{(k)})}_2 \to 0$.  Therefore, using the convergence of $\lambda_{\mathbf{A}^{(k)}} \to \lambda_{\mathbf{x}}$,
\[
\E \circ \tau_k[f^{(k)}(\mathbf{A}^{(k)},\psi(\mathbf{B}^{(k)}))] = \tau_k[g^{(k)}(\mathbf{A}^{(k)})]^{1/2} \to \tau[g(\mathbf{x})]^{1/2} = \tau \circ E_{\mathrm{W}^*(\mathbf{x})}[p(\mathbf{x},\mathbf{y})] = \tau[p(\mathbf{x},\mathbf{y})].
\]
This completes the proof of (2).

The external averaging property (4) follows by the same token because with the notation as above, we have
\[
\norm{\E[f^{(k)}(\mathbf{A}^{(k)},\psi(\mathbf{B}^{(k)}))]}_2 = \tau_k[g^{(k)}(\mathbf{A}^{(k)})^* g^{(k)}(\mathbf{A}^{(k)})]^{1/2} \to \tau[g(\mathbf{x})^* g(\mathbf{x})]^{1/2} = \norm{E_{\mathrm{W}^*(\mathbf{x})}[p(\mathbf{x},\mathbf{y})]}_2. \qedhere
\]
\end{proof}

Therefore, by Theorem \ref{thm:main}, we have the following result.

\begin{cor}
Let $(\mathbf{x},\mathbf{y})$ be the tuple of non-commutative random variables given above.  Let $\cM = \mathrm{W}^*(\mathbf{x},\mathbf{y})$ and $\cP = \mathrm{W}^*(\mathbf{x})$.  Then if $\cN\leq \cM$,  $\cN \cap \cP$ is diffuse and $h(\cN:\cM) = 0$, then $\cN \subseteq \cP$.
\end{cor}

In particular, this provides another proof for the case when $(\mathbf{x},\mathbf{y})$ is a free semicircular family, that is, the case of the free group factor which we discussed in the previous subsection.  Perhaps disappointingly, the application of this corollary turns out to not be any wider than the free group factor case.  Indeed, it was shown in \cite[Theorem 8.11]{Jekel2019} (which is also included in \cite[Theorem 17.1.9]{JekelThesis}) that under the same set of hypotheses, there is an isomorphism $\phi: \mathrm{W}^*(\mathbf{x},\mathbf{y}) \to L(\F_{m+n})$ that maps $\mathrm{W}^*(\mathbf{x})$ to the canonical copy of $L(\F_m)$ inside $L(\F_{m+n})$.  Nonetheless, given that the result about conditional expectations was somewhat easier to prove in \cite{Jekel2019} than the isomorphism result, we suspect that the technique described in this section will have applications to other situations in random matrix theory where the isomorphism to the free group factor setting is not true (or at least not known).

For instance, does this result about conditional expectation extend to the case where $V^{(k)}$ is not semi-concave, or even not convex?  We conjecture that to establish the convergence of the non-commutative law in the large $k$ limit and the external averaging property (which is weaker than the conclusion of Theorem \ref{thm:conditionalexpectation}), it should be sufficient to assume that $DV^{(k)}$ is asymptotically approximable by trace polynomials and globally $\norm{\cdot}_2$-Lipschitz, and that the distribution of $\mathbf{X}^{(k)}$ and the conditional distribution of $\mathbf{Y}^{(k)}$ given $\mathbf{X}^{(k)}$ satisfy the log-Sobolev inequality with dimension-independent constants (after the appropriate normalization).  If this is true, then Proposition \ref{prop:convexpotentials} would generalize to this case, hence Theorem \ref{thm:main} could be applied to produce examples of subalgebras $\cP \leq \cM$ that absorb other subalgebras of $1$-bounded entropy zero.

Another question is whether Proposition \ref{prop:convexpotentials} can generalize to measures that are cut off to an operator norm ball.  Specifically, suppose that $V^{(k)}$ is defined on the operator norm ball $\{\mathbf{A}: \norm{A_j} \leq R\}$, that $V^{(k)}(\mathbf{A}) - (c/2) \norm{\mathbf{A}}_2^2$ is convex, and that $DV^{(k)}$ is asymptotically approximable by trace polynomials on that ball.  Let $\mu^{(k)}$ be the measure supported on $\{\mathbf{A}: \norm{A_j} \leq R\}$ given by the density proportional to $e^{-k^2 V^{(k)}(\mathbf{A})} \,d\mathbf{A}$.  Then if there is a limiting non-commutative law $\lambda_{\mathbf{x}}$ such that $\norm{x_i}$ is \emph{strictly less} than $R$, we conjecture that something similar to Theorem \ref{thm:conditionalexpectation} will hold for the conditional expectation, perhaps after restricting to a smaller operator-norm ball, and hence Proposition \ref{prop:convexpotentials} would generalize to this case, without the need for the cut-off function $\psi$.

\section{Proof of Theorem \ref{thm:main2}} \label{sec:Theorem2}

\subsection{Overview} \label{subsec:freeproductoverview}

To establish Theorem \ref{thm:main2} in the separable case, we consider the amalgamated free product $(\cM,\tau) = (\cM_1,\tau_1) *_{(\cD,\tau_{\cD})} (\cM_2,\tau_2)$, where $(\cM_1,\tau_1)$ and $(\cM_2,\tau_2)$ are tracial $\mathrm{W}^*$-algebras with separable predual that can embed into $\mathcal{R}^\omega$, and where $(\cD,\tau_{\cD})$ is an atomic $\mathrm{W}^*$-algebra which is a common tracial $\mathrm{W}^*$-subalgebra of both $\cM_1$ and $\cM_2$.

Choose self-adjoint tuples $\mathbf{x}_1 \in (\cM_1)_{sa}^{I_1}$ and $\mathbf{x}_2 \in (\cM_2)_{sa}^{I_2}$ which generate $\cM_1$ and $\cM_2$ respectively, where $I_1$ and $I_2$ are countable index sets.  We use the following construction of random matrix models for $\mathbf{x} = (\mathbf{x}_1,\mathbf{x}_2)$, due to Brown--Dykema--Jung \cite{BDJ2008}:
\begin{enumerate}
	\item We choose integers $n(k) \to \infty$ and unital $*$-homomorphisms $\pi^{(k)}: \cD \to M_{n(k)}(\C)$ which are asymptotically trace-preserving as $k \to \infty$.
	\item We choose a sequence of (deterministic) microstates $\mathbf{X}_\ell^{(k)}$ for $\mathbf{x}_\ell$ for each $\ell = 1, 2$ and arrange that they are compatible with our chosen maps $\pi^{(k)}: \cD \to M_{n(k)}(\C)$ (see Lemma \ref{lem:freeproductmicrostates}).
	\item Let $\pi^{(k)}(\cD)'$ be the commutant of $\pi^{(k)}(\cD)$ in $M_{n(k)}(\C)$.  Let $U^{(k)}$ be a Haar random unitary from $\pi^{(k)}(\cD)'$.
	\item Then our random matrix models for generating tuple $\mathbf{x} = (\mathbf{x}_1,\mathbf{x}_2)$ in the free product will be $\mathbf{X}^{(k)} = (\mathbf{X}_1^{(k)}, U^{(k)} \mathbf{X}_2^{(k)} (U^{(k)})^*)$.
\end{enumerate}
Adapting the arguments of \cite{BDJ2008}, we will verify that the random matrix models satisfy all the hypotheses of Theorem \ref{thm:main} for $\cM_1 \leq \cM$, that is, boundedness in operator norm, convergence of the non-commutative moments in probability, exponential concentration, and the external averaging property.  At the end of the discussion, we will describe how to deduce the general case of the theorem from the separable case.

\subsection{Construction of $*$-homomorphisms on $\cD$ and Compatible Microstates} \label{subsec:freeproductmicrostates}

Now we construct the $*$-homomorphisms $\pi^{(k)}: \cD \to M_{n(k)}$ and compatible microstates for $\mathbf{x}_1$ and $\mathbf{x}_2$.

Recall that $\cD$ being atomic means that it is a direct sum of type I factors.  Since it has a faithful normal trace, all these factors must be matrix algebras, and there must be countably many of them.  It is well known that in this case, there exist natural numbers $n(k)$ and unital $*$-homomorphisms $\pi^{(k)}: \cD \to M_{n(k)}(\C)$ such that $\tau_{n(k)} \circ \pi^{(k)} \to \tau_{\cD}$ pointwise.  We recall an explicit construction of $\pi^{(k)}$ here, because the decomposition of $\pi^{(k)}(\cD)$ and its commutant in $M_{n(k)}(\C)$ will be used in the next section to check exponential concentration for Haar unitaries in $\pi^{(k)}(\cD)'$.

We assumed $\mathcal{D}$ is atomic, that is, a direct sum of matrix algebras.  So there are natural numbers $\{r(j)\}_{j \in J}$, where $J = \{1,\dots,N\}$ or $J = \N$, and positive weights $\{\gamma(j)\}_{j \in J}$ with $\sum_{j \in J} \gamma(j) = 1$ such that
\[
(\mathcal{D}, \tau_{\mathcal{D}}) = \left( \overline{\bigoplus}_{j \in J} M_{r(j)}(\C), \overline{\bigoplus}_{j \in J} \gamma(j) \tau_{r(j)} \right).
\]
For each $k$, we define
\[
m(j,k) = \floor*{\frac{k \gamma(j)}{r(j)}}.
\]
Note that for each $k$, $m(j,k)$ is nonzero for only finitely many values of $j$, because in the case where $J$ is infinite, we have $\gamma(j) \to 0$ as $j \to \infty$.  We define
\[
n(k) = \sum_{j \in J} r(j) m(j,k),
\]
Then we have $r(j) m(j,k) / k \leq \gamma(j)$ and $r(j) m(j,k) / k \to \gamma(j)$ as $k \to \infty$.  It follows that $n(k) / k \leq 1$ and $n(k) / k \to 1$ as $k \to \infty$, and in particular $n(k) \to \infty$.

We define $\pi^{(k)}: \cD \to M_{n(k)}(\C)$ by
\[
\pi^{(k)} \left( \bigoplus_{j \in J} z_j \right) = \bigoplus_{\substack{j \in J \\ m(j,k) > 0}} z_j \otimes I_{m(j,k)},
\]
or more explicitly, it is the block diagonal matrix
\[
\pi^{(k)} \left( \bigoplus_{j \in J} z_j \right) = \diag \bigl( \underbrace{z_1, \dots, z_1}_{m(1,k)}, \underbrace{z_2, \dots, z_2}_{m(2,k)}, \underbrace{z_3, \dots, z_3}_{m(3,k)}, \dots \bigr),
\]
where the $r(j) \times r(j)$ block $z_j$ is repeated $m(j,k)$ times.  Since only finitely many $m(j,k)$'s are positive for a given $k$, there are finitely many terms, and overall the matrix has size $n(k) = \sum_{j \in J} r(j) m(j,k)$.

To check that $\tau_{n(k)} \circ \pi^{(k)} \to \tau_{\cD}$ pointwise, it suffices to verify this on dense subset of $\cD$.  But if $z = \bigoplus_{j \in J} z_j \in \cD$ with only finitely many $z_j$'s nonzero, then
\[
\tau_{n(k)} \circ \pi^{(k)} \left( \bigoplus_{j \in J} z_j \right) = \sum_{j \in J} \frac{r(j) m(j,k)}{n(k)} \tau_{r(j)}(z_j) \to \sum_{j \in J} \gamma(j) \tau_{r(j)}(z_j) = \tau_{\cD} \left( \bigoplus_{j \in J} z_j \right).
\]

Next, we have to create microstates for $\mathbf{x}_\ell$ compatible with the chosen maps $\pi^{(k)}$.

\begin{lem} \label{lem:freeproductmicrostates}
For each $\ell = 1, 2$, let $\mathbf{x}_\ell$ be an $I_\ell$-tuple of generators for $\cM_\ell$, which we assumed to be tracially embeddable into $\mathcal{R}^\omega$.  Let $\mathbf{R} \in (0,+\infty)^{I_1 \sqcup I_2}$ with $\norm{(x_\ell)_i} \leq R_i$ for $i \in I_\ell$.  Let $\mathbf{d} = (d_j)_{j \in J}$ be a tuple of operators whose span is dense in $\cD$.  Then there exists a (deterministic) tuple $\mathbf{X}_\ell^{(k)} \in M_{n(k)}(\C)_{sa}^{I_\ell}$ such that $(\mathbf{X}_\ell^{(k)}, \pi^{(k)}(\mathbf{d}))$ converges in non-commutative $*$-moments to $(\mathbf{x}_\ell, \mathbf{d})$.
\end{lem}

\begin{proof}
Fix $\ell$.  It will be convenient to use the set of generators $(\mathbf{x}_\ell, \mathbf{d})$ for $\cM_\ell$ rather than $\mathbf{x}_\ell$, and we use the tuple of operator norm bounds $\mathbf{T} \in (0,+\infty)^{I_\ell \sqcup J}$ given by $T_i = R_i$ for $i \in I_\ell$ and $T_j = 2 \norm{d_j}$ for $j \in J$.  We claim that there exist tuples $\mathbf{Z}^{(k)} \in M_{n(k)}(\C)_{sa}^{I_\ell \sqcup J}$ satisfying $\norm{Z_i^{(k)}} \leq T_i$ for all $i \in I_\ell \sqcup J$ and $\lambda_{\mathbf{Z}^{(k)}} \to \lambda_{(\mathbf{x}_\ell, \mathbf{d})}$.

We assumed that $\cM_\ell$ is embeddable into $\mathcal{R}^\omega$, and it is well-known that a separable tracial von Neumann algebra can be embedded into $\mathcal{R}^\omega$ if and only if it can be embedded into $\prod_{n \to \omega} M_n(\C)$.  Indeed, both conditions are equivalent to the non-commutative law of the generators, in this case $(\mathbf{x}_\ell, \mathbf{d})$, being in the closure in $\Sigma_{\mathbf{T}}$ of the set of non-commutative laws of matrix tuples.  Thus, there exist tuples $\mathbf{Y}^{(j)}$ of $\tilde{n}(j)\times \tilde{n}(j)$ matrices, %in $M_{\tilde{n}(j)}(\C)$,
$j \in \N$, such that $\norm{Y_i^{(j)}} \leq R_i$ and $\lambda_{\mathbf{Y}^{(j)}} \to \lambda_{(\mathbf{x}_\ell, \mathbf{d})}$ as $j \to \infty$.

But we can modify our sequence as follows to obtain another sequence of matrix tuples (indexed by $k$ rather than $j$) with the prescribed size $n(k) \times n(k)$.  Choose $j(k)$ such that $j(k) \to \infty$ and $\tilde{n}(j(k)) / n(k) \to 0$.  Then by integer division, we write $n(k) = q(k) \tilde{n}(j(k)) + r(k)$ with $0 \leq r(k) < j(k)$, and we set
\[
Z_i^{(k)} := (Y_i^{(k)})^{\oplus q(k)} \oplus 0_{r(k) \times r(k)} \in M_{n(k)}(\C)_{sa}.
\]
It is an exercise to verify that $\lambda_{\mathbf{Z}^{(k)}} \to \lambda_{(\mathbf{x}_{\ell}, \mathbf{d})}$.

In particular, the non-commutative law of $\mathbf{Z}^{(k)}|_J$ converges to $\lambda_{\mathbf{d}}$.  Meanwhile, the non-commutative law of $\pi^{(k)}(\mathbf{d})$ also converges to $\lambda_{\mathbf{d}}$, and the algebra $\cD = \mathrm{W}^*(\mathbf{d})$ is hyperfinite.  So by Lemma \ref{lem:hyper folklore}, there exists a sequence $V^{(k)}$ of unitaries such that
\[
\lim_{k \to \infty} \norm{ V^{(k)} Z_j^{(k)} (V^{(k)})^* - \pi^{(k)}(d_j) }_2 = 0 \qquad \text{ for } j \in J.
\]
Now let $\mathbf{X}_\ell^{(k)} = V^{(k)} \mathbf{Z}^{(k)}|_I (V^{(k)})^*$.  The non-commutative law of $V^{(k)} \mathbf{Z}^{(k)} (V^{(k)})^*$ converges to that of $(\mathbf{x}_\ell, \mathbf{d})$, and also each entry of the tuple $V^{(k)} \mathbf{Z}^{(k)} (V^{(k)})^* - (\mathbf{X}_\ell^{(k)}, \pi^{(k)}(\mathbf{d}))$ goes to zero in $\norm{\cdot}_2$ (the $I$-indexed entries are already zero).  Thus, by Corollary \ref{cor:perturbedconvergence}, the non-commutative law of $(\mathbf{X}_\ell^{(k)}, \pi^{(k)}(\mathbf{d}))$ converges to $\lambda_{(\mathbf{x}_\ell,\mathbf{d})}$.
\end{proof}

\subsection{Exponential Concentration} \label{subsec:freeproductconcentration}

We want to show that the Haar random unitary from $\pi^{(k)}(\cD)'$ exhibits exponential concentration as $k \to \infty$ (Lemma \ref{lem:commutantconcentration} below).  The finite-dimensional algebra $\pi^{(k)}(\cD)$ is a direct sum of matrix algebras, and hence its unitary group is a direct product of unitary groups.  Our proof of concentration relies on the log-Sobolev inequality, its behavior under products, and the log-Sobolev constants of the unitary groups.

To state the log-Sobolev inequality, we treat the unitary groups as Riemannian manifolds.  Since a unitary matrix can be written as $e^{iA}$ for $A$ self-adjoint, the tangent space at each point of $U(M_n(\C))$ can be identified with $M_n(\C)_{sa}$.  The \emph{standard Riemannian metric} on $U(M_n(\C))$ is the one defined by using the inner product $\Tr(A^*B)$ for $A, B \/in M_n(\C)_{sa}$ with the standard trace (so $\Tr(I) = n$).

\begin{defn}
A probability measure $\mu$ on a Riemannian manifold $M$ is said to satisfy the \emph{log-Sobolev inequality with constant $c>0$} if for all $f \in C_c^\infty(M,\R)$, we have
\[
\int_M f(x)^2 \log \frac{f(x)^2}{\int f^2\,d\mu} \,d\mu(x) \leq 2c \int_M \norm{Df(x)}^2\,d\mu(x),
\]
where $Df$ and $\norm{Df}$ denote the gradient of $f$ and its norm with respect to the Riemannian metric.
\end{defn}

\begin{thm}[{see \cite[Theorem 15]{Meckes2013}}]
The unitary group $U(M_n(\C))$ with the standard Riemannian metric satisfies the log-Sobolev inequality with constant $c = 6/n$.
\end{thm}

\begin{remark} \label{rem:unitarygeodesic}
The geodesic distance function $d(U,V)$ on the unitary group that arises from this Riemannian metric is not the same as the Hilbert-Schmidt distance $d_{\text{HS}}(U,V) = \Tr((U - V)^*(U - V))^{1/2}$.  However, we have
\[
d_{\text{HS}}(U,V) \leq d(U,V) \leq \frac{\pi}{2} d_{\text{HS}}(U,V).
\]
See for instance \cite[Lemma 3.9.1]{Blower2009}.  We will continue to work with the geodesic distance.
\end{remark}

We need the following facts about the log-Sobolev inequality and concentration.

\begin{obs} \label{obs:logSobolevscaling}
Let $M$ be a Riemannian manifold and suppose $\mu$ is a probability measure satisfying the log-Sobolev inequality with constant $c$.  Fix $t > 0$, and let $\tilde{M}$ be the same manifold with the rescaled Riemannian metric $\ip{x,y}_{\tilde{M}} = t \ip{x,y}_M$ where $x$ and $y$ are tangent vectors at a point $p \in M$.  Then $\mu$ satisfies the log-Sobolev inequality with constant $ct$.
\end{obs}

\begin{lem}[{see \cite[Corollary 5.7]{Ledoux2001}}] \label{lem:logSobolevproduct}
Let $M_1$, \dots, $M_n$ be Riemannian manifolds, and let $\mu_j$ be a probability measure on $M_j$ satisfying the log-Sobolev inequality with constant $c_j$.  Let $M = M_1 \times \dots \times M_n$ with product Riemannian metric.  Then the product measure $\mu = \mu_1 \otimes \dots \otimes \mu_n$ on $M$ satisfies the log-Sobolev inequality with constant $c = \max(c_1,\dots,c_n)$.
\end{lem}

\begin{lem}[{see \cite[Corollary 5.4]{Ledoux2001}}] \label{lem:logSobolevconcentration}
Suppose $\mu$ is a probability measure on the Riemannian manifold $M$.  Let $d$ be the geodesic distance on $M$, that is the metric obtained from infimizing the lengths of paths (measured using the Riemannian metric), and let $N_\eps(\Omega)$ denote the $\eps$-neighborhood of a set.  Let $\alpha_\mu$ be the metric concentration function defined by
\[
\alpha_\mu(\eps) = \sup \{ \mu(N_\eps(\Omega)^c): \Omega \text{ \rm Borel with } \mu(\Omega) \geq 1/2\}.
\]
If $\mu$ satisfies the log-Sobolev inequality with constant $c$, then
\[
\alpha_\mu(\eps) \leq e^{-\eps^2 / 8c}.
\]
\end{lem}

\begin{proof}[Sketch of proof]
First, the argument of Herbst (see \cite[Theorem 5.3]{Ledoux2001} or \cite[Lemma 2.3.3]{AGZ2009}) shows that for the Lipschitz $f: M \to \R$ and $\lambda \geq 0$,
\[
A_\lambda := \int_M e^{\lambda (f(x) - \int f\,d\mu)}\,d\mu(x) \leq e^{c \lambda^2 \norm{f}_{\Lip}^2 / 2}.
\]
Second, one shows that for $\delta > 0$,
\[
\mu( \{x: f(x) - \smallint f\,d\mu \geq \delta \} ) \leq e^{-\delta^2 / 2c \norm{f}_{\Lip}^2}.
\]
This follows from the previous estimate by substituting $\lambda = \delta / c \norm{f}_{\Lip}^2$ and using Markov's inequality.

Finally, to estimate the concentration function $\alpha_\mu$, fix a Borel set $\Omega$ with $\mu(\Omega) \geq 1/2$.  Define
\[
f(x) = \min\left( \frac{d(x,\Omega)}{\eps},1 \right).
\]
Since $f \leq 1$ and $\mu(\Omega) \geq 1/2$, we have $\int f\,d\mu \leq 1/2$.  Also, $f$ is $1/\eps$-Lipschitz.  Thus,
\[
\mu(N_\eps(\Omega)^c) = \mu( \{x: f(x) =
1\}) \leq
\mu( \{x: f(x) - \smallint f\,d\mu \geq 1/2\} )
\leq e^{-(1/2)^2 / 2c \norm{f}_{\Lip}^2}
\leq e^{-\eps^2 / 8c}. \qedhere
\]
\end{proof}

\begin{lem} \label{lem:commutantconcentration}
Let $G^{(k)}$ be the unitary group of $\pi^{(k)}(\cD)'$, equipped with the Riemannian metric associated to the normalized trace $\tau_{n(k)}$ (in other word, identify the tangent space at a point with the self-adjoints of $\pi^{(k)}(\cD)'$ and use the inner product associated to the normalized trace from $M_{n(k)}(\C)$).  Let $\nu^{(k)}$ be the Haar measure on $G^{(k)}$, and let $\alpha_{\nu^{(k)}}$ denote the concentration function.  Then for every $\eps > 0$,
\[
\limsup_{k \to \infty} \frac{1}{n(k)^2} \log \alpha_{\nu^{(k)}}(\eps) < 0.
\]
\end{lem}

\begin{proof}
Let us first compute the unitary group of $\pi^{(k)}(\cD)'$ more explicitly.  The image of $\pi^{(k)}(\cD)$ can be expressed as
\[
\pi^{(k)}(\cD) = \bigoplus_{j \in J} M_{r(j)} \otimes I_{m(j,k)} \subseteq \bigoplus_{j \in J} M_{r(j)}(\C) \otimes M_{m(j,k)}(\C) \cong \bigoplus_{j \in J} M_{r(j) m(j,k)}(\C) \subseteq M_{n(k)}(\C),
\]
where the inclusions and identifications at each step are the standard ones.  It is well known (and easy to verify) that the commutant of this subalgebra is
\[
\pi^{(k)}(\cD)' = \bigoplus_{j \in J} I_{r(j)} \otimes M_{m(j,k)}(\C) \subseteq \bigoplus_{j \in J} M_{r(j) m(j,k)}(\C) \subseteq M_{n(k)}(\C).
\]
Hence,
\[
G^{(k)} \cong \prod_{\substack{j \in J \\ m(j,k) > 0}} U(M_{m(j,k)}(\C)),
\]
and of course the Haar measure on $G^{(k)}$ is the product of the Haar measures on the individual factors.  As in the case of the unitary groups, we identify the tangent space with the self-adjoints of $\pi^{(k)}(\cD)'$, or the direct sum of $M_{m(j,k)}(\C)_{sa}$ over $j$ with $m(j,k) > 0$.  Then the Riemannian metric is given by
\[
\ip*{\bigoplus_j A_j, \bigoplus_j B_j} =  \frac{1}{n(k)} \sum_{\substack{j \in J \\ m(j,k) > 0}} r(j) \Tr(A_j B_j).
\]
As in Remark \ref{rem:unitarygeodesic}, the corresponding geodesic distance on $G^{(k)}$ is bounded above and below by constants times the normalized Hilbert-Schmidt distance
\[
\norm{(U_j - V_j)_j}_2 = d_{\text{HS}}((U_j)_j, (V_j)_j) = \left( \frac{1}{n(k)} \sum_{\substack{j \in J \\ m(j,k) > 0}} r(j) \Tr((U_j - V_j)^*(U_j - V_j)), \right)^{1/2}.
\]

Since the number of direct summands is unbounded as $k \to \infty$ if $J$ is infinite, we will not study the log-Sobolev inequality on $G^{(k)}$ itself, but rather we will truncate to a fixed number of summands.  Fix $\eps > 0$.  Because the weights $\gamma(j)$ in our direct sum decomposition sum to $1$, there exists a finite $N \in \N$ such that
\[
\sum_{j=1}^N \gamma(j) > 1 - \frac{\eps^2}{4} \left( \frac{\pi}{2} \right)^2.
\]
Moreover, since $m(j,k) r(j) / n(k) \to \gamma(j)$ as $k \to \infty$, we know that for sufficiently large $k$,
\[
\sum_{j=1}^N \frac{m(j,k) r(j)}{n(k)} > 1 - \frac{\eps^2}{4} \left( \frac{\pi}{2} \right)^2.
\]
It follows that for $U = (U_j)_j$ and $V = (V_j)_j$ in $G^{(k)}$, if $U_j = V_j$ for $j \leq N$, then $d(U,V) \leq (2 / \pi) \norm{U - V}_2 < \eps / 2$.  Define
\[
G_{\leq N}^{(k)} = \prod_{j=1}^N U(M_{m(j,k)}(\C)), \qquad G_{>N}^{(k)} = \prod_{j > N} U(M_{m(j,k)}(\C))
\]
with the associated Riemannian metrics given by the weights $m(j,k) r(j) / n(k)$, and denote the associated Haar measures by $\nu_{\leq N}^{(k)}$ and $\nu_{>N}^{(k)}$.

Now, since $U(M_{m(j,k)}(\C))$ with the satisfies log-Sobolev with constant $6/m(j,k)$ with the standard Riemannian metric, if we rescale the metric by $r(j) / n(k)$, then the log-Sobolev inequality holds with constant $6 r(j) / n(k) m(j,k)$.  Then the direct product $G_{\leq N}^{(k)}$ satisfies the log-Sobolev inequality with constant
\[
c_N^{(k)} := \max \left\{ \frac{6 r(j)}{n(k) m(j,k)}: j = 1,\dots, N \right\}.
\]
Fix a Borel set $\Omega \subseteq G^{(k)}$ with $\nu^{(k)}(\Omega) \geq 1/2$.  Let $\Omega'$ be the projection of $\Omega$ onto $G_{\leq N}^{(k)}$.  By our choice of $N$, we have
\[
N_{\eps/2}(\Omega) \supseteq \Omega' \times G_{>N}^{(k)}.
\]
Here $\Omega'$ is a subset of $G_{\leq N}^{(k)}$ with Haar measure at least $1/2$.

On the other hand, by applying Lemma \ref{lem:logSobolevconcentration} on $G_{\leq N}^{(k)}$, we have
\[
\nu^{(k)}(N_\eps(\Omega)^c) \leq \nu_{\leq N}^{(k)}(N_{\eps/2}(\Omega')^c) \leq \alpha_{\nu_{\leq N}}^{(k)}(\eps/2) \leq e^{-\eps^2 / 32 c_N^{(k)}}.
\]
Hence,
\[
\frac{1}{n(k)^2} \log \alpha_{\nu^{(k)}}(\eps) \leq -\frac{\eps^2}{32 n(k)^2 c_N^{(k)}}.
\]
But
\[
n(k)^2 c_N^{(k)} = \max \left\{ \frac{6 r(j) n(k)}{m(j,k)}: j = 1,\dots, N \right\},
\]
and
\[
\frac{6 r(j) n(k)}{m(j,k)} = 6 r(j)^2 \frac{n(k)}{r(j) m(j,k)} \longrightarrow \frac{6r(j)^2}{\gamma(j)} < +\infty,
\]
so that $\limsup_{k \to \infty} n(k)^2 c_N^{(k)} < +\infty$, which implies that $\limsup_{k \to \infty} n(k)^{-2} \log \alpha_{\nu^{(k)}}(\eps) < 0$.
\end{proof}

\begin{cor} \label{cor:freeproductconcentration}
Let $\mathbf{X}_\ell^{(k)}$ be the microstates for $\mathbf{x}_\ell$, where $\ell = 1,2$, chosen in \S \ref{subsec:freeproductmicrostates}, and let $U^{(k)}$ be a Haar random unitary from $\pi^{(k)}(\cD)'$.  Then the random matrix tuple $\mathbf{X}^{(k)} = (\mathbf{X}_1^{(k)}, U^{(k)} \mathbf{X}_2^{(k)} (U^{(k)})^*)$ satisfies the exponential concentration property of Definition \ref{defn:concentration}.
\end{cor}

\begin{proof}
For $i \in I_2$, we have $f_i^{(k)}: U^{(k)} \mapsto U^{(k)} (X_2^{(k)})_i (U^{(k)}))^*$ is $2R_i$-Lipschitz with respect to $\norm{\cdot}_2$ because $\norm{(X_2^{(k)})_i} \leq R_i$ for $i \in I_2$.  But the normalized Hilbert-Schmidt norm is bounded by the geodesic distance on $G^{(k)}$, hence the function is also $2R_i$-Lipschitz if we use the geodesic distance in the domain.  Of course, for $i \in I_1$, the constant function $f_i^{(k)}: U^{(k)} \mapsto (X_1^{(k)})_i$ is trivially Lipschitz.  Now, $U^{(k)}$ has exponential concentration by the previous lemma. Thus, using the same argument as in Corollary \ref{cor:concentrationpushforward}, the exponential concentration is preserved when we push forward by the tuple of functions $(f_i^{(k)})_{i \in I_1 \sqcup I_2}$, since each $f_i^{(k)}$ is $\norm{\cdot}_2$-uniformly continuous with estimates independent of $k$.
\end{proof}

\subsection{Asymptotic Freeness with Amalgamation}

The key point in establishing the convergence in moments and external averaging property (hypotheses (2) and (4) of Theorem \ref{thm:main}) will be the following theorem. In the case where $\cD$ is finite-dimensional, this is a restatement of \cite[Theorem 3.9]{BDJ2008}, and we merely extend it to the atomic case by an approximation argument.  Note that this theorem considers countably many algebras $\cM_\ell$ rather than only two algebras $\cM_1$ and $\cM_2$ as in \S \ref{subsec:freeproductoverview}.

\begin{thm} \label{thm:asymptoticfreeness}
Let $\{\cM_\ell, \tau_{\cM_\ell})\}_{\ell \in \N}$, $(\cM_2, \tau_{\cM_2})$, \dots be tracial von Neumann algebras that contain common atomic subalgebra $(\cD, \tau_{\cD})$ in a trace-preserving way.  Let $\cM$ be the free product of $\{\cM_\ell\}_{\ell \in \N}$ with amalgamation over $\cD$, and let us view $\cM_\ell$ as a subalgebra of $\cM$ in the canonical way.

For each $\ell$, let $\mathbf{x}_\ell \in (\cM_\ell)_{sa}^{I_\ell}$ be a generating tuple for $\cM_\ell$ with $I_\ell$ an index set.  Let $I = \bigsqcup_{\ell \in \N} I_\ell$ and let $\mathbf{x}$ be the $I$-tuple $(\mathbf{x}_\ell)_{\ell \in \N}$.  Suppose $\mathbf{R} \in (0,+\infty)^I$ with $\norm{x_i} \leq R_i$ for $i \in I_\ell$.

Let $n(k) \to \infty$.  Let $\pi^{(k)}: \cD \to M_{n(k)}(\C)$ be a $*$-homomorphism satisfying $\tau_{n(k)} \circ \pi^{(k)} \to \tau_{\cD}$.  Let $\mathbf{d} \in \cD^J$ be a tuple of operators whose span is dense in $\cD$, obtained as the union of spanning sets for each of the finite-dimensional matrix algebras that are direct summands of $\cD$.  For each $\ell$, let $\mathbf{X}_\ell^{(k)} \in M_{n(k)}(\C)_{sa}^{I_\ell}$ with $\norm{(X_\ell)_i} \leq R_i$ and such that $(\mathbf{X}_\ell^{(k)}, \pi^{(k)}(\mathbf{d})) \to (\mathbf{x}_\ell,\mathbf{d})$ in non-commutative $*$-moments.  Let $\{U_\ell^{(k)}\}_{\ell \in \N}$ be a family of independent Haar unitaries from $\pi^{(k)}(\cD)'$.

Let $\mathbf{X}^{(k)}$ be the $I$-tuple $(U_\ell^{(k)} \mathbf{X}_\ell^{(k)} (U_\ell^{(k)})^*)_{\ell \in \N}$.  Then $\lambda_{\mathbf{X}^{(k)}} \to \lambda_{\mathbf{x}}$ in $\Sigma_{\mathbf{R}}$ in probability as $k \to \infty$.
\end{thm}

\begin{proof}
The case where $\cD$ is finite-dimensional and the index sets $I_\ell$ are finite follows from \cite[Theorem 3.9]{BDJ2008}.  %We remark that Brown-Dykema-Jung use the notation $E_k$ for what we would call $(\pi^{(k)})^{-1} \circ E^{(k)} \circ \E$ (if $\cD$ is finite-dimensional, then $\pi^{(k)}$ is invertible for sufficiently large $k$).
The extension to infinite index sets $I_\ell$ is automatic.  Indeed, for $\lambda_{\mathbf{X}^{(k)}}$ to converge to $\lambda_{\mathbf{x}}$ in probability means that for every $f \in \C\ip{t_i: i \in \bigsqcup_\ell I_\ell}$, we have $\tau_{n(k)}(f(\mathbf{X}^{(k)})) \to \tau(f(\mathbf{x}))$.  But $f$ only depends on finitely many of the $(\mathbf{x}_\ell)_i$'s so the convergence follows from the case for finite index sets.

Now we must extend from finite-dimensional $\cD$ to atomic $\cD$.  As in the proof of Lemma \ref{lem:commutantconcentration}, for general atomic $\cD$, we must approximate by truncating to finitely many direct summands.  Given $\eps > 0$, choose $N$ large enough that $\sum_{j=1}^N \gamma(j) > 1 - \eps^2$.  Let $p \in \cD$ be the (central) projection onto the first $N$ direct summands.  Let $P^{(k)} = \pi^{(k)}(p)$ be the $k$th matrix approximation of $p$.  Then $\tau_{\cD}(p) > 1 - \eps^2$ and hence $\tau_{n(k)}(P^{(k)}) \geq 1 - \eps^2$ for sufficiently large $k$.  It follows that for every matrix $C \in M_{n(k)}(\C)$, we have
\begin{equation} \label{eq:projectionerror}
\norm*{C - P^{(k)} C P^{(k)}}_2 \leq \norm*{(1 - P^{(k)}) C}_2 + \norm*{P^{(k)}C(1 - P^{(k)})}_2 \leq 2 \eps \norm{C}.
\end{equation}

To show that $\tau_{n(k)}(f(\mathbf{X}^{(k)})) \to \tau(f(\mathbf{x}))$ in probability for all $f \in \C\ip{t_i: i \in \bigsqcup_\ell I_\ell}$, it suffices to consider the case where $f$ is a monomial, say
\[
f(t_i: i \in \bigsqcup_\ell I_\ell) = \prod_{j=1}^J t_{i(j)},
\]
where the terms in the product are indexed from left to right.  Let $\ell(j)$ be the index with $i(j) \in I_{\ell(j)}$, so that
\[
f(\mathbf{X}^{(k)}) = \prod_{j=1}^J U_{\ell(j)}^{(k)} X_{i(j)}^{(k)} (U_{\ell(j)}^{(k)})^*,
\]
where $X_{i(j)}^{(k)}$ is the $i(j)$-indexed element of the tuple $\mathbf{X}_{\ell(j)}^{(k)}$.  Next, we replace each term in the product with its truncation by $P^{(k)}$ and estimate the error for each ``swap'' by \eqref{eq:projectionerror} to obtain
\begin{equation} \label{eq:truncationerror1}
\norm*{f(\mathbf{X}^{(k)}) - \prod_{j=1}^k P^{(k)} U_{\ell(j)}^{(k)} P^{(k)} X_{i(j)} P^{(k)} (U_{\ell(j)}^{(k)})^* P^{(k)} }_2 \leq (2 \eps)(3J) \prod_{j=1}^J \norm{X_{i(j)}^{(k)}}
\leq 6 \eps J \prod_{j=1}^J R_{i(j)}^{(k)},
\end{equation}
since the total number of swaps is $3J$ and $\norm{X_{i(j)}^{(k)}} \leq R_{i(j)}$.  Analogously,
\begin{equation} \label{eq:truncationerror2}
\norm*{f(\mathbf{x}) - \prod_{j=1}^J u_{\ell(j)} \mathbf{x}_{i(j)} u_{\ell(j)}^* }_2 \leq 6 \eps J \prod_{j=1}^J R_{i(j)}^{(k)}.
\end{equation}

Next, we will evaluate the limit of the trace of the truncated version of $f(\mathbf{X}^{(k)})$ by applying the finite-dimensional case of the theorem to the algebras $p \cM_\ell p$ and the finite-dimensional subalgebra $p \cD p$.  We consider each of the compressions $p \cM_\ell p$, $p \cD p$, and $p \cM p$ as tracial von Neumann algebras by renormalizing the trace from the ambient algebra by $1 / \tau(p)$, as is standard.  Note that $\{p \cM_\ell p\}_{\ell \in \N}$ are freely independent in $p \cM p$ with amalgamation over $p \cD p$; this is an exercise to check from the definition of free independence and the fact that $p \in \cD$.

Similarly, we consider the truncated matrix algebra $P^{(k)} M_{n(k)}(\C) P^{(k)}$ with the trace renormalized by $1 / \tau_{n(k)}(P^{(k)})$.  Note that $\{P^{(k)} U_\ell^{(k)} P^{(k)}\}_{\ell \in \N}$ are independent Haar unitaries from $\pi^{(k)}(p\cD p)' \cap P^{(k)} M_{n(k)}(\C) P^{(k)}$.  The matrix tuples $(P^{(k)} \mathbf{X}_\ell^{(k)} P^{(k)}, P^{(k)} \pi^{(k)}(\mathbf{d}) P^{(k)})$ converge in non-commutative $*$-moments to $(\mathbf{x}_\ell, p \mathbf{d} p)$ with respect to the renormalized traces on the truncated algebras (of course, since $\tau_{n(k)}(P^{(k)}) \to \tau(p)$, it would be equivalent to show this convergence with respect to the traces on the ambient algebras).  But we chose $\mathbf{d}$ to contain a spanning set for each of the direct summands of $\cD$, and hence $p \in \Span(\mathbf{d})$.  Therefore, any $*$-polynomial $f$ in $(p\mathbf{x}_\ell p, p \mathbf{d} p)$ can be expressed as a $*$-polynomial $g$ in $(\mathbf{x}_\ell, \mathbf{d})$ such that we also have $f(P^{(k)} \mathbf{X}_\ell^{(k)} P^{(k)}, P^{(k)} \pi^{(k)} (\mathbf{d}) P^{(k)}) = g(\mathbf{X}_\ell^{(k)}, \pi^{(k)}(\mathbf{d}))$ (this uses the fact that $\pi^{(k)}$ is a $*$-homomorphism).

All these observations mean that the case of the theorem where $\cD$ is finite-dimensional can be applied to the truncated version of $f(\mathbf{X}^{(k)})$, and so we obtain that in probability,
\[
\lim_{k \to \infty} \frac{1}{\tau_{n(k)}(P^{(k)})} \tau_{n(k)} \left[ \prod_{j=1}^k P^{(k)} U_{\ell(j)}^{(k)} P^{(k)} X_{i(j)} P^{(k)} (U_{\ell(j)}^{(k)})^* P^{(k)} \right] = \frac{1}{\tau(p)} \tau\left[ \prod_{j=1}^J pu_{\ell(j)} x_{i(j)} u_{\ell(j)}^* \right].
\]
Of course, since $\tau_{n(k)}(P^{(k)}) \to \tau(p) > 0$, we can remove the terms $1 / \tau_{n(k)}(P^{(k)})$ and $1 / \tau(p)$ from the equation.  We combine this with \eqref{eq:truncationerror1} and \eqref{eq:truncationerror2} to conclude that in probability,
\[
\limsup_{k \to \infty} \left| \tau_{n(k)}[f(\mathbf{X}^{(k)})] - \tau[f(\mathbf{x})] \right| \leq 12 \eps J \prod_{j=1}^J R_{i(j)}^{(k)}.
\]
Since $\eps$ was arbitrary, we are finished.
\end{proof}

\begin{remark} \label{rem:finitelymanyfree}
Clearly, the same theorem is true if we only have finitely many algebras $\cM_\ell$ rather than countably many.
\end{remark}

\begin{remark} \label{rem:removefirstunitary}
The conclusion of the theorem still holds if we replace the first unitary $U_1^{(k)}$ by the identity.  This is because the non-commutative law is invariant under unitary conjugation, so we can conjugate $\mathbf{X}^{(k)}$ by $(U_1^{(k)})^*$ and note that $\{ (U_1^{(k)})^* U_\ell^{(k)}\}_{\ell \geq 2}$ has the same probability distribution as $\{U_\ell^{(k)}\}_{\ell \geq 2}$.
\end{remark}

\subsection{Conclusion to the Proof}

\begin{proof}[Proof of Theorem \ref{thm:main2}] ~

{\bf The separable case:} First, assume that $\cM_1$ and $\cM_2$ are have separable predual.  We retain all the setup and notation from \S \ref{subsec:freeproductoverview} and \S \ref{subsec:freeproductmicrostates}.  We want to show that random matrix tuples $\mathbf{X}^{(k)} = (\mathbf{X}_1^{(k)}, U^{(k)} \mathbf{X}_2^{(k)} (U^{(k)})^*)$ satisfy the hypotheses of Theorem \ref{thm:main} with respect to $\cM_1 \subseteq \cM = \cM_1 *_{\cD} \cM_2$ and the chosen generators $\mathbf{x} = (\mathbf{x}_1, \mathbf{x}_2)$ and operator norm bounds $\mathbf{R}$.  We check each of the four hypotheses of Theorem \ref{thm:main} in turn.

(1) By construction $\norm{(X_1)_i^{(k)}} \leq R_i$ for $i \in I_1$ and $\norm{U^{(k)} (X_2)_i^{(k)} (U^{(k)})^*} = \norm{(X_2)_i^{(k)}} \leq R_i$ for $i \in I_2$.

(2) The fact that $\lambda_{\mathbf{X}^{(k)}} \to \lambda_{\mathbf{x}}$ follows from Theorem \ref{thm:asymptoticfreeness} and the two remarks given above.  In other words, we apply the version of the theorem for two algebras with the first unitary replaced by the identity.

(3) Corollary \ref{cor:freeproductconcentration} showed that these random matrix models have exponential concentration.

(4) To check the external averaging property, let $f \in \C\ip{t_i: i \in I_1 \sqcup I_2}$, and we will show that
\[
\lim_{k \to \infty} \norm{\E[f(\mathbf{X}^{(k)})]}_2 = \norm{E_{\cM_1}[f(\mathbf{x})]}_2.
\]
We write
\begin{align*}
\norm{\E[f(\mathbf{X}^{(k)})]}_2^2 &= \tau_{n(k)}\left[\E[f(\mathbf{X}_1^{(k)}, U^{(k)} \mathbf{X}_2^{(k)} (U^{(k)})^*)]^* \E[f(\mathbf{X}_1^{(k)}, V^{(k)} \mathbf{X}_2^{(k)} (V^{(k)})^*)] \right] \\
&= \E \circ \tau_{n(k)} \left[f(\mathbf{X}_1^{(k)}, U^{(k)} \mathbf{X}_2^{(k)} (U^{(k)})^*)^* f(\mathbf{X}_1^{(k)}, V^{(k)} \mathbf{X}_2^{(k)} (V^{(k)})^*) \right]
\end{align*}
where $V^{(k)}$ is an independent Haar unitary from $\pi^{(k)}(\cD)'$.

Let $\cM_3$ be an isomorphic copy of $\cM_2$ and let $\mathbf{x}_3$ be the tuple corresponding to $\mathbf{x}_2$.  Let $\cQ = \cM_1 * \cM_2 * \cM_3$.  Applying Theorem \ref{thm:asymptoticfreeness} to the free product of $\cM_1$, $\cM_2$, and $\cM_3$ (again using the two remarks following the theorem), we obtain
\[
\lim_{k \to \infty} \E \circ \tau_{n(k)} \left[f(\mathbf{X}_1^{(k)}, U^{(k)} \mathbf{X}_2^{(k)} (U^{(k)})^*)^* f(\mathbf{X}_1^{(k)}, V^{(k)} \mathbf{X}_2^{(k)} (V^{(k)})^*) \right] = \tau_{\cQ}[f(\mathbf{x}_1, \mathbf{x}_2)^* f(\mathbf{x}_1, \mathbf{x}_3)].
\]
There is a canonical isomorphism
\[
\cQ \cong (\cM_1 *_{\cD} \cM_2) *_{\cM_1} (\cM_1 *_{\cD} \cM_3),
\]
or in other words, $\cM_1 *_{\cD} \cM_2$ and $\cM_1 *_{\cD} \cM_3$ are freely independent in $\cQ$ with amalgamation over $\cM_1$ (see \cite[Proposition 4.1]{Hou07}).  This implies that
\[
E_{\cM_1}[f(\mathbf{x}_1, \mathbf{x}_2)^* f(\mathbf{x}_1, \mathbf{x}_3)] = E_{\cM_1}[f(\mathbf{x}_1, \mathbf{x}_2)]^* E_{\cM_1}[f(\mathbf{x}_1, \mathbf{x}_3)].
\]
Of course, using the obvious isomorphism $\cM_1 * \cM_2 \cong \cM_1 * \cM_3$, we have $E_{\cM_1}[f(\mathbf{x}_1,\mathbf{x}_3)] = E_{\cM_1}[f(\mathbf{x}_1,\mathbf{x}_2)]$.  Thus, we have
\[
\tau_{\cQ}[f(\mathbf{x}_1, \mathbf{x}_2)^* f(\mathbf{x}_1, \mathbf{x}_3)] = \tau[E_{\cM_1}[f(\mathbf{x}_1, \mathbf{x}_2)]^* E_{\cM_1}[f(\mathbf{x}_1, \mathbf{x}_2)]],
\]
and hence
\[
\lim_{k \to \infty} \norm{\E[f(\mathbf{X}_1^{(k)}, U^{(k)} \mathbf{X}_2^{(k)} (U^{(k)})^*)]}_2^2 = \norm{E_{\cM_1}[f(\mathbf{x}_1,\mathbf{x}_2)]}_2^2.
\]

Therefore, the hypotheses of Theorem \ref{thm:main} are satisfied, and so Theorem \ref{thm:main} implies Theorem \ref{thm:main2} in the separable case.

{\bf The general case:} Now suppose that $\cM = \cM_1 *_{\cD} \cM_2$ where $\cD$ is atomic and $\cM_1$ and $\cM_2$ are not necessarily separable.  Suppose that $\cN \cap \cM_1$ is diffuse and $h(\cN: \cM) = 0$.  It suffices to show that if $\cN^0$ is a diffuse separable subalgebra of $\cN$, then $\cN^0 \subseteq \cM_1$.

Note that $h(\cN^0: \cM) \leq h(\cN: \cM) = 0$ by Property 2 of $h$ from \S \ref{subsec: 1 bounded entropy properties}.  Letting $\mathbf{x} \in (\cN^0)_{sa}^I$ be a generating tuple for $\cN^0$ and $\mathbf{y}$ be a generating tuple for $\cM$, this means that for every finite $F \subseteq I$ and $\eps > 0$ and $\delta > 0$, there exists a neighborhood $\mathcal{U}$ of $\lambda_{\mathbf{x},\mathbf{y}}$ such that
\[
\limsup_{n \to \infty} \frac{1}{n^2} \log K_{F,\eps}(\Gamma_{\mathbf{R},n}(\mathbf{x}:\mathbf{y}|C^{(n)} \rightsquigarrow z; \mathcal{U})) < \delta.
\]
Recall $\inf_{\cU}$ of the above expression will increase if $\eps$ becomes smaller and $F$ becomes larger.  Thus, by iterating over countably many values of $F$, $\eps$, and $\delta$, we see that only countably many neighborhoods $\cU$ are needed to witness that the $\sup_{F, \eps} \inf_{\cU}$ of the above expression is zero.  And each neighborhood only specifies conditions on finitely many generators of $\cM$ Thus, there exists a separable subalgebra $\cM^0$ such that $h(\cN^0: \cM^0) = 0$.

We can choose separable subalgebras $\cM_1^0 \leq \cM_1$ and $\cM_2^0 \leq \cM_2$ such that $\cM^0 \leq \cM_1^0 *_{\cD} \cM_2^0$, and hence $h(\cN^0: \cM_1^0 *_{\cD} \cM_2^0) = 0$.  By making $\cM_1^0$ larger if necessary, we can arrange that $\cN^0 \cap \cM_1^0$ is diffuse. Then it follows from the separable case of the theorem that $\cN^0 \subseteq \cM_1^0 \subseteq \cM_1$ as desired.
\end{proof}

\begin{remark}
While much of the proof could be adapted to the case of amenable $\cD$, Theorem \ref{thm:main2} does not hold for $\cD$ amenable, as we saw in the introduction. For example, suppose $\cP_{i}$, $i=1,2$, is a Pinsker algebra in $\cM_{i}$, $i=1,2$, and that $\cD$ is a common diffuse subalgebra of $\cP_{1}$ and $\cP_{2}$. If $\cD \neq \cP_{2}$, then $\cP_{1}$ is \emph{not} Pinsker in $\cM=\cM_{1}*_{D}\cM_{2}$. To see this, note that  $\cP_{1}\cap \cP_{2}$ is diffuse by virtue of containing $\cD$. Thus $h(\cP_{1}*_{\cD}\cP_{2}:\cM)\leq 0$ by Property \ref{I:subadditivity of 1 bdd ent}, while on the other hand we know $\cP_{1}*_{\cD}\cP_{2}\neq \cP_{1}$ since $\cD\ne \cP_{2}$ . The place where the proof of Theorem \ref{thm:main2} would break down is in achieving exponential concentration of measure. Indeed, even using matrix approximations for an amenable algebra, the exponential concentration on the relative commutants of $\pi^{(k)}(\cD)$ in the matrix models decays in the limit. Thus we lose the exponential concentration of measure when $\cD$ is amenable and non-atomic.
\end{remark}

\section{Further Remarks} \label{sec:remarks}

\subsection{The External Averaging Property and Ultraproducts} \label{subsec:ultraproducts}

The external averaging property (hypothesis (4) from Theorem \ref{thm:main}) can alternatively be stated in terms of a commuting square property \cite{Popa1983, Popa1990} for embeddings into ultraproducts.

Let $(\cM,\tau)$ be a tracial $\mathrm{W}^*$-algebra, let $\mathbf{x}$ be an $I$-tuple of self-adjoint generators, and let $\mathbf{R} \in (0,+\infty)^I$ with $\norm{x_i} \leq R_i$.  Suppose that $\mathbf{X}^{(k)}$ is an $I$-tuple of random $n(k) \times n(k)$ self-adjoint matrices satisfying $\norm{X_i^{(k)}} \leq R_i$ and $\lambda_{\mathbf{X}^{(k)}} \to \lambda_{\mathbf{x}}$ in probability.

Let us realize the $\mathbf{X}^{(k)}$'s on the same probability space $\Omega$.  Then $X_i^{(k)}$ can be viewed as an element of $L^\infty(\Omega,M_{n(k)}(\C)) = L^\infty(\Omega) \otimes M_{n(k)}(\C)$.  Let $\omega$ be a free ultrafilter on $\N$, and consider the tracial $\mathrm{W}^*$-ultraproduct
\[
\prod_{k \to \omega} L^\infty(\Omega, M_{n(k)}(\C)).
\]
Let $y_i$ be the element of the ultraproduct represented by $\{X_i^{(k)}\}_{k \in \N}$.  Then $\mathbf{y} = (y_i)_{i \in I}$ has the same non-commutative law as $\mathbf{x}$, and hence there is an embedding $\rho_\omega: \cM \to \prod_{k \to \omega} L^\infty(\Omega,M_{n(k)}(\C))$ given by $\rho(x_i) = y_i$.

We also remark that there is a canonical inclusion $M_{n(k)}(\C) \to L^\infty(\Omega, M_{n(k)}(\C))$, mapping $A$ to the deterministic function $A$ on $\Omega$, and hence we have an inclusion
\[
\iota_\omega: \prod_{k \to \omega} M_{n(k)}(\C) \to \prod_{k \to \omega} L^\infty(\Omega, M_{n(k)}(\C)).
\]

\begin{prop}
Let $(\cM, \tau)$, $\mathbf{x}$, $\mathbf{R}$, $\mathbf{X}^{(k)}$, $\mathbf{y}$ be as above, and let $\cP \subseteq \cM$ be a $\mathrm{W}^*$-subalgebra.  Then the following are equivalent:
\begin{enumerate}
    \item $\mathbf{X}^{(k)}$ satisfies the external averaging property for $\cP \subseteq \cM$.
    \item For every free ultrafilter $\omega$ on $\N$, the embedding $\rho_\omega: \cM \to \prod_{k \to \omega} L^\infty(\Omega, M_{n(k)}(\C))$ described above satisfies
    \[
    \rho_\omega(\cP) = \rho_\omega(\cM) \cap \iota_\omega \left( \prod_{k \to \omega} M_{n(k)}(\C) \right),
    \]
    and we have commutativity of the diagram
    \[
    \begin{tikzcd}
    \cM \arrow{r}{\rho_\omega} \arrow{d}{E_{\cP}} & \prod_{k \to \omega} L^\infty(\Omega,M_{n(k)}(\C)) \arrow{d}{E_\omega} \\
    \cP \arrow{r}{\rho_\omega} & \prod_{k \to \omega} M_{n(k)}(\C),
    \end{tikzcd}
    \]
    where $E$ denotes the trace-preserving conditional expectation.
\end{enumerate}
\end{prop}

\begin{proof}
(1) $\implies$ (2).  Fix $z \in \cM$ and write $z = f(\mathbf{x})$ for some $f \in \mathcal{F}_{\mathbf{R},\infty}$.  By Lemma \ref{lem:EAPequivalences}, we have $f(z) \in \cP$ if and only if
\[
\lim_{k \to \infty} \norm{f(\mathbf{X}^{(k)}) -\E[f(\mathbf{X}^{(k)})]}_{L^2(\Omega,L^2(M_{n(k)}(\C)),\tau_{n(k)})} = 0.
\]
But since ultraproducts commute with trace-preserving conditional expectation, this occurs if and only if
\[
\norm{f(\mathbf{y}) - E_\omega[f(\mathbf{y})]}_2 = 0.
\]
This proves that $\rho_\omega(\cP) = \rho_\omega(\cM) \cap \iota_\omega ( \prod_{k \to \omega} M_{n(k)}(\C) )$.

To check the commuting square condition, it suffices to consider $z \in \cP$ (which we already did) and $z \perp \cP$.  By Lemma \ref{lem:EAPequivalences}, $z \perp \cP$ if and only if
\[
\lim_{k \to \omega} \norm{\E[f(\mathbf{X}^{(k)})]}_2 = 0.
\]
But this is equivalent to $\norm{E_\omega[f(\mathbf{y})]}_2 = 0$.

(2) $\implies$ (1).  Fix $f \in \mathcal{F}_{\mathbf{R},\infty}$.  Then for every free ultrafilter $\omega$, we have
\[
\lim_{k \to \omega} \norm{\E[f(\mathbf{X}^{(k)})]}_2 = \norm{E_\omega[f(\mathbf{y})]}_2 = \norm{E_{\cP}[f(\mathbf{x})]}_2,
\]
where the last equality follows from the commutativity of the diagram and the fact that $\rho_\omega(\mathbf{x}) = \mathbf{y}$.  Since the ultrafilter $\omega$ was arbitrary, we have $\lim_{k \to \infty} \norm{\E[f(\mathbf{X}^{(k)})]}_2 = \norm{E_{\cP}[f(\mathbf{x})]}_2$.
\end{proof}

\subsection{Change of Generators and Passing to Subalgebras} \label{subsec:changeofgenerators}

Now we show that the hypotheses of Theorem \ref{thm:main} are independent of the choice of generators for $\cM$.  More precisely, if there exist measures satisfying these hypotheses for one choice of generators $\mathbf{x}$ for $\cM$, and if $\mathbf{y}$ is another choice of generators for $\cM$, then there also exist measures satisfying these hypotheses for $\mathbf{y}$.  This is the special case of Proposition \ref{prop:independentofgenerators} below where we take $\cQ = \cM$.

Another special case worth pointing out is if $\cP \leq \cQ \leq \cM$; in this case, the proposition says that if such measures $\mu^{(k)}$ for the inclusion $\cP \leq \cM$, then they also exist for the inclusion $\cP \leq \cQ$ when $\cQ$ is an intermediate subalgebra between $\cP$ and $\cM$.

\begin{prop} \label{prop:independentofgenerators}
Let $(\cM,\tau)$ be a tracial $\mathrm{W}^*$-algebra, let $\cP$ be a $\mathrm{W}^*$-subalgebra.  Suppose that $\mathbf{x} \in \cM_{sa}^I$ generates $\cM$, let $\mathbf{y} \in \cM_{sa}^{I'}$, and let $\cP = \mathrm{W}^*(\mathbf{y})$.  Suppose that $\cQ$ satisfies the commuting square condition that $E_{\cP} \circ E_{\cQ} = E_{\cP \cap \cQ}$.

Let $\mathbf{R} \in (0,+\infty)^I$ and $\mathbf{R}' \in (0,\infty)^{I'}$ be bounds for the operator norms of $x_i$ and $y_i$ respectively, and $\mathbf{f} \in \mathcal{F}_{\mathbf{R},\mathbf{R}'}$ be a function such that $\mathbf{y} = \mathbf{f}(\mathbf{x})$.

Suppose that $\{\mu^{(k)}\}$ is a sequence of probability measures on $M_{n(k)}(\C)_{sa}^I$ satisfying the hypotheses (1) - (4) with respect to $\cP \subseteq \cM$ and $\mathbf{x}$ and $\mathbf{R}$.  Then $\{\mathbf{f}_* \mu^{(k)}\}$ satisfies hypotheses (1) - (4) with respect to $\cP \cap \cQ \subseteq \cQ$ and $\mathbf{y}$ and $\mathbf{R}'$.
\end{prop}

\begin{proof}
(1) We have $\norm{f_i(\mathbf{a})}_\infty \leq R_i'$ for all $i \in I'$ and any self-adjoint tuple $\mathbf{a}$ with $\norm{a_i} \leq R_i$ for $i \in I$.  In particular, this holds when we evaluate $f_i$ on the random matrix tuple $\mathbf{X}^{(k)}$ associated to $\mu^{(k)}$.

(2) Since $\lambda_{\mathbf{X}^{(k)}} \to \lambda_{\mathbf{x}}$ in probability, we have $\lambda_{\mathbf{f}(\mathbf{X}^{(k)})} \to \lambda_{\mathbf{f}(\mathbf{x})}$ in probability by Proposition \ref{prop:pushforwardcontinuity}.

(3) Exponential concentration is inherited by $\mathbf{f}_* \mu^{(k)}$ by Corollary \ref{cor:concentrationpushforward}.

(4) Let $q \in \C\ip{t_i: i \in I'}$.  Then $q \circ \mathbf{f} \in \mathcal{F}_{\mathbf{R},\infty}$ since $\mathcal{F}_{\mathbf{R},\infty}$ is a $\mathrm{C}^*$-algebra.  Therefore,
\begin{align*}
\lim_{k \to \infty} \norm*{ \int q\,d(\mathbf{f}_*\mu^{(k)}) }_2
&= \lim_{k \to \infty} \norm*{ \int q \circ \mathbf{f} \,d\mu^{(k)} }_2  \\
&= \norm{ E_{\cP}[q(\mathbf{f}(\mathbf{x}))] }_2 = \norm{ E_{\cP}[q(\mathbf{y})] }_2 \\
&= \norm{ E_{\cP} \circ E_{\cQ}[q(\mathbf{y})] }_2 = \norm{ E_{\cP \cap \cQ}[q(\mathbf{y})] }_2,
\end{align*}
where we applied the external averaging property for $\mu^{(k)}$ and the commuting square condition for $\cP \cap \cQ$.
\end{proof}

\subsection{Separability Issues}

Note that the proof of Theorem \ref{thm:main} did not require any separability assumptions on $\cP$ and $\cM$.  However, in the non-separable case, we should not expect there to exist a \emph{sequence} of random matrix models in general, since indeed $\cM$ might not embed into the ultraproduct $\prod_{n \to \omega} M_n(\C)$ when $\omega$ is a free ultrafilter on $\N$.

From an abstract point of view, the natural adaptation to the non-separable case would be to use a \emph{net} of random matrix models.  That is, we take $k$ to be an index from some directed set $K$, and then consider random matrix tuples $\mathbf{X}^{(k)}$ indexed by $k \in K$.  We replace the limits as $k \to \infty$ in $\N$ by limits as over $k \in K$, and then the proof of Theorem \ref{thm:main} will go through without further changes.  The same is true for the statements in \ref{subsec:ultraproducts} and \ref{subsec:changeofgenerators}; we leave the details to the reader.

The proof of Theorem \ref{thm:main2} could also be modified to construct a net of random matrix models for $\cM_1 *_{\cD} \cM_2$ directly without going through the reduction to the separable case.  This only requires modification of the construction of microstates for $\cM_1$ and $\cM_2$ in \S \ref{subsec:freeproductmicrostates} to obtain a net of $*$-homomorphisms and microstates indexed by $k$ in some directed set $K$ rather than $k \in \N$.  The results about concentration and asymptotic freeness still apply if we arrange that $n(k) \to \infty$ over $k \in K$.

However, for the typical situations that arise in random matrix theory, one would only need to consider \emph{separable} algebras and \emph{sequences} of random matrix models.  Moreover, as we saw in the proof of Theorem \ref{thm:main2}, the von Neumann algebraic applications can often be reduced to the separable case because $1$-bounded entropy is something defined through finitary approximations to begin with.

We remark that the existence of a net of random matrix models satisfying the hypotheses of Theorem \ref{thm:main} for a given generating set $\mathbf{x} \in \cM_{sa}^I$ is equivalent to following statement:  For every finite $F \subseteq I$, there exists a sequence $\mathbf{X}^{(k)}$ of random $F$-tuples of $n(k) \times n(k)$ matrices satisfying (1) operator norm bounds, (2) convergence of $\lambda_{\mathbf{X}^{(k)}}$ to $\lambda_{\mathbf{x}|_{F}}$ in probability, (3) exponential concentration, and (4) for every $f \in \C\ip{t_i: i \in F}$,
\[
\lim_{k \to \infty} \norm{\E[f(\mathbf{X}^{(k)}]}_2 = \norm{E_{\cP}[f(\mathbf{x}|_I)]}_2.
\]

\subsection{A Free Decomposability Conjecture}

So far, we have not found any examples \emph{other} than amalgamated free products where Theorem \ref{thm:main} applies.  In fact, it is conceivable, based on an analogy with ergodic theory, that the only way the hypotheses of Theorem \ref{thm:main} can be satisfied is if there exists such an amalgamated free product decomposition.  Therefore, we propose the following conjecture.

\begin{conj*}\label{C:free decomposability}
Let $\cP\leq \cM$ be an inclusion of tracial von Neumann algebras. Suppose there exist generators $(x_{i})_{i\in I}$, a sequence $n(k)\to\infty$ of natural numbers, and self-adjoint $n(k)\times n(k)$ random matrices $(X^{(k)}_{i})_{i\in I}$ satisfying (1)-(4) of Theorem \ref{thm:main}. Then there is von Neumann subalgebra $\cN \leq \cM$, and an atomic $\cD \leq \cM$ so that $\cN,\cP$ are freely independent with amalgamation over $\cD$, and so that $\cM = \cN \vee \cP$.
\end{conj*}

This conjecture is motivated by the following analogy with ergodic theory. Suppose $\Z \actson^{T} (X,\mu)$ is a probability measure-preserving action, and suppose we are given a measurable map $\alpha\colon X\to A$ where $A$ is a finite set. Assume further that $\alpha$ is a \emph{generator}, i.e. $\bigvee_{n}T^{n}(\{f\circ \alpha:f\in \C^{A}\})=L^{\infty}(X,\mu)$. In this case, the partition $\{1_{\alpha^{-1}(\{a\})}:a\in A\}$ is analogous to the $(x_{i})_{i\in I}$ in the free probabilistic context. Moreover, there is a natural interpretation of \emph{microstates} for such actions (see \cite{Bow}).

We can then form a natural analogue of having a sequence of measures which is asymptotically supported on the microstate spaces, satisfies exponential concentration of measure, and has the external averaging property. In the setting of probability measure-preserving $\Z$-actions, the existence of a sequence of measures which satisfies the analogue of these three properties is equivalent to the \emph{almost blowing up property}  (first defined in \cite[Definition 4]{MartonShieldsABUP} under the name ``blowing up property"). This property is one of the conditions from \emph{Ornstein theory} (see \cite{OrnClassify1, OrnClassify2, OrnWeiss, OrnYaleLectures, ShieldsBern}) which guarantees that the $\Z$-action is isomorphic to a Bernoulli shift. (This is shown in \cite[3.2.2]{MartonShieldsABUP}.) This relation to concentration of measure and the action being Bernoulli was recently used to great effect in \cite{AustinWeakPinsker} to solve the weak Pinsker conjecture.

So in the context of free probability, it should be the case that these three conditions (asymptotic concentration on microstates, exponential concentration of measure, and the external averaging property)  imply that $\cM$ is ``freely Bernoulli over $\cP$.''  Note that free products may be viewed as a free analogue of Bernoulli shifts. For example if we consider the infinite free product $P^{*\Z}$ and if $\Z\curvearrowright P^{*\Z}$ by free shifts, then $P^{*\Z}\rtimes \Z\cong P*L(Z)$.  So one interpretation of being ``freely Bernoulli over $\cP$" is being free complemented with amalgamation over an atomic subalgebra. So Conjecture \ref{C:free decomposability} may be viewed as the natural free probability analogue of the fact from Ornstein theory that the almost blowing up property for a $\Z$-action implies the action is Bernoulli.

\newcommand{\etalchar}[1]{$^{#1}$}
\providecommand{\bysame}{\leavevmode\hbox to3em{\hrulefill}\thinspace}
\providecommand{\MR}{\relax\ifhmode\unskip\space\fi MR }
\providecommand{\MRhref}[2]{%
  \href{http://www.ams.org/mathscinet-getitem?mr=#1}{#2}
}
\providecommand{\href}[2]{#2}

%
%
%\bibliographystyle{amsalpha}
%\bibliography{Pinsker}

\end{document}